\documentclass[reqno,11pt,dvipsnames]{article}
\usepackage[full]{textcomp}
\usepackage{newtxtext} 
\usepackage{graphicx,amsmath,amssymb}  
\usepackage{enumitem}
\setlist{topsep=1pt,parsep=1pt,itemsep=1pt} 

\usepackage{tikz}
\usetikzlibrary{matrix,arrows}
\usetikzlibrary{decorations.pathmorphing}
\usetikzlibrary{decorations.markings}
\usepgflibrary{fpu}
\usepackage{xcolor}
\usepackage{mathrsfs} 
\usepackage{bbm}

\usepackage{cleveref}
\usepackage{amsthm}  
\usepackage{autonum} 
\usepackage[margin=20pt,font=small,labelfont=sc]{caption}

\usepackage{thmtools}

\input{stochbif.sty}

\input{regstruc_frac.sty}

\input{diagrams_frac.sty}

\newcommand{\unit}{\vone}

\newcommand{\bPi}{\mathbf{\Pi}}

\newcommand{\Red}[1]{\textcolor{red}{#1}}

\newcommand{\ninner}{n_{\rm{inner}}}
\newcommand{\nsym}{n_{\rm{sym}}}
\newcommand{\nleaves}{n_{\rm{leaves}}}
\newcommand{\kmax}{k_{\rm{max}}}
\newcommand{\kbarmax}{\bar k_{\rm{max}}}

\newcommand{\sA}{{\mathscr A}}
\newcommand{\sC}{{\mathscr C}}
\newcommand{\sE}{{\mathscr E}}
\newcommand{\sF}{{\mathscr F}}
\newcommand{\sG}{{\mathscr G}}
\newcommand{\sI}{{\mathscr I}}
\newcommand{\sO}{{\mathscr O}}
\newcommand{\sP}{{\mathscr P}}
\newcommand{\sR}{{\mathscr R}}
\newcommand{\sS}{{\mathscr S}}
\newcommand{\sV}{{\mathscr V}}
\newcommand{\sW}{{\mathscr W}}

\newcommand{\bT}{\mathbf{T}}
\newcommand{\bn}{\mathbf{n}}

\newcommand{\sFs}{{\mathscr F}_{\mathrm{s}}}
\newcommand{\sFu}{{\mathscr F}_{\mathrm{u}}}
\newcommand{\Fsafe}[1]{\sF_{#1}^{(\mathrm{s})}}

\DeclareMathOperator{\scale}{scale}
\DeclareMathOperator{\intT}{int}
\DeclareMathOperator{\extT}{ext}
\DeclareMathOperator{\reg}{reg}

\def\out{\mathrm{out}}

\def\fraks{\mathfrak{s}}

\newcommand{\abss}[1]{\abs{#1}_\fraks}
\def\Deltam{\Delta^{\!-}}

\newcommand{\Labe}{\mathfrak{e}}

\newcommand{\Labn}{\mathfrak{n}}

\newcommand{\Adm}{\mathfrak{A}}

\renewcommand{\Deltam}{\Delta^{\!-}}

\newcommand{\Deltatm}{\tilde\Delta^{\!-}}

\newcommand{\rhocrit}{\rho_{\mathrm{c}}}
\newcommand{\epscrit}{\eps_{\mathrm{c}}}
\newcommand{\epsbarcrit}{\bar\eps_{\mathrm{c}}}

\newcommand{\restr}{\mathord{\upharpoonright}}
\newcommand{\s}{\mathfrak{s}}
\newcommand{\funit}{\mathbbm{1}}

\definecolor{darkgreen}{rgb}{0.25, 0.66, 0.18}

\newcommand\full{full}          
\newcommand\afull{almost full}  
\newcommand\Full{Full}          
\newcommand\Afull{Almost full}  

\begin{document}


\title{BPHZ renormalisation and vanishing subcriticality 
\\ asymptotics
of the fractional $\Phi^3_d$ model}
\author{Nils Berglund and Yvain Bruned}
\date{February 21, 2024}

\maketitle

\begin{abstract}
We consider stochastic PDEs on the $d$-dimensional torus with fractional 
Laplacian of parameter $\rho\in(0,2]$, quadratic nonlinearity and driven by 
space-time white noise. These equations are known to be locally subcritical, 
and thus amenable to the theory of regularity structures, if and only if $\rho 
> d/3$. Using a series of recent results by the second named author, A.\ 
Chandra, I.\ Chevyrev, M.\ Hairer and L.\ Zambotti, we obtain precise 
asymptotics on the renormalisation counterterms as the mollification parameter 
$\eps$ becomes small and $\rho$ approaches its critical value. In particular, 
we show that the counterterms behave like a negative power of $\eps$ if $\eps$ 
is superexponentially small in $(\rho-d/3)$, and are otherwise of order 
$\log(\eps^{-1})$. This work also serves as an illustration of the general 
theory of BPHZ renormalisation in a relatively simple situation. 
\end{abstract}

\leftline{\small 2010 {\it Mathematical Subject Classification.\/} 
60H15, 
35R11 (primary), 
81T17, 
82C28 (secondary). 
}


\section{Introduction}
\label{sec:intro} 

The last years have witnessed tremendous progress in the theory of singular 
stochastic partial differential equations (SPDEs). The theory of regularity 
structures, introduced by Martin Hairer in~\cite{Hairer1}, provides a 
functional analysis framework in which many so-called locally subcritical 
singular SPDEs can be shown to admit (local in time) solutions. 
The theory has been successfully applied to a number of different SPDEs,  
including the KPZ equation~\cite{Hairer2,Hairer_Shen_17,Hoshino} and its 
generalisations to polynomial nonlinearities~\cite{Hairer_Quastel_18} and 
non-polynomial nonlinearities \cite{HX18},
the dynamic $\Phi^4_3$ model~\cite{Hairer1,HX16}, 
the continuum parabolic Anderson model~\cite{Hairer_Labbe_15}, 
the Navier--Stokes equation~\cite{ZhuZhu_2015}, 
the motion of a random string on a curved surface~\cite{Hairer4,BGHZ19},
the FitzHugh--Nagumo SPDE~\cite{Berglund_Kuehn_15}, 
the dynamical Sine--Gordon model~\cite{Hairer_Shen_16,CHS18},
the heat equation driven by space-time fractional noise~\cite{Deya_2016},
reaction-diffu\-sion equations with a fractional Laplacian~\cite{BK17}, and 
the multiplicative stochastic heat equation~\cite{wong, Hairer_Labbe_18}. 

A limitation of the theory introduced in~\cite{Hairer1} is that, while it 
provides function spaces allowing to prove fixed-point theorems in a very 
general setting, the applications to SPDEs also require a renormalisation 
procedure, which had to be carried out in an \textit{ad hoc} manner in each 
case. This situation has been remedied in a series of papers by the second 
named author, Ajay Chandra, Ilya Chevyrev, Martin Hairer and Lorenzo 
Zambotti~\cite{BrunedHairerZambotti,ChandraHairer16,BCCH17}. These works 
provide a kind of black box, allowing to automatically renormalise any locally 
subcritical SPDE. Owing to its great generality, however, this theory is rather 
abstract, making it somewhat difficult of access. 

A first goal of the present work is to illustrate the general theory in one 
of the simplest possible, yet interesting examples. This example is the 
$\Phi^3$ model with fractional Laplacian $\Delta^{\rho/2}$ on the 
$d$-dimensional torus, driven by space-time white noise $\xi$, whose equation 
before renormalisation reads 
\begin{equation} 
\label{eq:SPDE} 
 \partial_t u - \Delta^{\rho/2}u = u^2 + \xi\;.
\end{equation} 
A family of SPDEs with fractional Laplacian, including the above example, was 
considered in~\cite{BK17}. Results in that work imply in particular that the 
above equation is locally subcritical if and only if $\rho > \rhocrit = 
\frac{d}{3}$. 
As the parameter $\rho$ of the fractional Laplacian decreases towards its 
critical value $\rhocrit$, the size of the model space describing a regularity 
structure for~\eqref{eq:SPDE} diverges exponentially fast in 
$1/(\rho-\rhocrit)$. As we shall see, this has an effect on the 
renormalisation procedure for the equation, since the counterterms entering 
this procedure involve sums over elements of the model space having negative 
degree (see~\cite[Thm.~2.21]{BrunedHairerZambotti} and~\eqref{eq:counterterm} 
below). This should be a general phenomenon for models approaching the 
subcriticality threshold.

The fact that the nonlinearity in~\eqref{eq:SPDE} is quadratic entails a number 
of significant simplifications when applying the general theory 
of~\cite{BrunedHairerZambotti,ChandraHairer16,BCCH17}, owing to the fact 
that the model space can be described precisely in terms of binary trees. This 
considerably simplifies a number of combinatorial arguments.
Throughout the analysis, we provide numerous 
examples, which should help to illustrate the general abstract theory. 

A second goal of this work is to analyse in detail the limit 
$\rho\searrow\rhocrit$, i.e.\ when approaching the threshold where local 
subcriticality is lost. The hope is that this will improve the understanding of 
the role of subcriticality in renormalisation of singular SPDEs and the theory 
of regularity structures. The renormalisation procedure requires to modify the 
SPDE~\eqref{eq:SPDE} by mollifying space-time noise $\xi$ on scale $\eps$, and 
adding $\eps$-dependent counterterms to the equation. Our main result, 
Theorem~\ref{thm:main}, analyses the asymptotic behaviour of these counterterms 
as a function of $\eps$ and $\rho-\rhocrit$. We obtain that if $\eps$ is 
superexponentially small in terms of $\rho-\rhocrit$, the counterterms scale 
like a negative power of $\eps$, while for larger $\eps$, they have order 
$\log(\eps^{-1})$. 

Note that fractional models near criticality have been studied before, in 
particular in the context of constructive Quantum Field Theory (QFT). For 
instance, the large-volume (infrared) behaviour of the static 
$\Phi^4_{4-\delta}$ model has been studied 
in~\cite{Brydges_Dimock_Hurd_non-Gaussian_98}, by modifying the Laplacian of the 
$\Phi^4_4$ model in order to make it subcritical. The picture that emerges from 
a renormalisation group (RG) analysis is that while for $\delta=0$, the RG flow 
converges to a Gaussian fixed point, for $\delta>0$, this fixed point becomes 
unstable, and a non-Gaussian fixed point appears. Recently, 
in~\cite{Aizenman_Duminilcopin_20} Aizenman and Duminil--Copin proved that 
by taking both the large-volume and zero-spacing (ultraviolet) limit of a 
lattice model converging to the $\Phi^4_4$ model, one converges to a model with 
Gaussian fluctuations. It is thus of interest to try to connect what is known 
on static models at and near criticality, with what happens to the 
renormalisation procedure in near-critical dynamical models. 

A final motivation for this article is that the equation~\eqref{eq:SPDE} is 
interesting in its own right. For instance, it approximates the Fisher--KPP 
equation for population dynamics~\cite{Fisher37,KPP37} for intermediate 
population values. Note that the real Fisher--KPP equation contains a 
factor of the form $\sqrt{u(1-u)}$ in front of the noise $\xi$, and it currently 
seems unlikely that such a nonlinearity could be handled with the help of 
regularity structures.
However, an understanding of the equation with additive noise may provide some 
useful first insights on its dynamics. See also~\cite{BK17} for further 
motivation on considering SPDEs with fractional Laplacians as a way to 
regularise coupled SPDE--ODE systems. 

The remainder of this article is organised as follows. Section~\ref{sec:model} 
gives a detailed description of the model, and states the main result, 
Theorem~\ref{thm:main}, on the asymptotic behaviour of counterterms. 
Section~\ref{sec:modelspace} summarises the construction of the model space, and 
the main results from~\cite{BrunedHairerZambotti,ChandraHairer16,BCCH17} needed 
to compute the renormalised equation. The most difficult step in applying the 
general theory is to compute the expectation of the renormalised canonical model 
elements, and is presented in the next three sections. 
Section~\ref{sec:modelexp} describes how these expectations can be represented 
in terms of Feynman diagrams (see Definition~\ref{def:feynman_diagram}). 
Section~\ref{sec:forests} introduces the notions of forests (see 
Definition~\ref{def:forest}) and Hepp sectors (see 
Definition~\ref{def:Hepp_sector}), needed to apply ideas from BPHZ 
renormalisation theory, as explained in~\cite{Hairer_BPHZ} in the Euclidean 
case. 
Most of our formalism is taken from \cite{Hairer_BPHZ}, which transposes the 
algebraic construction in \cite{BrunedHairerZambotti} and the proof of the 
renormalised model convergence  in \cite{ChandraHairer16} to Feynman diagrams. 
It has a strong connection with the algebraic structures observed by Connes and 
Kreimer in \cite{CK1, CK2}. The main difference comes from the presence of 
Taylor expansions, which are encoded at the level of the diagrams by changing 
decorations.  In our model, we can to a large extent circumvent these Taylor 
expansions and thus be closer to the extraction-contraction renormalisation 
procedure on Feynman diagrams. 
The actual bounds on the expectations are then obtained in 
Section~\ref{sec:bounds}, and the asymptotic analysis completing the proof of 
the main result is given in Section~\ref{sec:asymp}.

\subsubsection*{Acknowledgments}

We would like to thank Christian Kuehn for many useful discussions. Part of 
this work was carried out while the authors attended the programme 
\lq\lq Scaling limits, rough paths, quantum field theory\rq\rq\ (SRQ)  
held at the Isaac Newton Institute (INI) in Cambridge. We would like to thank 
the organisers of this trimester for putting together a stimulating programme, 
and the members of INI for providing a friendly working atmosphere. NB thanks 
the School of Mathematics at the University of Edinburgh, and YB thanks the 
Institut Denis Poisson at the University of Orl\'eans for hospitality during 
mutual visits. YB gratefully acknowledges funding
support from the European Research Council (ERC) through the ERC Starting Grant Low
Regularity Dynamics via Decorated Trees (LoRDeT), grant agreement No. 101075208. Finally, we thank the anonymous referees for their remarks, 
which led to an improvement in the presentation.


\section{Model and results}
\label{sec:model} 

We are interested in the SPDE 
\begin{equation} \label{eq:laplacienrho}
 \partial_t u - \Delta^{\rho/2}u = u^2 + \xi
\end{equation} 
for the unknown $u=u(t,x)$ with $(t,x)\in\R_+\times\T^d$, 
where $\Delta^{\rho/2} = -(-\Delta)^{\rho/2}$ denotes the fractional Laplacian 
with $0<\rho\leqs2$, and $\xi$ denotes space-time white noise. 
As such, this equation is not well-posed in general, and a renormalisation 
procedure is required. The general form of the renormalised equation is 
expected to be 
\begin{equation}
\label{eq:renormalised} 
 \partial_t u - \Delta^{\rho/2}u = u^2 + C(\eps,\rho,u) + \xi^\eps\;,
\end{equation}  
where $\xi^\eps = \varrho^\eps * \xi$ denotes space-time white noise mollified 
on scale $\eps$, and $C(\eps,\rho,u)$ is a counterterm which diverges as 
$\eps\searrow0$. Here $\varrho^\eps(t,x) = 
\eps^{-(\rho+d)}\varrho(\eps^{-\rho}t,\eps^{-1}x)$ for a smooth, compactly 
supported mollifier $\varrho$ integrating to $1$, and $*$ denotes space-time convolution.

\begin{figure}
\begin{center}
\begin{tikzpicture}[>=stealth',main 
node/.style={circle,radius=0.01cm,draw},x=1.5cm ,y=1.5cm]

\draw[thin,Green!20,fill=Green!20] (0,0) -- (2,2) -- (0,2) -- cycle;

\draw[thin,blue!20,fill=blue!20] (0,0) -- (6,2) -- (2,2) -- cycle;

\draw[thin,red!20,fill=red!20] (0,0) -- (7,0) -- (7,2) -- (6,2) -- 
cycle;

\draw[semithick,Green] (0,2) -- (2,2);
\draw[semithick,blue] (0,0) -- (2,2) -- (6,2);
\draw[semithick,red] (0,0) -- (6,2) -- (7,2);

\draw[->,thick] (0,0) -> (7.5,0);
\draw[->,thick] (0,0) -> (0,3);

\foreach \x in {2,...,5}
{
   \draw[semithick,blue!50] (\x,{\x/3}) -- (\x,2);
   \draw[blue,fill=blue] (\x,2) circle (0.03);
}

\foreach \x in {1,...,6}
{
   \draw[semithick,red!50] (\x,0) -- (\x,{\x/3});
   \draw[red,fill=red] (\x,{\x/3}) circle (0.03);
}

\draw[semithick,blue!50] (1,{1/3}) -- (1,1);
\draw[semithick,Green!50] (1,1) -- (1,2);

\draw[blue,fill=blue] (1,1) circle (0.03);
\draw[Green,fill=Green] (1,2) circle (0.03);

\node[Green] at (0.8,1.7) {\small Non-singular};
\node[blue] at (3.5,1.7) {\small Locally subcritical};
\node[red] at (5.5,0.8) {\small Supercritical};

\node[red,rotate=20] at (3.5,1.0) {\small $\Red{\rho=\rhocrit}$};

\foreach \x in {1,...,6}
{
  \draw[thick] (\x,-0.1) -- (\x,0.1);
  \node[] at (\x,-0.4) {$\x$};
}

\foreach \y in {1,2}
{
  \draw[thick] (-0.1,\y) -- (0.1,\y);
  \node[] at (-0.3,\y) {$\y$};
}
\node[] at (7.1,-0.25) {$d$};
\node[] at (-0.2,2.6) {$\rho$};
\end{tikzpicture}
\vspace{-5mm}
\end{center}
\caption[]{Parameter space $(d,\rho)$. Results in this article apply to the 
locally subcritical regime $\rhocrit = \frac{d}{3} < \rho < d$, with $\rho 
\leqs 2$. }
\label{fig_rho_d}
\end{figure}

The theory of regularity structures introduced in~\cite{Hairer1} applies, 
provided the equation~\eqref{eq:laplacienrho} is \emph{locally subcritical}, or 
superrenormalisable in physicist's terms. As shown 
in~\cite[Theorem~4.3]{BK17}, \eqref{eq:laplacienrho} is locally subcritical for 
\begin{equation}
 \rho > \rhocrit(d) = \frac{d}{3}\;.
\end{equation} 
Note that $\rhocrit < 2$ imposes $d\leqs 5$ (\figref{fig_rho_d}). One 
can guess this threshold by a scaling argument. Indeed, let us set $ \bar 
u(t,x) = \lambda^{\alpha} u(\lambda^{\beta} t, \lambda x)$ with $ \lambda > 0 $ 
and $\alpha, \beta \in \mathbb{R}$. Then, $\bar u$ solves the equation
\begin{equation}
\label{eq:scaled_SPDE} 
\partial_t \bar u - \lambda^{\beta- \rho} \Delta^{\rho / 2} \bar u  
 = \lambda^{\beta-\alpha} \bar u^2 + \lambda^{\alpha + \beta} 
\xi_{\lambda^{\beta}, \lambda}
  = \lambda^{\beta-\alpha} \bar u^2 + \lambda^{\alpha + \frac{\beta}{2} - 
\frac{d}{2}}\xi\;,
\end{equation} 
where the second equality is in law, and $\xi_{\lambda^{\beta}, \lambda}$ 
denotes scaled space-time white noise given by 
\begin{equation}
\label{eq:scaling_noise} 
 \pscal{\xi_{\lambda^{\beta},\lambda}}{\ph} 
 = \pscal{\xi}{\ph^{\lambda^{\beta}, \lambda}}\;, \qquad
 \ph^{\lambda^{\beta}, \lambda}(t,x) 
 = \frac{1}{\lambda^{\beta+d}}
 \ph\biggpar{\frac{t}{\lambda^\beta},\frac{x}{\lambda}}
\end{equation} 
for any compactly supported test function $\ph$.
Setting $\alpha = \frac{d-\beta}{2}$, the noise intensity is the same 
in~\eqref{eq:laplacienrho} and~\eqref{eq:scaled_SPDE}. 
Then one has $\beta - \alpha = \frac{3}{2} (\beta  - \rho_c)$, so that 
\begin{equation}
\label{eq:scaled_SPDE2} 
\partial_t \bar u - \lambda^{\beta- \rho} \Delta^{\rho / 2} \bar u  
= g \bar u^2 +  \xi\;, \qquad g = \lambda^{\frac{3}{2}(\beta - \rho_c)}\;.
\end{equation}
The natural choice is then $\beta = \rho$, which corresponds to the 
fractional scaling $\s =  (\rho,1,\dots,1)$ (cf.~\eqref{eq:def_scaling}). One 
thus obtains two regimes:
\begin{itemize}
\item 	If $\rho > \rho_c$ and we let $\lambda$ tend to $0$, then $g$ tends 
to $0$, i.e.~\eqref{eq:scaled_SPDE2} converges to a linear equation. This is 
exactly the definition of local subcriticality.
\item 	If $\rho = \rho_c$, we recover the non-linear equation we started with, 
i.e.\ the system is invariant under this particular scaling. This is 
reminiscent of what is called a fixed point of the Wilsonian renormalisation 
group in the language of physicists. 
\end{itemize}

The counterterm $C(\eps,\rho,u)$ in~\eqref{eq:renormalised}
is expected to diverge also in the limit $\rho\searrow\rhocrit$, and the main 
goal of this work is to determine how $C(\eps,\rho,u)$ behaves as a function of 
$\eps$ and $\rho-\rhocrit$ for small values of these parameters. 

In order to formulate the main result, we define, for $a\in\R$ and $k>0$, 
the threshold value  
\begin{equation}
 \epscrit(\rho,a,k) 
= \exp\biggset{-\frac{1}{\rho-\rhocrit} \biggbrak{\log k + a - 
\frac{\log(k+1)}{2k}}}\;. 
\end{equation} 
Then we set 
\begin{equation}
 \epscrit(\rho,a) = \epscrit(\rho,a,\kmax)\;, 
 \qquad 
 \epsbarcrit(\rho,a) = \epscrit(\rho,a,\kbarmax)\;, 
\end{equation} 
where 
\begin{equation}
 \kmax = \frac{d-\rho}{3(\rho-\rhocrit)}
 \qquad\text{and}\qquad 
 \kbarmax = \frac{d-2\rho}{3(\rho-\rhocrit)}\;.
\end{equation} 
The integer parts of $\kmax$ and $\kbarmax$ measure the size of the model space 
of the regularity structure (cf.~\cite[Thm.~4.18]{BK17}), where $\kmax$ is 
associated with the part of the counterterm $C(\eps,\rho,u)$ that does not 
depend on $u$, while $\kbarmax$ determines its part linear in $u$. 
Note that $\epscrit(\rho,a) > \epsbarcrit(\rho,a)$, and that as 
$\rho$ decreases to $\rhocrit$, $\epscrit(\rho,a)$ and $\epsbarcrit(\rho,a)$ 
both go to zero superexponentially fast, namely like 
\begin{equation}
\label{eq:superexponential} 
\exp\biggset{-\frac{1}{\rho-\rhocrit} 
\biggbrak{\log\biggpar{\frac{1}{\rho-\rhocrit}} + \Order{1}}}\;.
\end{equation} 
Finally, for $\eta<0$, we denote by $\cC^\eta(\T^d)$ the Besov--H\"older 
space defined as the set of distributions $\zeta$ on $\T^d$ such that 
$\lambda^{-\eta}\abs{\pscal{\zeta}{{\sS^\lambda_x\varphi}}}$ is bounded 
uniformly in $\lambda\in(0,1]$ for any $x\in\T^d$ and any compactly supported 
test function $\varphi$ of class $\cC^{\intpartplus{-\eta}}$, where 
$(\sS^\lambda_x\varphi)(y) = \lambda^{-d}\varphi(\lambda^{-1}(y-x))$. 

Our main result is then the following.

\begin{theorem}[Main result]
\label{thm:main} 
Assume $\rho < \frac{d}{2}$ and $\rho\in(\rhocrit,2]$. Then there
exist functions $C_i(\eps,\rho)$, $i\in\set{0,1}$, such that for any 
initial condition $u_0\in \cC^\eta(\T^d)$ with $\eta>-\frac\rho2$, the 
regularised renormalised SPDE~\eqref{eq:renormalised} with counterterm 
\begin{equation}
 C(\eps,\rho,u)
 = C_0(\eps,\rho) + C_1(\eps,\rho) u 
\end{equation} 
admits a sequence of local solutions $u^\eps$, converging in probability to a 
limiting process as $\eps\to0$. Furthermore, there exist constants $a$, $M$, 
$A_0$ and  $\bar A_0$, all independent of $\eps$ and $\rho$, such that, 
writting $\epscrit=\epscrit(\rho,a)$ and $\epsbarcrit=\epsbarcrit(\rho,a)$, 
the first counterterm satisfies 
\begin{align}
\label{eq:bound_C0} 
\bigabs{C_0(\eps,\rho)} &\leqs  
M\epscrit^{-(d-\rho)} \biggbrak{\log(\eps^{-1}) + \frac{1}{\rho-\rhocrit}
\biggpar{\dfrac{\epscrit}{\eps}}^{3(\rho-\rhocrit)}}
&& \text{if $\eps \geqs \epscrit$\;,} \\
\biggabs{\frac{C_0(\eps,\rho)}{A_0 \eps^{-(d-\rho)}} - 1}
&\leqs \frac{M}{\rho-\rhocrit}  
\biggpar{\dfrac{\eps}{\epscrit}}^{3(\rho-\rhocrit)}
&& \text{if $\eps < \epscrit$\;,} 
\end{align}
while the second counterterm satisfies 
\begin{align}
\label{eq:bound_C1} 
\bigabs{C_1(\eps,\rho)} &\leqs  
M\epsbarcrit^{\,-(d-2\rho)} \biggbrak{\log(\eps^{-1}) + 
\frac{1}{\rho-\rhocrit} \biggpar{\dfrac{\epsbarcrit}{\eps}}^{3(\rho-\rhocrit)}}
&& \text{if $\eps \geqs \epsbarcrit$\;,} \\
\biggabs{\frac{C_1(\eps,\rho)}{\bar A_0 \eps^{-(d-2\rho)}} - 1}
&\leqs \frac{M}{\rho-\rhocrit}   
\biggpar{\dfrac{\eps}{\epsbarcrit}}^{3(\rho-\rhocrit)}
&& \text{if $\eps < \epsbarcrit$\;.} 
\end{align}
\end{theorem}

\begin{remark}
Convergence is in probability in $\cC^\alpha_\fraks([0,T],\T^d)$, for any 
fixed $T>0$, and for the process stopped when its $\cC^\alpha_\fraks$-norm 
exceeds a fixed large cut-off $L$. Here $\alpha$ is any real number satisfying 
$\alpha < -\frac12(d-\rho)$, and $\cC^\alpha_\fraks$ is the scaled 
H\"older--Besov space associated with the scaling of the fractional Laplacian. 
This space is defined in an analogous way as $\cC^\alpha(\T^d)$, but with a 
fractional scaling given by $(\sS^\lambda_{(t,x),\fraks}\varphi)(s,y) = 
\lambda^{-(\rho+d)}\varphi(\lambda^{-\rho}(s-t),\lambda^{-1}(y-x))$. 
\end{remark}

\begin{remark}
The condition $\rho < \frac{d}{2}$ is due to the fact that we focus here 
on the asymptotic regime when the counterterms are given by sums of many 
divergent terms that are indexed by Feynman diagrams. More precisely, the 
results are meaningful when $\kmax$ and $\kbarmax$ are both large. What happens 
for $\rho \geqs \frac{d}{2}$ is in fact well known. When $\rho = d$, only the 
counterterm $C_0(\eps,\rho)$ is required, and it diverges like 
$\log(\eps^{-1})$. When $\frac{d}{2} < \rho < d$, $C_0(\eps,\rho)$ is still the 
only required counterterm, but it diverges like $\eps^{-(d-\rho)}$. When 
$\rho = \frac{d}{2}$, it becomes necessary to include the second counterterm 
$C_1(\eps,\rho)$, which then diverges like $\log(\eps^{-1})$. 
\end{remark}

\begin{remark}
The condition $\eta>-\frac\rho2$ is a consequence of the critical 
regularity of the initial condition in the fractional heat equation. Indeed, 
the scaling property $P_\rho(t,x) = t^{-d/\rho}P_\rho(1,t^{-1/\rho}x)$ of the 
fractional heat kernel implies that if $u_0\in\cC^\eta$ with $\eta<0$, then 
$(P_\rho u_0)(t,x)$ blows up like $t^{\eta/\rho}$. Therefore, $(P_\rho 
u_0)(t,x)^2$ blows up like $t^{2\eta/\rho}$, and its space-time convolution 
with $P_\rho$ is bounded if and only if $\eta>-\frac\rho2$.

More technically, the condition is related to the exponents of the space 
$\cD^{\gamma,\eta}$ of modelled distributions in which one solves a fixed-point 
equation, where $\gamma$ measures the H\"older regularity, while $\eta$ controls 
the singularity at time zero  (cf.~\cite[Def.~6.2]{Hairer1}). Indeed, 
\cite[Lemma~7.5]{Hairer1} shows that if $u_0\in \cC^\eta(\T^d)$, then its 
convolution with the Green function of the fractional Laplacian can be 
identified with an element of $\cD^{\gamma,\eta}$ for every $\gamma > 
\max\set{\eta,0}$. The fixed point $U$ cannot be more regular than the 
fractional stochastic convolution, which has regularity $\alpha$ for any $\alpha 
< -\frac12(d-\rho)$ (cf.~\cite[Sect.~4.1]{Berglund_Kuehn_15}). If 
$U\in\cD^{\gamma,\eta}$ has regularity $\alpha<0$, then $U^2$ has regularity 
$\bar\alpha=2\alpha$, while~\cite[Prop.~6.12]{Hairer1} shows that 
$U^2\in\cD^{\bar\gamma,\bar\eta}$ with $\bar\gamma = \gamma+\alpha$ and 
$\bar\eta = \eta + \min\set{\alpha,\eta}$. In order to 
apply~\cite[Thm.~7.8]{Hairer1} yielding existence of a unique fixed point, one 
needs to fulfill the condition $\min\set{\bar\eta,\bar\alpha} > -\rho$, which 
holds if $-\frac\rho2<\eta\leqs\alpha$. (The other required condition $\eta < 
\min\set{\bar\eta,\bar\alpha} + \rho$ is automatically satisfied if $\eta<0$.) 
\end{remark}

\begin{figure}
\begin{center}
\begin{tikzpicture}[>=stealth',main node/.style={circle,minimum
size=0.25cm,fill=blue!20,draw},x=5cm,y=3cm]

\draw[->,thick] (-0.1,0) -> (1.15,0);
\draw[->,thick] (0,-0.1) -> (0,1.15);

\draw[Green,thick,-,smooth,domain=0.01:1,samples=75,/pgf/fpu,/pgf/fpu/output
format=fixed] plot (\x, { exp(-ln(2/\x)/\x) });

\draw[blue,thick,-,smooth,domain=0.01:1,samples=75,/pgf/fpu,/pgf/fpu/output
format=fixed] plot (\x, { exp(-ln(1/\x)/\x) });

\node[] at (1,-0.1) {$\rho-\rhocrit$};
\node[] at (0.05,1) {$\eps$};
\node[blue] at (1.12,1) {$\epscrit(\rho)$};
\node[Green] at (1.12,0.5) {$\epsbarcrit(\rho)$};

\node[blue] at (0.3,0.7) {\small $C_0 \simeq \log(\eps^{-1})$};
\node[Green] at (0.3,0.4) {\small $C_1 \simeq \log(\eps^{-1})$};

\node[blue] at (1,0.65) {\small $C_0 \simeq \eps^{-(d-\rho)}$};
\node[Green] at (0.9,0.1) {\small $C_1 \simeq \eps^{-(d-2\rho)}$};
\end{tikzpicture}
\vspace{-5mm}
\end{center}
\caption[]{Behaviour of the counterterms as a function of $\rho-\rhocrit$ and 
$\eps$. The small-$\eps$ asymptotics of $C_0$ changes on the blue curve 
$\eps=\epscrit(\rho)$, while the asymptotics of $C_1$ changes on the green 
curve $\eps=\epsbarcrit(\rho)$.}
\label{fig_epscrit}
\end{figure}

In less technical terms, the first estimate in Theorem~\ref{thm:main} shows 
that, up to error terms which are small unless $\eps$ is close to $\epscrit$, 
\begin{equation}
 C_0(\eps,\rho) \simeq 
 \begin{cases} 
  \epscrit^{-(d-\rho)} \log(\eps^{-1})
  & \text{if $\eps \geqs \epscrit$\;,} \\
  A_0 \eps^{-(d-\rho)} 
  & \text{if $\eps < \epscrit$\;.} 
 \end{cases}
\end{equation} 
In the same spirit, the second counterterm satisfies 
\begin{equation}
 C_1(\eps,\rho) \simeq 
 \begin{cases} 
  \epsbarcrit^{\,-(d-2\rho)} \log(\eps^{-1})
  & \text{if $\eps \geqs \epsbarcrit$\;,} \\
  \bar A_0 \eps^{-(d-2\rho)} 
  & \text{if $\eps < \epsbarcrit$\;. } 
 \end{cases}
\end{equation} 
We thus obtain a saturation effect at values of the mollification parameter 
$\eps$ which are not superexponentially small: for $\eps$ larger than its 
critical value, the counterterms are of order $\log(\eps^{-1})$, with a 
prefactor becoming very large when $\rho$ approaches $\rhocrit$ 
(Figure~\ref{fig_epscrit}). For superexponentially small $\eps$, on the other 
hand, the counterterms diverge respectively like $\eps^{-(d-\rho)}$ and 
$\eps^{-(d-2\rho)}$. This is due to the fact that both counterterms can be 
written as the sum of a large number of contributions. Only one of these terms, 
which has the strongest singular behaviour as $\eps$ goes to $0$, dominates for 
superexponentially small $\eps$. The vast majority of the terms diverge only 
logarithmically, but their number is large enough for them to dominate when 
$\eps$ is larger than its critical value. 

The constants $A_0$ and $\bar A_0$ can be characterised more precisely. 
Assuming that the mollifier has the form $\varrho^\eps(t,x) = 
\varrho^\eps_0(t)\varrho^\eps_1(x)$ 
with $\varrho^\eps_0(t) = \eps^{-\rho}\varrho_0(\eps^{-\rho}t)$, 
$\varrho^\eps_1(x) = \eps^{-d}\varrho_1(\eps^{-1}x)$,
and $\varrho_1$ even, we have 
\begin{equation}
\label{eq:A0} 
 A_0 = 
 -\frac12 \lim_{\eps\to0}
 \eps^{d-\rho}(\varrho^\eps_1 *_x G_\rho)(0)
 = -\frac12 \lim_{\eps\to0}
 \int_{\R^d} \varrho_1(x) \eps^{d-\rho}G_\rho(\eps x) \6x\;,
\end{equation} 
where $G_\rho = (\Delta^{\rho/2})^{-1}$ is the Green function of the 
fractional Laplacian and $*_x$ denotes convolution in space. 
Scaling properties of $G_\rho$ (see for instance~\cite[Section~4]{Kwasnicki}) 
imply that $A_0$ is indeed finite. We also have 
\begin{equation}
\label{eq:Abar0}
\bar A_0 = -2 \lim_{\eps\to0} \eps^{d-2\rho}
\int_{\R^{d+1}} P_\rho(t,x) (G^\eps_\rho *_x \tilde P^\eps_\rho)(\abs{t},x) 
\6t\6x\;,
\end{equation} 
where $P_\rho$ is the fractional heat kernel, $G^\eps_\rho = \varrho^\eps_1 
*_x G_\rho$, and $\tilde P^\eps_\rho = P_\rho * \varrho^\eps * \varrho^\eps$.

The main insight provided by Theorem~\ref{thm:main} is as follows. The usual 
way of renormalising the singular SPDE~\eqref{eq:SPDE} is to fix $\rho > 
\rhocrit$, and then to take the limit $\eps\to0$. Our result then shows that a 
well-defined limit exists, provided one adds counterterms to the equation that 
behave logarithmically in $\eps$ as long as $\eps$ is not too small, but 
ultimately diverge like a negative power of $\eps$. On the other hand, one 
could also fix a small positive value of $\eps$ and look at the limit 
$\rho\searrow\rhocrit$. In physical terms, this would model a situation where 
space-time is discrete at very small scales, perhaps defined by Planck's scale. 
Since discrete models are usually harder to solve than continuous ones, the 
vanishing $\eps$ limit can be considered as an idealised mathematical object 
that really only approximates the real system. Note that for $\eps>0$, the SPDE 
is no longer singular, and local existence of solutions does not pose a problem 
at all. What our result says in this case, is that in order to have a chance to 
be close, for small $\eps$, to a well-defined continuous model, one should add 
counterterms of order $\log(\eps^{-1})$, but which diverge superexponentially 
fast in $\rho-\rhocrit$ in the sense of~\eqref{eq:superexponential}.

A more ambitious goal would be to look at possible limiting dynamics when $\eps$ 
and $\rho-\rhocrit$ simultaneously converge to zero, along some path in the 
$(\rho,\eps)$ plane, cf.~\figref{fig_epscrit}. There are two reasons why 
obtaining such a convergence result is currently out of reach. The first reason 
is that when changing $\rho$, one changes both the model space and the space of 
modelled distributions in which one tries to solve a fixed-point equation, so 
that the general theory of convergence in regularity structures does not 
immediately apply. The second, more serious reason is that as 
$\rho\searrow\rhocrit$, the number of symbols in the model space having negative 
degree diverges exponentially. However, many arguments in the theory of 
regularity structures only apply when the number of these symbols remains 
bounded. This fact is then crucial in showing that the sequence of 
$\eps$-dependent models converges in an appropriate topology to a well-defined 
limiting model. It is not clear at this point whether a similar convergence 
argument can be obtained when the number of symbols having negative degree is 
unbounded.  

Before moving to the proof of Theorem~\ref{thm:main}, we list some extensions 
and interesting open questions related to our 
results.
\begin{itemize}
\item 	Obtaining a matching lower bound on the counterterms in the regime of 
large $\eps$ seems out of reach at this stage, because of the existence of 
cancellations in the sums defining these counterterms. However, as explained in 
Section~\ref{ssec:asymp_remark}, one can show that there exist terms in the sum 
defining $C_0(\eps,\rho)$ which have the same asymptotic behaviour as the upper 
bound obtained above. Therefore, the counterterm can only be of smaller order in 
case unexpected cancellations occur in this sum. 
\item 	One can extend the results to the following generalisation 
of~\eqref{eq:SPDE}:
\begin{equation}
\partial_t u - \gamma \Delta^{\rho/2} u = g u^2 +  \sigma\xi\;.
\end{equation}
Its renormalised version reads 
\begin{equation}
\partial_t u - \gamma \Delta^{\rho/2} u = g u^2 
+ C^{\gamma,g,\sigma}(\eps,\rho,u)
+  \sigma\xi^\eps\;,
\end{equation}
where 
\begin{equation}
 C^{\gamma,g,\sigma}(\eps,\rho,u) 
 = C_0^{\gamma,g,\sigma}(\eps,\rho) 
 + C_1^{\gamma,g,\sigma}(\eps,\rho)u\;.
\end{equation}
One can then show (see Section~\ref{ssec:other_parameters}) that 
\begin{align}
\bigabs{C_0^{\gamma,g,\sigma}(\eps,\rho)} &\lesssim  
\biggpar{\frac{g^2\sigma^2}{\gamma^3}}^{\kmax} \frac{g\sigma^2}{\gamma} 
\widehat C_0(\eps,\rho)
&& \text{if $\eps \geqs \epscrit$\;,} \\
C_0^{\gamma,g,\sigma}(\eps,\rho) 
&=  \frac{g\sigma^2}{\gamma}
\biggbrak{C_0(\eps,\rho) + \biggOrder{\frac{g^2\sigma^2}{\gamma^3} 
\biggpar{\frac{\eps}{\epscrit}}^{3(\rho-\rhocrit)}}}
&& \text{if $\eps < \epscrit$\;,} 
\label{eq:C_full_parameters} 
\end{align}
where $\widehat C_0(\eps,\rho)$ denotes the upper bound on 
$\abs{C_0(\rho,\eps)}$ in~\eqref{eq:bound_C0}, and we write $a(\eps,\rho) 
\lesssim b(\eps,\rho)$ if there exists a constant $M\geqs 1$, independent of 
$\eps$ and $\rho$, such that $a(\eps,\rho) \leqs M b(\eps,\rho)$ holds for 
$\eps$ and $\rho-\rhocrit$ small enough.
In a similar way, we have 
\begin{align}
\bigabs{C_1^{\gamma,g,\sigma}(\eps,\rho)} &\lesssim  
\biggpar{\frac{g^2\sigma^2}{\gamma^3}}^{\kmax} \frac{g^2\sigma^2}{\gamma^2} 
\widehat C_1(\eps,\rho)
&& \text{if $\eps \geqs \epsbarcrit$\;,} \\
C_1^{\gamma,g,\sigma}(\eps,\rho) 
&=  \frac{g^2\sigma^2}{\gamma^2}
\biggbrak{C_1(\eps,\rho) + \biggOrder{\frac{g^2\sigma^2}{\gamma^3} 
\biggpar{\frac{\eps}{\epsbarcrit}}^{3(\rho-\rhocrit)}}}
&& \text{if $\eps < \epsbarcrit$\;,} 
\label{eq:C_almost_full_parameters} 
\end{align}
where $\widehat C_1(\eps,\rho)$ denotes the upper bound on 
$\abs{C_1(\rho,\eps)}$ in~\eqref{eq:bound_C1}. 
Note that for $\eps \geqs \epscrit$, the important parameter is 
$g^2\sigma^2\gamma^{-3}$. In particular, \eqref{eq:superexponential} implies 
\begin{equation}
 \abs{C_0^{\gamma,g,\sigma}(\eps,\rho)} \lesssim 
 \frac{g\sigma^2}{\gamma} 
 \exp\biggset{\frac{d-\rho}{\rho-\rhocrit} 
\biggbrak{\log\biggpar{\frac{g^{2/3}\sigma^{2/3}}{\gamma(\rho-\rhocrit)}} + 
\Order{1}}}\;.
\end{equation} 
A similar relation holds for 
$C_1^{\gamma,g,\sigma}(\eps,\rho)$ for $\eps 
\geqs \epsbarcrit$. Thus if $\gamma$, $g$ and $\sigma$ are fixed, the 
counterterms diverge in the same way as for $\gamma=g=\sigma=1$ as 
$\rho\searrow\rhocrit$. However, if $\gamma$, $g$ and $\sigma$ are allowed to 
depend on $\rho$, new regimes can occur. 

\item	The above choice of counterterms is not unique. In this work, we 
have chosen the BPHZ renormalisation, which is natural in some sense. However, 
as shown in~\cite{BrunedHairerZambotti}, the set of all potential choices of 
counterterms is parametrised by a group, called the renormalisation group. This 
group can be very large, since its dimension as a Lie group is equal to the number of symbols in 
the model space having negative degree. However, in our case only a 
two-parameter family of counterterms really matters: this family is obtained by 
adding constants to both $C_0(\eps,\rho)$ and $C_1(\eps,\rho)$. It is 
interesting to note that a one-parameter family of these choices of 
counterterms can be realised by a simple shift $v = u + k$ of the random field, 
where $k$ is a constant. Indeed, the equation for $v$ reads 
\begin{equation}
\partial_t v - \Delta^{\rho/2} v = 
v^2 + \bigpar{C_0(\eps,\rho) - C_1(\eps,\rho) k + k^2}  
+ (C_1(\eps,\rho)- 2k) v + \xi^{\eps}\;.
\end{equation}
In fact, one can observe that this is nothing but the equation obtained by 
applying the BPHZ renormalisation to the equation
\begin{equation}
 \partial_t v - \Delta^{\rho/2} v 
 = v^2 - 2k v + k^2 + \xi\;.
\end{equation} 
Indeed, the term $C_1(\eps,\rho) k$ comes from the fact that almost full 
binary trees (as defined in Section~\ref{sec:modelspace} below) can be 
generated by $k v$, and they will come with a factor $k$. Note that 
time-dependent shifts are currently out of the scope of the general theory, 
though one may expect that they lead to time-dependent renormalisation 
constants.

\item 	A common way to analyse the effect of the interaction term as 
$\rho\searrow\rhocrit$ is to study moments of the solution of the form
\begin{equation}
\bigexpec{u(t,x_1)u(t,x_2)u(t,x_3)}\;.
\end{equation}
So far, such moments have been computed only for very specific models such as 
the two-dimensional parabolic Anderson model, see~\cite{GuXu_18}. The main issue 
of such an approach is that in our case, the solutions are only local in time. 
However, it may be possible to obtain moment estimates for the process stopped 
when its H\"older norm exceeds some large threshold, and analysing their 
behaviour as $\eps\to0$ and $\rho\searrow\rhocrit$ may yield information on the 
potential convergence to a non-trivial model.
\end{itemize}


\section{Model space and renormalised equation}
\label{sec:modelspace} 

In order to apply the theory of regularity structures, the first step is to 
introduce a model space. This is a graded vector space spanned by abstract 
symbols, which allow to represent solutions of~\eqref{eq:laplacienrho} by 
an abstract fixed-point equation of the form 
\begin{equation}
\label{eq:abstract_FP} 
 U = \cI_\rho(\Xi + U^2) + P(U)\;.
\end{equation} 
Here $U$ represents the solution, $\Xi$ stands for space-time white noise, 
$\cI_\rho$ is an abstract integration operator standing for convolution with 
the fractional heat kernel, and $P(U)$ is a polynomial part, required by a 
recentering procedure. 

More precisely, let $\fraks=(\rho,1,\dots,1)\in \R_+^{d+1}$ be the scaling 
associated with the fractional Laplacian. Then we construct a set of symbols 
$\tau$, each admitting a degree $\abss{\tau}\in\R$, in the following way.

\begin{itemize}
\item 	For each multiindex $k=(k_0,\dots,k_d)\in\N_0^{d+1}$, 
we define the polynomial symbol $\mathbf{X}^k = X_0^{k_0}\dots X_d^{k_d}$, 
which has degree $\abss{\mathbf{X}^k} = \abss{k} = \rho k_0 + k_1 + \dots + 
k_d$. In particular, $\mathbf{X}^0$ is denoted $\unit$ and has degree 
$\abss{\unit}=0$. 

\item 	The symbol $\Xi$ representing space-time white noise has degree 
$\abss{\Xi} = -\frac12(\rho+d) - \kappa$, where $\kappa>0$ is arbitrarily 
small.  

\item 	If $\tau, \tau'$ are two symbols, then $\tau\tau'$ is a new symbol of 
degree $\abss{\tau\tau'} = \abss{\tau} + \abss{\tau'}$. 

\item 	Finally, if $\tau$ is a symbol which is not of the form $\mathbf{X}^k$, 
then $\cI_\rho(\tau)$ denotes a new symbol of degree $\abss{\tau} + \rho$, 
while for $k\in\N_0^{d+1}$, $\partial^k\cI_\rho(\tau)$, 
stands for a new symbol of degree $\abss{\tau} + \rho - \abss{k}$ (where we use 
the multiindex notation $\partial^k = 
\partial_t^{k_0}\partial_{x_1}^{k_1}\dots\partial_{x_d}^{k_d}$).
\end{itemize}

It is convenient to represent symbols by trees, in which edges stand for 
integration operators $\cI_\rho$, leaves stand for noise symbols $\Xi$, and 
multiplication of symbols is represented by joining them at the root. For 
instance, 
\begin{equation}
 \RSV = \cI_\rho(\Xi)^2\;, 
 \qquad 
 \RScombhSix = \Bigl[ \cI_\rho\Bigl(\cI_\rho\bigl(\cI_\rho(\Xi)^2\bigr) 
\cI_\rho(\Xi)\Bigr) \Bigr]^2\;.
\end{equation} 
Multiplication by a polynomial symbol $\mathbf{X}^k$ is represented by adding a 
node decoration $k$ to the relevant node of the tree, while derivatives 
$\partial^\ell\cI_\rho$ are denoted by edge decorations $\ell$. Thus for 
instance
\begin{equation}
 \RScombOnenodeedge{k}{\ell}
 = \cI_\rho(\mathbf{X}^k \partial^\ell\cI_\rho(\Xi))\;.
\end{equation}
The degree of a tree with $p$ leaves (for the
noise), $q$ edges (for integration operators), node decorations of total
exponent $k$ and edge decorations of total exponent $\ell$ is given by 
\begin{equation}
\label{eq:deg_pqkl} 
 \abss{\tau} = \biggpar{-\frac{\rho+d}{2}-\kappa} p + \rho q + \abss{k} -
\abss{\ell}\;.
\end{equation}
Not all symbols are needed to represent the abstract fixed-point 
equation~\eqref{eq:abstract_FP}. In fact, for its right-hand side, we only need 
the smallest set $T$ such that 
\begin{itemize}
\item 	$\mathbf{X}^k\in T$ for any $k\in\N_0^{d+1}$,
\item 	$\Xi\in T$,
\item if $\tau, \tau' \in T$,  $ k \in \N_0^{d+1}$,  one has  $  X^k \cI_\rho(\tau) , \cI_\rho(\tau)  \cI_\rho(\tau') 
\in T$,
\item 	if $d > 2$ and $\tau \in T$, then 
$\partial_{x_i} \cI_\rho(\tau)\in T$  for every $ 1 \leqs i \leqs d $.
\end{itemize}
We denote by $ \cT $ the linear span of $ T $. It is a consequence of local subcriticality that $T$ has only finitely many 
symbols of degree smaller than any $\alpha<\infty$ (see~\cite[Lemma 
8.10]{Hairer1}). The difference between $d\leqs2$ and $d>2$ is due to the 
fact that for $d \leqs 2$, one has $\rho < 1$ when $\rho$ is close to $\rhocrit 
= \frac{d}{3} \leqs \frac{2}{3}$. This means that the abstract operator  
$\partial_{x_i}\cI_\rho$ decreases the degree of the tree. Therefore, if we 
were to keep this rule, we would break subcriticality. For both cases, we have 
exhibited rules which are complete in the sense that they are stable under the 
action of the renormalisation. 

Let $T_- \subset T$ denote the set of symbols/decorated trees of negative 
degree, and $\cT_{-}$ (resp.\ $\hat \cT_{-}$) the linear span of the forests 
composed of elements in $T_{-}$ (resp.\ $T$). On $\cT_-$ we define a commutative 
and associative forest product. The product of two forests $\tau_1$ and $\tau_2$ 
is simply the forest containing all the trees of both forests, where the same 
tree may occur several times. The neutral element for this product is the empty 
forest, that we will denote by $\funit$. 

The structure of the trees in $T_-$ will be very important later on to 
control the renormalisation constants, which will be expressed in terms of sums 
over all trees of negative degree.
We know from~\cite[Prop.~4.17]{BK17} that trees in $T_-$ are necessarily  
either \emph{\full}\ binary trees (every vertex has either two children or no 
child), in which case $q=2p-2$, or \full\ binary trees with one edge missing 
(then $q=2p-1$), which we will call \emph{\afull}\ binary trees. It turns out 
that for symmetry reasons, \full\ binary trees can only contribute to the 
renormalized equation if they contain no nontrivial node decoration, while the 
\afull\ ones can contain one node decoration $k$ with $\abss{k}=1$. 
Furthermore, \eqref{eq:deg_pqkl} implies that the latter can only have negative 
degree if $d>3$. 

The form of the renormalised equation can be determined using the methods
introduced in~\cite{BrunedHairerZambotti} and expanded in~\cite{BCCH17}. 
As shown in~\cite[Thm. 2.21]{BCCH17}, it has the form~\eqref{eq:renormalised} 
with 
\begin{equation}
\label{eq:counterterm} 
 C(\eps,\rho,u) = \sum_{\tau\in T_-^{F}} c_\eps(\tau) 
\frac{\Upsilon^F(\tau)(u)}{S(\tau)}\;, 
\end{equation} 
where the terms $\Upsilon^F(\tau)(u)$ describe the effect of the nonlinearity 
$F(u,\xi) = u^2+\xi$, $S(\tau)$ is a symmetry factor, and $c_\eps(\tau)$ is the 
expectation of the element of the Wiener chaos represented by $\tau$.  

More precisely, the terms $\Upsilon^F(\tau)(u)$ are elementary differential 
operators defined recursively by $\Upsilon^F(\Xi)(u) = 1$ and 
\begin{equation}
\label{eq:def_Upsilon} 
\Upsilon^F\biggl(\mathbf{X}^{k}  \prod_{j=1}^m 
\mathcal{I}_{\rho}[\tau_j] \biggr)(u) = \left( \prod_{j=1}^{m}  
\Upsilon^{F}(\tau_j)(u) \right)  \partial^{k} \partial_{u}^{m}  u^{2}\;.
\end{equation}
We write $T_-^{F}$ for the subset of elements of $T_-$ for which $\Upsilon^{F}$ 
is non-zero, see~\cite[Def.~2.12]{BCCH17}. We could extend the previous 
definition of $\Upsilon^{F}$ to elements of the form 
$\partial_{x_i}\cI_\rho(\tau_j)$ by using the derivative  
$\partial_{\partial_{x_i}u}$. However, such a derivative applied to $F$ gives 
zero, which is why we omit this case in the definition of $\Upsilon^{F}$.

\begin{lemma}
Let $\ninner(\tau)$ be the number of \emph{inner nodes} of $\tau\in T_-$, where 
an inner node is any node which is not a leaf (including the root). Then 
\begin{equation}
 \Upsilon^F(\tau)(u) = 
\begin{cases}
2^{\ninner(\tau)} & \text{if $\tau$ is a \full\ binary tree\;,} \\
2^{\ninner(\tau)} u & \text{if $\tau$ is an \afull\ binary tree without 
decoration $X_i$\;,} \\
2^{\ninner(\tau)} \partial_{x_i}u & \text{if $\tau$ is an \afull\ binary 
tree 
with a decoration $X_i$\;.} 
\end{cases}
\end{equation} 
\end{lemma}
\begin{proof}
By induction on the size of the tree. The base case follows from 
$\ninner(\Xi)=0$. If $\tau$ is a \full\ binary tree, then it can be written as 
$\tau=\cI_\rho(\tau_1)\cI_\rho(\tau_2)$, where each $\tau_i$ is a \full\ tree 
with $n_i$ inner nodes. Then~\eqref{eq:def_Upsilon} and the induction hypothesis 
yield $\Upsilon^F(\tau)(u) = 2^{n_1+n_2+1}$, where $n_1+n_2+1$ is exactly the 
number of inner nodes of $\tau$. 

If $\tau$ is an \afull\ tree without decoration, there are two 
possibilities. Either $\tau = \cI_\rho(\tau_1)$ is a planted tree, where 
$\tau_1$ is a \full\ tree with $n_1$ inner nodes. Then~\eqref{eq:def_Upsilon} 
yields $\Upsilon^F(\tau)(u) = 2^{n_1+1} u$, where $n_1+1$ is the number of inner 
nodes of $\tau$. Or $\tau=\cI_\rho(\tau_1)\cI_\rho(\tau_2)$, where $\tau_1$ is 
\full\ with $n_1$ inner nodes, and $\tau_2$ is \afull\ with $n_2$ inner 
nodes. In that case, we obtain $\Upsilon^F(\tau)(u) = 2^{n_1+n_2+1}u$, where  
$n_1+n_2+1$ is again the number of inner nodes of $\tau$. 

The case of an \afull\ tree with decoration $X_i$ is straightforward, 
because then $\partial^k = \partial_{x_i}$ commutes with the other terms. 
\end{proof}

The second new quantity appearing in~\eqref{eq:counterterm} is the symmetry 
factor $S(\tau)$. It is defined inductively by setting $S(\Xi) = 1$, while if 
$\tau$ is of the form $\mathbf{X}^{k} \bigl( 
\prod_{j=1}^m \mathcal{I}_{\rho}[\tau_j]^{\beta_j}\bigr)$
with $\tau_i \neq \tau_j$ for $i \neq j$, then 
\begin{align}
S(\tau)
=
k!
\Big(
\prod_{j=1}^{m}
S(\tau_{j})^{\beta_{j}}
\beta_{j}!
\Big)\;.
\end{align}

\begin{lemma}
Let $\nsym(\tau)$ be the number of inner nodes of $\tau\in T_-$ having two 
identical lines of offspring. Then $S(\tau) = 2^{\nsym(\tau)}$. 
\end{lemma}
\begin{proof}
First note that $k!=1$ for any $\tau\in T_-$. Then the proof proceeds by 
induction on the size of the tree, noting that $m=1$ and $\beta_1=2$ whenever 
two identical trees are multiplied, while $m=2$ and $\beta_1=\beta_2=1$ 
when two different trees are multiplied, and $m=\beta_1=1$ when $\tau$ is a 
planted tree of the form $\cI_\rho(\tau_1)$. 
\end{proof}

\begin{remark}
Note that $S(\tau)$ is exactly the order of the symmetry group of the 
tree, which is generated by the $\nsym(\tau)$ reflections around symmetric 
inner nodes. For instance, $S(\tau) = 2$ for a comb tree, that is, a \full\ 
binary tree in which each generation but the root has exactly two individuals, 
i.e. 
\begin{equation}
 S(\RSV) = S(\RScombHTwo) = S(\RScombhThree) 
 = S(\RScombhFour) = \dots = 2\;.
\end{equation} 
Maximal symmetry is reached for regular trees, in which all individuals of the 
$s$ first generations have exactly two offspring, while those of the last 
generation have no offspring. For such a tree, $\nsym(\tau)=2^s-1$, and thus 
$S(\tau) = 2^{2^s-1}$, e.g.
\begin{equation}
\label{eq:regular_tree} 
 S(\RSV) = 2\;, 
 \qquad 
 S(\RStreeFour) = 2^3\;, 
 \qquad 
 S(\RStreeEight) = 2^7\;.
 \qedhere
\end{equation} 
\end{remark}

The final new quantity appearing in~\eqref{eq:counterterm} is the 
$\eps$-dependent factor $c_\eps(\tau)$, which is related to the expectation of 
the model of $\tau$. We analyse it in the next sections. 


\section{Canonical model}
\label{sec:modelexp} 

As in~\cite[Section~5]{Hairer1}, we decompose the fractional heat kernel 
$P_\rho$ as the sum  
\begin{equation}
\label{eq:PKR} 
 P_\rho(z) = K_\rho(z) + R_\rho(z)\;,
\end{equation} 
where $R_\rho$ is smooth and uniformly bounded in $\R^{d+1}$, while $K_\rho$ 
is compactly supported and has special algebraic properties. More precisely,
let 
\begin{equation}
\label{eq:def_scaling} 
 \abss{z} = |z_0|^{1/\rho} + \sum_{i=1}^d |z_i|
\end{equation} 
be the pseudonorm associated with the fractional scaling. Then
by~\cite[Lemma~5.5]{Hairer1}, we may assume that $K_\rho$ is supported in the 
ball $\setsuch{z}{\abss{z}\leqs 1}$, that $K_\rho=P_\rho$ in the 
ball $\setsuch{z}{\abss{z}\leqs \frac12}$, and that $K_\rho$ integrates to zero 
all polynomials of degree up to $2$. In addition, $K_\rho$ and its derivatives 
satisfy a number of analytic bounds, cf.~\cite[(3.1)--(3.4)]{BK17}. 
See also~\cite{chiarini2019} for a derivation of the associated Schauder 
estimate. We also assume the following two properties for the kernel $ K^\eps_\rho = \varrho^\eps * K_\rho $:
\begin{enumerate}
	 \item 	Non-anticipation: $K^\eps_\rho(t,x) = 0$ for $t\leqs -\eps^\rho$;
\item 	Spatial symmetry: $K^\eps_\rho(t,-x) = K^\eps_\rho(t,x)$.
\end{enumerate}

To any symbol $\tau\in T$, we associate the \emph{canonical model} 
$\bPi^\eps\tau$, defined (cf.~\cite[proof of Prop.~8.27]{Hairer1}) by 
\begin{equation}
\label{eq:model0}
 (\bPi^\eps\unit)(z) = 1\;, \qquad 
 (\bPi^\eps X_i)(z) = z_i\;, \qquad 
 (\bPi^\eps \Xi)(z) = \xi^\eps(z)\;,
\end{equation} 
and extended inductively by the relations 
\begin{align}
 (\bPi^\eps \tau\bar\tau)(z) &= (\bPi^\eps \tau)(z)(\bPi^\eps\bar\tau)(z)\;, \\
 (\bPi^\eps \partial^k\cI_\rho\tau)(z) &= \int \partial^k K_\rho(z-\bar z)
 (\bPi^\eps \tau)(\bar z)\6\bar z\;.
\label{eq:model1} 
\end{align} 
We then set 
\begin{equation}
 E(\tau) = \bigexpec{(\bPi^\eps\tau)(0)}\;,
\end{equation} 
which has in general the form of a Gaussian iterated integral. The computations 
will be greatly simplified by removing symbols that are in the 
kernel of $E$. We denote by  $I_{E}$ the ideal generated by forests having at 
least one decorated tree $ \tau $ satisfying one of the following properties:
\begin{itemize}
\item $\tau$ has an odd number of leaves;
\item $\tau$ is a planted tree (i.e., of the form $\cI_\rho(\tau')$  or 
$\partial_{x_i} \mathcal{I}_{\rho}(\tau')$);
\item $\tau $ has one $X_i$ as a node decoration and no edge of the 
form $\partial_{x_i} \mathcal{I}_{\rho}$. 
\end{itemize}

\begin{prop} 
\label{prop:ker_E} 
Let $\tau$ be a decorated tree. Then $E(\tau) = 0$ whenever $\tau \in I_{E}$.
\end{prop}
\begin{proof} 
If $\tau$ has an odd number of leaves, then $(\bPi^\eps\tau)(0)$  is 
centered as the product of an odd number of centered Gaussians
has zero mean. If $\tau = \mathcal{I}_{\rho}(\tau')$, then $E(\tau) 
= 
K_{\rho} * E(\tau') = E(\tau') K_{\rho} * 1$ by translation invariance. The 
term 
$K_{\rho} * 1$ is equal to zero by definition of the kernel $K_{\rho}$ 
($K_{\rho}$ integrates polynomials to zero up to a certain order). For the 
last case, the conclusion follows by noticing that $(\bPi^\eps \tau)(t,-x) = - 
(\bPi^\eps \tau)(t,x)$. 
\end{proof}


\subsection{Simplifying the twisted antipode}
\label{ssec:antipode}

The $\eps$-dependent coefficients $c_\eps(\tau)$ are defined by 
\begin{equation}
\label{eq:def_ceps} 
 c_\eps(\tau) = E(\tilde\cA_-\tau)\;,  
\end{equation}
where $\tilde\cA_-: \cT_- \rightarrow \hat \cT_-$ is a linear map encoding the 
renormalisation procedure, called the \emph{twisted antipode}. The twisted 
antipode is defined in~\cite[Proposition~6.6]{BrunedHairerZambotti}, in terms of 
a coaction $\Deltam : \hat \cT_- \rightarrow \cT_- \otimes \hat \cT_- $ which is 
close in spirit to the Connes--Kreimer extraction-contraction coproduct 
introduced in~\cite{ConnesKreimer98}. However, the coaction $\Deltam$ is more 
complicated than the one used in~\cite{ConnesKreimer98}, because it acts on 
decorated trees, where the decorations encode multiplication by monomials and 
derivatives appearing in Taylor expansions. This results in rather complicated 
expressions for the twisted antipode, cf.\ 
Proposition~\ref{prop:twisted_antipode_tree} below. It turns out, however, that 
thanks to Proposition~\ref{prop:ker_E}, in our case many terms of 
$\tilde\cA_-\tau$ give a vanishing contribution when applying $E$. The purpose 
of this section is to derive the simplified expression~\eqref{eq:AtildeE} of 
$\tilde\cA_-$, which only involves extraction of subtrees and contractions, 
without any decorations. Furthermore, this simplified expression allows to 
define $\tilde\cA_-$ in an iterative way, which does not involve the coaction 
$\Deltam$ at all.

In order to derive the simplified expression of the twisted antipode, we 
have to start with the general construction given 
in~\cite{BrunedHairerZambotti}.
The twisted antipode can be defined inductively by setting 
$\tilde\cA_-(\funit)=\funit$ for the empty forest $\funit$, and 
\begin{equation}
\label{eq:antipode} 
 \tilde\cA_-\tau = - \hat\cM_- (\tilde\cA_- \otimes \id)(\Deltam \tau -
\tau\otimes\unit)\;,
\end{equation}
cf.~\cite[Prop.~6.6]{BrunedHairerZambotti}. Here $\hat\cM_-$ is the 
multiplication operator (acting on forests), and 
$\tau$ is a tree of negative degree (we have omitted the 
natural injection of $\cT_-$ into $\hat \cT_-$ because we view $\cT_-$ 
as a subset of $\hat \cT_-$). Elements of $\hat\cT_-$ are of the form  
$(F,\Labn,\Labe)$ where $F$ is a forest with node set $N_F$ and edge set 
$E_F$, $ \Labn : N_F \rightarrow \N_0^{d+1} $ represents the node decoration 
and $\Labe : E_F \rightarrow \N_0^{d+1}$ represents the edge decoration. The 
forest product is defined by 
 \begin{equation}
 (F,\Labn,\Labe) \cdot (G,\bar \Labn,\bar \Labe) = (F \cdot G,\bar \Labn + 
\Labn,\bar \Labe + \Labe)\;, 
 \end{equation} 
where the sums $\bar \Labn + \Labn$ and  $\bar \Labe + \Labe$ mean that 
decorations defined on one of the forests are extended to the disjoint union by 
setting them to vanish on the other forest. Then the map $\Deltam : \cT 
\rightarrow \cT_- \otimes \cT$ defined in~\cite{BrunedHairerZambotti} is given 
for $T^{\Labn}_{\Labe} \in \cT$ by 
\begin{equation}
\label{eq:co-action_minus}
 \Deltam  T^{\Labn}_{\Labe}  = 
 \sum_{A \in \Adm(T) } \sum_{\Labe_A,\Labn_A}  \frac1{\Labe_A!}
\binom{\Labn}{\Labn_A}
 (A,\Labn_A+\pi\Labe_A, \Labe \restr E_A) 
   \otimes( \cR_A T, [\Labn - \Labn_A]_A, \Labe + \Labe_A)\;, 
\end{equation}
where we use the following notations. 
\begin{itemize}
\item 	Factorials and binomial coefficients are understood in multiindex 
notation, and the latter vanish unless $\Labn_A$ is pointwise smaller than or equal to
$\Labn$. 
\item 	For $C \subset D$ and $f : D \rightarrow \N_0^{d+1}$, $f \restr C$ 
is the restriction of $f$ to $C$. 
\item 	The first sum runs over $\Adm(T)$, the set of all subforests $A$ of $T$, 
where $A$ may be empty. The second sum runs over all  $\Labn_A : N_A 
\rightarrow 
\N_0^{d+1}$ and $\Labe_{A} : \partial(A,T) \rightarrow \N_0^{d+1}$ where 
$\partial(A,T)$ denotes the edges in $E_T \setminus E_A$ that are adjacent to 
$N_A$ as a child, not a parent.
\item  	We write $\cR_A T$ for the tree obtained by contracting the connected 
components of $A$. Then we have an action on the decorations, in the sense that 
for $f : N_T \rightarrow \N_0^{d+1}$ and $A \subset T$, one has $[f]_A(x) = 
\sum_{x \sim_{A} y} f(y)$, where $x$ is an equivalence class of $\sim_A$, and 
$x \sim_A y$ means that $x$ and $y$ are connected in $A$. For $g : E_T  
\rightarrow \N_0^{d+1}$, we define for every $x \in N_T$, $(\pi g)(x) = 
\sum_{e=(x,y) \in E_T} g(e)$.
\end{itemize}

\begin{remark}
The name \lq\lq twisted antipode\rq\rq\ is due to the fact that $\tilde\cA_-$ 
satisfies the relation 
\begin{equation}
 \hat\cM_-(\tilde\cA_-\otimes\id) \Deltam \tau = \funit^\star(\tau)\funit\;,
\end{equation} 
where $\funit^\star$ is the projection on the empty forest and $\Deltam$ is a 
coaction (but not a coproduct). If the spaces 
$\cT_-$ and $\hat\cT_-$ were equal, 
$(\cT_-,\cdot,\Delta^-,\funit,\funit^\star,\tilde\cA_-)$ would be a Hopf 
algebra, similar to the extraction-contraction Connes--Kreimer Hopf algebra 
of~\cite{ConnesKreimer98} which involves trees without decoration. 
\end{remark}

\begin{example}
Consider the case $\tau=\RSV$ (with zero node and edge decorations). Then 
\begin{equation}
 \Deltam \,\RSV 
 = \funit \otimes\RSV + 2\sum_k \frac{1}{k!} \RSInode{k} \otimes \RSIedge{k} 
 + \RSV \otimes \unit\;,
\end{equation} 
where the sum is over $k\in\N_0^{d+1}$ such that the extracted symbol has 
negative degree. Here the first term corresponds to extracting $A=\funit$, the 
second one to $A=\RSI$, and the last one to $A=\RSV$. 

Consider now a case when the tree has one node decoration, say $\tau = 
\RScombTwonode{k}$. Then 
\begin{equation}
 \Deltam  \RScombTwonode{k}
 =\funit \otimes \RScombTwonode{k} 
 + \sum_m\frac{1}{m!} \, \RSInode{m} \otimes \RScombTwonodeb{k}
 + \sum_{\ell} \binom{k}{\ell}  
\RScombTwonode{\ell} \otimes \Rnode{k-\ell}\;, 
\end{equation} 
where we first extract $A = \funit$, then $A=\RSI$ and finally $A = 
\RScombTwo$. As before, the sums on $\ell$ and $m$ are restricted by the 
fact that the extracted symbol has to have a negative degree.
\end{example}

As a short-hand notation for~\eqref{eq:co-action_minus}, we use
\begin{equation}\label{co-action minus}
 \Deltam  T^{\Labn}_{\Labe}  = 
 \sum_{A \in \Adm(T) } \sum_{\Labe_A,\Labn_A}  \frac1{\Labe_A!}
\binom{\Labn}{\Labn_A}
 A^{\Labn_A+\pi\Labe_A}_{\Labe} 
 \otimes  \cR_A T^{\Labn - \Labn_A}_{ \Labe + \Labe_A} \;. 
\end{equation}
We extend this map to $\hat \cT_{-}$ by multiplicativity regarding the forest 
product. Then one can turn this map into a coproduct $\Deltam : \cT_-  
\rightarrow \cT_{-} \otimes \cT_{-}$ and obtain a Hopf algebra for $\cT_{-}$ 
endowed with this coproduct and the forest product see 
\cite[Prop.~5.35]{BrunedHairerZambotti} and \cite[Cor.~6.37]{BrunedHairerZambotti}. The main difference here is that 
we do not consider extended decorations, but the results for the Hopf algebra 
are the same as in \cite{BrunedHairerZambotti}. 

Using the definition of $\Delta^-$, one can write a recursive formulation 
for $\tilde\cA_-$ in which one doesn't see any tensor product. It is convenient 
to introduce the \emph{reduced coaction} $\Deltatm \tau = \Deltam 
\tau - \tau \otimes \unit - \funit \otimes \tau$. Then, using Sweedler's 
notation, if $\Deltatm \tau = \sum_{(\tau)} \tau' \otimes \tau''$ 
one has
\begin{equation}
\tilde\cA_-\tau = - \tau - \sum_{(\tau)} (\tilde\cA_- \tau') \tau'' \;.
\end{equation}

\begin{prop}
\label{prop:twisted_antipode_tree} 
For a decorated tree $ T^{\Labn}_{\Labe} $ with negative degree, one has the 
relation 
\begin{align}\label{co-action minus}
 \tilde\cA_- T^{\Labn}_{\Labe} & = - T^{\Labn}_{\Labe} - 
 \sum_{A \in \Adm^{\star}(T) } \sum_{\Labe_A,\Labn_A}  \frac1{\Labe_A!}
\binom{\Labn}{\Labn_A}
\tilde\cA_- A^{\Labn_A+\pi\Labe_A}_{ \Labe}  \cdot  \cR_A T^{\Labn - \Labn_A}_{ 
\Labe + \Labe_A}\;, 
\end{align}
where $\Adm^{\star}(T)  = \Adm(T) \setminus \lbrace  \funit, T \rbrace $.
\end{prop}
\begin{proof} 
The proof follows from a straightforward manipulation of the definitions:
\begin{align}
\tilde\cA_-  T^{\Labn}_{\Labe} & = -\hat\cM_-(\tilde\cA_-\otimes\id)(\Deltam  
T^{\Labn}_{\Labe}-  T^{\Labn}_{\Labe} \otimes \unit)
\\ & = - \hat\cM_-(\tilde\cA_-\otimes\id) ( \funit \otimes T^{\Labn}_{\Labe}  
) - \hat\cM_-(\tilde\cA_-\otimes\id)\Deltatm  T^{\Labn}_{\Labe}
 \\ & = - T^{\Labn}_{\Labe} - \sum_{A \in \Adm^{\star}(T) } 
\sum_{\Labe_A,\Labn_A}  \frac1{\Labe_A!}
\binom{\Labn}{\Labn_A}
\tilde\cA_- A^{\Labn_A+\pi\Labe_A}_{\Labe} 
 \cdot  \cR_A T^{\Labn - \Labn_A}_{ \Labe + \Labe_A}  
\;, 
\end{align}
where we have treated separately the cases $A=\funit$ and $A=T$. 
\end{proof}

The construction of the twisted antipode can be substantially simplified by 
using Proposition~\ref{prop:ker_E}. Indeed, one has the property $\Deltam I_E 
\subset I_E \otimes \hat \cT_- + \cT_- \otimes I_E$, which makes $I_E$ a kind of 
biideal associated to $\Deltam$. Therefore, $\Deltam$ is a well-defined map from 
$\cT_-^{E}$ into $\cT_-^{E} \otimes  \hat \cT_-^{E}$, where $\cT_-^{E} = \cT_{-} 
/ I_E$ and $\hat \cT^{E}_{-} = \hat \cT_{-} / I_E$ (in other words, if 
$\tau'-\tau \in I_E$, then $\Deltam(\tau') - \Deltam(\tau)$ belongs to $I_E 
\otimes \hat \cT_- + \cT_- \otimes I_E$, and thus equivalence classes modulo 
$I_E$ are mapped into equivalence classes). 

In what follows, we will use the notation $\tilde\cA_-^{E}$ when $\tilde\cA_-$ 
is considered as acting on $\cT_-^{E}$. As the consequence of the biideal 
property, we get

\begin{prop} 
\label{prop:ceps} 
One has $c_\eps(\tau) =  E(\tilde\cA_-\tau) = E(\tilde\cA^{E}_-\tau)$.
\end{prop}
\begin{proof}
This follows from Proposition~\ref{prop:ker_E}, which implies $I_E \subset 
\ker E$.
\end{proof}

\begin{prop} 
\label{prop:AtildeE_minus} 
If we consider $\Deltam$ as a map from $\cT_-^{E}$ into $\cT_-^{E} \otimes  
\hat \cT_-^{E}$, then it reduces to an extraction-contraction map with some 
restrictions: for any tree $\tau \in \cT_-^{E}$, we have
\begin{equation}
\Deltam \tau = \sum_{\tau_1 \cdot \ldots \cdot \tau_n \subset_{E} \tau} \tau_1 
\cdot \ldots \cdot \tau_n \otimes \tau/(\tau_1 \cdot \ldots \cdot \tau_n)\;,
\end{equation}
where $\subset_E$ means that we consider all the subforests $ \tau_1 
\cdot \ldots \cdot \tau_n  $ of $\tau$ such that the trees $ \tau_i$ belong to $ \cT_-^{E}$, and 
$\tau/(\tau_1 \cdot \ldots \cdot \tau_n)$ denotes the tree 
obtained by contracting $\tau_1,\dots,\tau_n$ to a single node. 
Therefore, one can define a multiplicative map $\tilde \cA^{E}_{-}$ for the 
forest product as
\begin{equation}
\label{eq:AtildeE} 
\tilde\cA_-^{E}\tau 
= - \tau  - \sum_{ \funit \varsubsetneq \tau_1 \cdot \ldots \cdot \tau_n 
\varsubsetneq_{E} \tau} 
\tilde \cA_-^{E} (\tau_1 \cdot \ldots \cdot \tau_n) \cdot \tau / (\tau_1 
\cdot \ldots \cdot \tau_n)\;.
\end{equation}
\end{prop}
\begin{proof} 
The simplification for $\Deltam$ and $\tilde \cA_-^{E}$ comes from the precise 
description of $\cT_-^{E}$ which is composed of \full\ and \afull\ binary 
trees. Therefore, $\Deltam \tau$ does not contain any sum on the node 
decorations and there remains only the extraction-contraction procedure.
\end{proof}

\begin{remark}
The very simple expression~\eqref{eq:AtildeE} for the twisted antipode
is a direct consequence of the fact that we may remove trees with one $X_i$ as 
a node decoration. This expression may be useful for numerical computations of 
the constants.
\end{remark}

\begin{example}
\label{ex:Atilde_comb} 
We have $\tilde\cA_-^{E}(\RScombTwo) = -\RScombTwo$, since no nontrivial tree 
can be extracted. Therefore, we obtain 
\begin{align}
\tilde\cA_-^{E} \RScombhSix & \; = - \RScombhSix \; - 4 \;\tilde\cA_-^{E} ( 
\RScombTwo) \cdot  \RScombhSixa \; - 4 \; \tilde\cA_-^{E} (\RScombTwo \cdot  
\RScombTwo) \cdot \RScombhSixab
\\ & = - \RScombhSix \; + 4 \; \RScombTwo \cdot  \RScombhSixa \; - 4 \; 
\RScombTwo \cdot  \RScombTwo \cdot \RScombhSixab\;, 
\label{eq:tildeAE_Example} 
\end{align}
where $\RScombTwo \cdot  \RScombTwo  \in \cT_-$ and $\RScombhSixab \in 
\hat \cT_- \setminus \cT_-$. 
\end{example}


\subsection{From expectations to Feynman diagrams}
\label{ssec:feynman} 

We now discuss the computation of expectations $E(\tau)$, starting with some 
examples. 

\begin{example}
It follows from~\eqref{eq:model0} and~\eqref{eq:model1} that
\begin{equation}
 (\bPi^\eps \RSI)(0) = \int K_\rho(-z) \xi^\eps(z) \6z 
 = \int K^\eps_\rho(-z) \xi(\6z)\;,
\end{equation} 
where we have assumed that $\xi^\eps = \varrho^\eps * \xi$ for a scaled 
mollifier $\varrho^\eps$, and defined $K^\eps_\rho = K_\rho * \varrho^\eps$. 
Since this is a centred Gaussian random variable, we have $E(\RSI) = 0$, in 
accordance with Proposition~\ref{prop:ker_E}. It then follows from the defining 
property of space-time white noise that 
\begin{align}
\label{eq:E_cherry} 
E(\RSV) 
&= \bigexpec{(\bPi^\eps \RSV)(0)} 
= \bigexpec{(\bPi^\eps \RSI)(0)^2}\\
&= \int K^\eps_\rho(-z_1) K^\eps_\rho(-z_2) \bigexpec{\xi(\6z_1) \xi(\6z_2)} \\
&= \int K^\eps_\rho(-z_1)^2 \6z_1\;.
\qedhere
\end{align}
\end{example}

\begin{example}
A more complicated example is 
\begin{align}
E(\RStreeFour) 
&= \bigexpec{(\bPi^\eps\RSY)(0)^2} \\
&= \biggexpec{\biggpar{\int K_\rho(-z)K^\eps_\rho(z-z_1)K^\eps_\rho(z-z_2)
\xi(\6z_1)\xi(\6z_2)\6z}^2}\;.
\end{align}
Wick calculus implies that $\expec{\xi(\6z_1)\xi(\6z_2)\xi(\6\bar 
z_1)\xi(\6\bar z_2)}$ is a sum of three terms, which can be symbolised by the 
pairings 
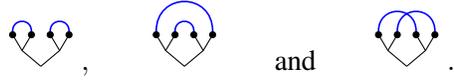
\begin{equation}
\tikz{%
\path[use as bounding box] (-0.43, -0.05) rectangle (0.43,0.92);
\FDtreeRegFour
\draw[pairing] (1) edge [distance=5,out=90,in=90] (2);
\draw[pairing] (3) edge [distance=5,out=90,in=90] (4);
\drawbox;
}\;, \qquad 
\tikz{%
\path[use as bounding box] (-0.43, -0.05) rectangle (0.43,0.92);
\FDtreeRegFour
\draw[pairing] (1) edge [distance=15,out=90,in=90] (4);
\draw[pairing] (2) edge [distance=5,out=90,in=90] (3);
\drawbox;
} \qquad \text{and} \qquad 
\tikz{%
\path[use as bounding box] (-0.43, -0.05) rectangle (0.43,0.92);
\FDtreeRegFour
\draw[pairing] (1) edge [distance=10,out=90,in=90] (3);
\draw[pairing] (2) edge [distance=10,out=90,in=90] (4);
\drawbox;
}\;.
\end{equation} 
The first pairing yields 
\begin{equation}
 \biggpar{\int K_\rho(-z)K^\eps_\rho(z-z_1)^2 \6z\6z_1}^2 = 0\;,
\end{equation} 
owing to the fact that $K_\rho$ integrates to zero. By symmetry, the second and 
third pairing yield the same value, namely 
\begin{equation}
\label{eq:pairing02} 
 \int K_\rho(-z)K^\eps_\rho(z-z_1)K^\eps_\rho(\bar z-z_1) 
 K_\rho(-\bar z)K^\eps_\rho(z-z_2)K^\eps_\rho(\bar z-z_2) 
 \6z\6\bar z\6z_1\6z_2\;.
\end{equation} 
It is convenient to represent such an integral graphically by the diagram 
\begin{equation}
\label{eq:pairing03} 
  \tikz{
 \path[use as bounding box] (-0.85, -0.1) rectangle (0.85,1.7);
 \node[rootnode] (0) at (0,0) {};
 \node[blacknode] (x) at (-0.75,0.8) {};
 \node[blacknode] (y) at (0.75,0.8) {};
 \node[blacknode] (z) at (0,0.8) {};
 \node[blacknode] (a) at (0,1.6) {};
 \draw[Kedge] (x) -- (0);
 \draw[Kedge] (y) -- (0);
 \draw[Kepsedge] (z) -- (y);
 \draw[Kepsedge] (z) -- (x);
 \draw[Kepsedge] (a) -- (x);
 \draw[Kepsedge] (a) -- (y);
 \drawbox;
 }\;,
\end{equation} 
where small black vertices denote integration variables, the large green vertex 
denotes the point $0$, solid arrows denote kernels $K_\rho$, and broken arrows 
denote kernels $K^\eps_\rho$. The benefit of the graphical 
representation~\eqref{eq:pairing03}, besides saving space, is that it will 
allow to represent in a more visual way the extraction-contraction operations 
associated with renormalisation. 
\end{example}

These examples motivate the following definition, which is a particular case 
of~\cite[Def.~2.1]{Hairer_BPHZ}. 

\begin{definition}[Feynman diagram]
\label{def:feynman_diagram} 
A \emph{Feynman diagram} (or, more precisely, a \emph{vacuum diagram}) is a 
finite oriented graph $\Gamma = (\sV,\sE)$, with a distinguished node 
$v^\star\in\sV$, and in which each edge $e\in\sE$ has a 
\emph{type} $\mathfrak{t}$ belonging to a finite set of types $\mathfrak{L}$.
With each type $\mathfrak{t}\in\mathfrak{L}$, we associate a degree 
$\deg(\mathfrak{t})\in\R$ and a kernel 
$K_{\mathfrak{t}}:\R^{d+1}\setminus\set{0}\to\R$. The 
\emph{degree} of\/ $\Gamma$ is defined by 
\begin{equation}
\label{eq:deg_Gamma} 
  \deg(\Gamma) = (\rho+d) (\abs{\sV}-1) + \sum_{e\in\sE} \deg(e)\;,
\end{equation} 
where $\abs{\sV}$ denotes the cardinality of $\sV$ and $\deg(e) = 
\deg(\mathfrak{t}(e))$. The \emph{value} of the diagram $\Gamma = (\sV,\sE)$ is 
defined as 
\begin{equation}
\label{eq:EGamma} 
 E(\Gamma) = \int_{(\R^{d+1})^{\sV\setminus v^\star}} \prod_{e\in\sE} 
K_{\mathfrak{t}(e)}
 (z_{e_+}-z_{e_-}) \6z\;,
\end{equation} 
where each oriented edge is written $e=(e_-,e_+)\in\sV^2$, and $z_{v^\star}=0$. 
\end{definition}

The graph in~\eqref{eq:pairing03} is an example of Feynman diagram, with a set 
of types $\mathfrak{L}$ consisting of $2$ types corresponding to the kernels 
$K_\rho$ and $K^\eps_\rho$. We define their degrees by 
\begin{equation}
 \deg({\tikz{
 \path[use as bounding box] (0, -0.1) rectangle (1.2,0.1); 
 \draw[Kedge] (0,0) -- (1.2,0);
 \drawbox;
 }}) 
 =  \deg({\tikz{
 \path[use as bounding box] (0, -0.1) rectangle (1.2,0.1); 
 \draw[Kepsedge] (0,0) -- (1.2,0);
 \drawbox;
 }})
 =-d\;.
 \label{eq:edge_degree1} 
\end{equation}
To each symbol $\tau\in T$ without decorations, we associate a 
linear combination of Feynman diagrams in the following way. 

\begin{definition}[Pairing]
Let $\tau\in T \setminus I_E$ be a symbol without decorations, and denote its 
set of leaves by $N_\tau$. A \emph{pairing} of $\tau$ is a partition $P$ of 
$N_\tau$ into two-elements blocks. We denote the set of pairings of $\tau$ by 
$\cP_\tau^{(2)}$. Then $\Gamma(\tau,P)$ is the Feynman diagram obtained by 
merging the leaves of a same block, and assigning to every edge adjacent to a 
former leaf the type $K^\eps_\rho$, and to all other edges the type $K_\rho$.  
\end{definition}

\begin{prop}
\label{prop:degE_Gamma} 
Let $\tau\in T \setminus I_E$. If $\tau$ has $p$ 
leaves and $q$ edges, then each $\Gamma(\tau,P)$ has $q+1-\frac{p}{2}$ 
vertices and $q$ edges. Therefore, 
\begin{equation}
\label{eq:deg_Gamma_tau} 
\deg(\Gamma(\tau,P)) = \abss{\tau} \Bigr\vert_{\kappa=0}
\end{equation}
holds for any $P\in\cP_\tau^{(2)}$. In addition, we have 
\begin{equation}
\label{eq:E_Gamma_tau} 
 E(\tau) = \sum_{P\in\cP_\tau^{(2)}} E(\Gamma(\tau,P))\;.
\end{equation}
\end{prop}
\begin{proof}
By~\eqref{eq:deg_pqkl}, we have $\abss{\tau} = -\frac{p}{2}(\rho+d) + \rho q - 
p\kappa$. Since $\tau$ is a tree, it has $q+1$ nodes, and therefore $q+1-p$ 
inner nodes. When contracting the $p$ leaves pairwise, one obtains a Feynman 
diagram with $q$ edges of type $K_\rho$ or $K^\eps_\rho$, and 
$q+1-p+\frac{p}{2}$ vertices. Therefore its degree is given by 
$-qd+(\rho+d)(q-\frac{p}{2})$, which agrees with~\eqref{eq:deg_Gamma_tau}. The 
relation~\eqref{eq:E_Gamma_tau} is then a direct consequence of the 
rules~\eqref{eq:model1} defining the model and Wick calculus. 
\end{proof}

The following simple result shows that we can limit the analysis to Feynman 
diagrams which are at least $2$-connected. 

\begin{lemma}
\label{lem:Gamma_1-connected} 
If\/ $\Gamma$ is $1$-connected (i.e., if one can split $\Gamma$ into two 
disjoint 
graphs by removing one edge), then $E(\Gamma)=0$.  
\end{lemma}
\begin{proof}
If $\Gamma=(\sV,\sE)$ is $1$-connected, then there exist two vertex-disjoint 
subgraphs $\Gamma_1=(\sV_1,\sE_1)$ and $\Gamma_2=(\sV_2,\sE_2)$ such that $\sV = 
\sV_1 \cup \sV_2$ and $\sE = \sE_1 \cup \sE_2 \cup \set{e_0}$. By a linear 
change of variables, we may arrange that $e_{0}=(v^\star,v_1)$ where 
$v_1\in\sV_1$. We thus obtain 
\begin{equation}
 E(\Gamma) = \int_{(\R^{d+1})^{\sV_1}} K_{\mathfrak{t}(e_0)}(z_1) 
 \prod_{e\in\sE_1} K_{\mathfrak t(e)}(z_{e_+}-z_{e_-}) \6z \,
 E(\Gamma_2) \;.
\end{equation} 
Performing the change of variables $z_v = \bar z_v + z_1$ for all 
$v\in\sV_1\setminus\set{v_1}$, we can factor out the integral over $z_1$. This 
integral vanishes by construction. 
\end{proof}


\subsection{Simplification rules for Feynman diagrams}
\label{ssec:feynman_simplify} 

Integrals of the type encountered above can be somewhat simplified by using the 
fact that $P_\rho$ is the kernel of a Markov semigroup, describing a 
rotationally symmetric $\rho$-stable L\'evy process (see for 
instance~\cite{Kwasnicki}). While this is not essential for the general 
argument, it reduces the size of diagrams and thus improves the graphical 
representation. It also allows to compute the explicit expressions for the 
renormalisation constants $A_0$ and $\bar A_0$ given in~\eqref{eq:A0} 
and~\eqref{eq:Abar0}. 

\begin{lemma}
\label{lem:CK-Green} 
Assume the scaled mollifier has the form $\varrho^\eps(t,x) = \eps^{-(\rho+d)} 
\varrho(\eps^{-\rho}t,\eps^{-1}x)$, where 
$\varrho(t,x)=\varrho_0(t)\varrho_1(x)$ is even in $x$, supported in a ball of 
scaled radius $1$, and integrates to $1$. Then $K^\eps_\rho$ satisfies the 
following properties for all $(t,x)\in\R^{d+1}$:
\begin{enumerate}
\item 	Chapman--Kolmogorov equation: there exists a function 
$R_1^\eps:\R^{d+2}\to\R$, uniformly bounded and integrable in its first two 
arguments, such that 
\begin{equation}
\label{eq:CK} 
 \int K^\eps_\rho(t,x-y) K^\eps_\rho(s,y) \6y = \tilde K^\eps_\rho(t+s,x)
 + R_1^\eps(t,s,x)\;,
\end{equation} 
where $\tilde K_\rho^\eps =  K_\rho^\eps * \varrho^\eps = K_\rho * \varrho^\eps 
* \varrho^\eps$ is a kernel with a different mollifier;
\item 	Green function: there exists a uniformly bounded function 
$R_2^\eps:\R^{d+1}\to\R$ such that 
\begin{equation}
\label{eq:Green} 
 \int_t^\infty K^\eps_\rho(s,x) \6s = -(G_\rho *_x P^\eps_\rho)(t,x)
 + R_2^\eps(t,x)\;,
\end{equation} 
where $G_\rho = (\Delta^{\rho/2})^{-1}$ is the Green function of the fractional 
Laplacian, $P_\rho^\eps =  P_\rho*\varrho^\eps$ and $*_x$ denotes convolution 
in space. 
\end{enumerate}
\end{lemma}
\begin{proof}
 For the first property, we use the Chapman--Kolmogorov relation $P_\rho(t,\cdot) *_x 
P_\rho(s,\cdot) = P_\rho(t+s,\cdot)$ to obtain 
\begin{align}
 K_\rho(t,\cdot) *_x K_\rho(s,\cdot) ={}& K_\rho(t+s,\cdot) 
 + R_\rho(t+s,\cdot) \\
&{}- K_\rho(t,\cdot) *_x R_\rho(s,\cdot)
 - R_\rho(t,\cdot) *_x K_\rho(s,\cdot)
 - R_\rho(t,\cdot) *_x R_\rho(s,\cdot)\;.
\end{align} 
Using the fact that $R_\rho$ is bounded and 
that $K_\rho$ and $R_\rho = P_\rho - K_\rho$ are integrable 
($P_\rho$ being integrable and $K_\rho$ having compact support), one obtains 
that all terms involving $R_\rho$ are bounded. The relation~\eqref{eq:CK} then 
follows upon convolving twice with $\varrho^\eps$. The last relations  
follows from the fact that
\begin{equation}
 \Delta^{\rho/2} \int_t^\infty P_\rho(s,\cdot) \6s = 
 \int_t^\infty \Delta^{\rho/2} \e^{s\Delta^{\rho/2}} \6s = 
 \e^{s\Delta^{\rho/2}} \biggr\vert_t^\infty = 
 - P_\rho(t,\cdot)\;.
\end{equation} 
Convolving with $G_\rho$, we obtain 
\begin{equation}
 \int_t^\infty P_\rho(s,x) \6s =
 - G_\rho *_x P_\rho(t,x)\;.
\end{equation} 
The result then follows by decomposing $P_\rho$ on the left-hand side into 
$K_\rho+R_\rho$, and convolving with $\varrho^\eps$. 
\end{proof}

Applying these properties to~\eqref{eq:E_cherry}, we obtain 
\begin{equation}
 E(\RSV) = \iint K^\eps_\rho(-t,-x)K^\eps_\rho(-t,x) \6x\6t
 = \int \tilde K^\eps_\rho(-2t,0) \6t + \Order{1}
 = \frac12 G^\eps_\rho(0) + \Order{1}\;,
\end{equation} 
where $G^\eps_\rho = \varrho_1^\eps *_x G_\rho$, $\Order{1}$ denotes a constant 
bounded uniformly in $\eps$, and we used the fact that $P_\rho(0,x) = 
\delta(x)$. Note that this implies the expression~\eqref{eq:A0} for the 
counterterm associated with $\RSV$. The expression~\eqref{eq:Abar0} for $\bar 
A_0$ is obtained by a similar argument applied to the 
element \smash{$\RScombTwo$}.

\begin{lemma}
\label{lem:KtoG} 
There exists a uniformly bounded function $R_3^\eps:\R^{2(d+1)}\to\R$ such that 
\begin{align}
 \int K^\eps_\rho(z_1-z) K^\eps_\rho(z_2-z) \6z 
 &= \int K^\eps_\rho(z-z_1) K^\eps_\rho(z-z_2) \6z \\
 &= -\frac12 (G^\eps_\rho *_x \tilde P^\eps_\rho) (\abs{t_1-t_2},x_1-x_2)
 + R_3^\eps(z_1,z_2)\;,
\label{eq:GKdef} 
\end{align} 
where $\tilde P^\eps_\rho = P^\eps_\rho * \varrho^\eps$. 
\end{lemma}
\begin{proof}
The first two terms in~\eqref{eq:GKdef} are equal, as can be seen by a change 
of variables $z\mapsto-z$. Using~\eqref{eq:CK} and setting $s=t_1+t_2-2t$, we 
obtain that 
\begin{align}
\int K^\eps_\rho(z_1-z) &K^\eps_\rho(z_2-z) \6z  \\
&= \int K^\eps_\rho(t_1-t, x_1-x) K^\eps_\rho(t_2-t,x_2-x) 
\indexfct{t<t_1\wedge t_2} \6t\6x + R_{3,1}^\eps(z_1,z_2) \\
&= \int K^\eps_\rho(t_1+t_2-2t, x_1-x_2) \indexfct{t<t_1\wedge t_2} \6t  + 
R_{3,2}^\eps(z_1,z_2)\\
&= \frac12 \int_{\abs{t_1-t_2}}^\infty K^\eps_\rho(s,x_1-x_2) \6s + 
R_{3,2}^\eps(z_1,z_2)
\end{align}
for some uniformly bounded remainders $R_{3,1}^\eps$ and $R_{3,2}^\eps$. 
The result follows from~\eqref{eq:Green}. 
\end{proof}

We represent~\eqref{eq:GKdef} symbolically, for $\eps=0$ and $\eps\neq0$, by 
\begin{align}
\FDarrowsout{z_1}{z_2} &= \FDarrowsin{z_1}{z_2} = -\frac12 \FDwave{z_1}{z_2}\;, 
\\
\FDarrowsouteps{z_1}{z_2} &= \FDarrowsineps{z_1}{z_2} = -\frac12 
\FDwaveeps{z_1}{z_2}\;,
\label{eq:graphs_GK} 
\end{align} 
where we do not put arrows on edges representing kernels that are symmetric in 
both variables, and discard terms bounded uniformly in $\eps$. 

\begin{example}
\label{ex:four_leaves} 
Applying Lemma~\ref{lem:KtoG} to~\eqref{eq:pairing03}, and using the 
fact that the root, marked by the green vertex, can be moved to a different 
node by a linear change of variables in the integral, we obtain 
\begin{equation}
 E(\RStreeFour) = -\frac14 \FDThreewaveeps\;.
\end{equation} 
Here and below, we will sometimes make a slight abuse of notation, by 
identifying a Feynman diagram $\Gamma$ with its value $E(\Gamma)$.
A similar computation yields 
\begin{equation}
 E(\RScombhThree) 
 = 2 \raisebox{-7.5mm}{
 \tikz{
 \path[use as bounding box] (-1.1, -0.1) rectangle (1.1,2.1);
 \node[rootnode] (0) at (0,0) {};
 \node[blacknode] (x) at (-0.5,0.8) {};
 \node[blacknode] (y) at (-1,1.6) {};
 \node[blacknode] (z) at (0,1.6) {};
 \node[blacknode] (a) at (1,1.6) {};
 \draw[Kedge] (x) -- (0);
 \draw[Kedge] (y) -- (x);
 \draw[Kepsedge] (z) -- (y);
 \draw[Kepsedge] (z) -- (x);
 \draw[Kepsedge] (a) -- (0);
 \draw[Kepsedge] (a) edge [out=140,in=40] (y);
 \drawbox;
 }
 }
 = \dfrac12\;
 \FDFourEdges{Kedge}{GKepsedge}{GKepsedge}{Kedge}
\end{equation} 
Moving the root and introducing the new kernel 
\begin{equation}
\label{eq:graphs_Q} 
\FDQ{0}{z} {}={} \FDKQeps{0}{z}\;,
\end{equation} 
we obtain 
\begin{equation}
\label{eq:E_comb_fourleaves} 
E(\RScombhThree) {}={} \frac12 \;
 \FDQKwaveeps\;. 
\end{equation}
Proceeding in the same way, we obtain for instance 
\begin{equation}
 E(\RScombhSix) 
 = \frac18
 \raisebox{-8mm}{
\tikz{
 \path[use as bounding box] (-0.8, -0.9) rectangle (0.8,0.9);
 \node[rootnode] (1) at (-0.6,-0.6) {};
 \node[blacknode] (2) at (0.6,-0.6) {};
 \node[blacknode] (3) at (-0.6,0.6) {};
 \node[blacknode] (4) at (0.6,0.6) {};
 \draw[] (3) edge [Kedge,out=-120,in=120] (1);
 \draw[] (4) edge [Kedge,out=-60,in=60] (2);
 \draw[] (3) edge [GKepsedge,out=30,in=150] (4);
 \draw[] (3) edge [GKepsedge,out=-30,in=-150] (4);
 \draw[] (1) edge [GKepsedge,out=30,in=150] (2);
 \draw[] (1) edge [GKedge,out=-30,in=-150] (2);
 \drawbox;
 }
}
 + \frac14
 \raisebox{-8mm}{
\tikz{
 \path[use as bounding box] (-0.8, -0.9) rectangle (0.8,0.9);
 \node[rootnode] (1) at (-0.6,-0.6) {};
 \node[blacknode] (2) at (0.6,-0.6) {};
 \node[blacknode] (3) at (-0.6,0.6) {};
 \node[blacknode] (4) at (0.6,0.6) {};
 \draw[] (3) edge [Kedge,out=-120,in=120] (1);
 \draw[] (4) edge [Kedge,out=-60,in=60] (2);
 \draw[] (3) edge [GKepsedge,out=-60,in=60] (1);
 \draw[] (4) edge [GKepsedge,out=-120,in=120] (2);
 \draw[] (1) edge [GKedge,out=-30,in=-150] (2);
 \draw[] (3) edge [GKepsedge,out=30,in=150] (4);
 \drawbox;
 }
}
 + \frac14
 \raisebox{-8mm}{
\tikz{
 \path[use as bounding box] (-0.8, -0.9) rectangle (0.8,0.9);
 \node[rootnode] (1) at (-0.6,-0.6) {};
 \node[blacknode] (2) at (0.6,-0.6) {};
 \node[blacknode] (3) at (-0.6,0.6) {};
 \node[blacknode] (4) at (0.6,0.6) {};
 \draw[] (3) edge [Kedge,out=-120,in=120] (1);
 \draw[] (4) edge [Kedge,out=-60,in=60] (2);
 \draw[] (1) edge [GKedge,out=-30,in=-150] (2);
 \draw[] (3) edge [GKepsedge,out=30,in=150] (4);
 \draw[GKepsedge] (3) -- (2);
 \draw[GKepsedge] (4) -- (1);
 \drawbox;
 }
}
\label{eq:E_combhSix} 
\end{equation} 
and
\begin{equation}
E(\RScombhFive)
= -\frac14 \left[
\raisebox{-8mm}{
\tikz{
 \path[use as bounding box] (-0.8, -0.9) rectangle (0.8,0.9);
 \node[rootnode] (1) at (-0.6,-0.6) {};
 \node[blacknode] (2) at (0.6,-0.6) {};
 \node[blacknode] (3) at (-0.6,0.6) {};
 \node[blacknode] (4) at (0.6,0.6) {};
 \draw[] (3) edge [Kedge,out=-120,in=120] (1);
 \draw[] (2) edge [Kedge,out=60,in=-60] (4);
 \draw[->] (4) edge [GKKedge,out=150,in=30] (3);
 \draw[] (3) edge [GKepsedge,out=-30,in=-150] (4);
 \draw[] (2) edge [GKepsedge,out=150,in=30] (1);
 \draw[] (1) edge [Kedge,out=-30,in=-150] (2);
 \drawbox;
 }
}
+
\raisebox{-8mm}{
\tikz{
 \path[use as bounding box] (-0.8, -0.9) rectangle (0.8,0.9);
 \node[rootnode] (1) at (-0.6,-0.6) {};
 \node[blacknode] (2) at (0.6,-0.6) {};
 \node[blacknode] (3) at (-0.6,0.6) {};
 \node[blacknode] (4) at (0.6,0.6) {};
 \draw[] (3) edge [Kedge,out=-120,in=120] (1);
 \draw[] (2) edge [Kedge,out=60,in=-60] (4);
 \draw[] (1) edge [GKepsedge,out=60,in=-60] (3);
 \draw[] (4) edge [GKepsedge,out=-120,in=120] (2);
 \draw[->] (4) edge [GKKedge,out=150,in=30] (3);
 \draw[] (1) edge [Kedge,out=-30,in=-150] (2);
 \drawbox;
 }
}
+
\raisebox{-8mm}{
\tikz{
 \path[use as bounding box] (-0.8, -0.9) rectangle (0.8,0.9);
 \node[rootnode] (1) at (-0.6,-0.6) {};
 \node[blacknode] (2) at (0.6,-0.6) {};
 \node[blacknode] (3) at (-0.6,0.6) {};
 \node[blacknode] (4) at (0.6,0.6) {};
 \draw[] (3) edge [Kedge,out=-120,in=120] (1);
 \draw[] (2) edge [Kedge,out=60,in=-60] (4);
 \draw[->] (4) edge [GKKedge,out=150,in=30] (3);
 \draw[] (1) edge [Kedge,out=-30,in=-150] (2);
 \draw[GKepsedge] (3) -- (2);
 \draw[GKepsedge] (1) -- (4);
 \drawbox;
 }
}
+
\raisebox{-8mm}{
\tikz{
 \path[use as bounding box] (-0.8, -0.9) rectangle (0.8,0.9);
 \node[rootnode] (1) at (-0.6,-0.6) {};
 \node[blacknode] (2) at (0.6,-0.6) {};
 \node[blacknode] (3) at (-0.6,0.6) {};
 \node[blacknode] (4) at (0.6,0.6) {};
 \draw[] (3) edge [Kedge,out=-120,in=120] (1);
 \draw[] (2) edge [Kedge,out=60,in=-60] (4);
 \draw[->] (4) edge [GKKedge,out=-120,in=120] (2);
 \draw[] (3) edge [GKepsedge,out=30,in=150] (4);
 \draw[] (3) edge [GKepsedge,out=-60,in=60] (1);
 \draw[] (1) edge [Kedge,out=-30,in=-150] (2);
 \drawbox;
 }
}
+
\raisebox{-8mm}{
\tikz{
 \path[use as bounding box] (-0.8, -0.9) rectangle (0.8,0.9);
 \node[rootnode] (1) at (-0.6,-0.6) {};
 \node[blacknode] (2) at (0.6,-0.6) {};
 \node[blacknode] (3) at (-0.6,0.6) {};
 \node[blacknode] (4) at (0.6,0.6) {};
 \draw[] (3) edge [Kedge,out=-120,in=120] (1);
 \draw[] (2) edge [Kedge,out=60,in=-60] (4);
 \draw[] (4) edge [GKepsedge,out=150,in=30] (3);
 \draw[] (1) edge [Kedge,out=-30,in=-150] (2);
 \draw[GKepsedge] (3) -- (2);
 \draw[GKKedge,->] (4) -- (1);
 \drawbox;
 }
}
\right]\;.
\end{equation} 
The corresponding pairings are
\begin{equation}
\tikz{%
\path[use as bounding box] (-0.5, -0.05) rectangle (0.5,1.1);
\FDtreeRegSix
\draw[pairing] (3) edge [distance=5,out=90,in=90] (4);
\draw[pairing] (2) edge [distance=7,out=90,in=90] (5);
\draw[pairing] (1) edge [distance=15,out=90,in=90] (6);
\drawbox;
} \qquad 
\tikz{%
\path[use as bounding box] (-0.5, -0.05) rectangle (0.5,1.1);
\FDtreeRegSix
\draw[pairing] (2) edge [distance=5,out=60,in=60] (3);
\draw[pairing] (4) edge [distance=5,out=120,in=120] (5);
\draw[pairing] (1) edge [distance=15,out=90,in=90] (6);
\drawbox;
} \qquad 
\tikz{%
\path[use as bounding box] (-0.5, -0.05) rectangle (0.5,1.1);
\FDtreeRegSix
\draw[pairing] (2) edge [distance=5,out=60,in=110] (4);
\draw[pairing] (3) edge [distance=5,out=70,in=120] (5);
\draw[pairing] (1) edge [distance=15,out=90,in=90] (6);
\drawbox;
}
\end{equation}
for the first three diagrams, and
\begin{equation}
\tikz{%
\path[use as bounding box] (-0.625, -0.05) rectangle (0.3,1.4);
\FDtreeCombSix
\draw[pairing] (1) edge [distance=25,out=60,in=60] (6);
\draw[pairing] (2) edge [distance=12,out=60,in=60] (5);
\draw[pairing] (3) edge [distance=5,out=60,in=60] (4);
\drawbox;
} \qquad 
\tikz{%
\path[use as bounding box] (-0.625, -0.05) rectangle (0.3,1.4);
\FDtreeCombSix
\draw[pairing] (1) edge [distance=25,out=60,in=60] (6);
\draw[pairing] (2) edge [distance=5,out=60,in=60] (3);
\draw[pairing] (4) edge [distance=5,out=60,in=60] (5);
\drawbox;
} \qquad 
\tikz{%
\path[use as bounding box] (-0.625, -0.05) rectangle (0.3,1.4);
\FDtreeCombSix
\draw[pairing] (1) edge [distance=25,out=60,in=60] (6);
\draw[pairing] (2) edge [distance=10,out=60,in=60] (4);
\draw[pairing] (3) edge [distance=10,out=60,in=60] (5);
\drawbox;
} \qquad 
\tikz{%
\path[use as bounding box] (-0.625, -0.05) rectangle (0.3,1.4);
\FDtreeCombSix
\draw[pairing] (1) edge [distance=25,out=60,in=60] (5);
\draw[pairing] (2) edge [distance=5,out=60,in=60] (3);
\draw[pairing] (4) edge [distance=12,out=60,in=60] (6);
\drawbox;
} \qquad 
\tikz{%
\path[use as bounding box] (-0.625, -0.05) rectangle (0.3,1.4);
\FDtreeCombSix
\draw[pairing] (1) edge [distance=25,out=60,in=60] (5);
\draw[pairing] (2) edge [distance=10,out=60,in=60] (4);
\draw[pairing] (3) edge [distance=15,out=60,in=60] (6);
\drawbox;
} 
\end{equation}
for the last five. 
\end{example}

Definition~\ref{def:feynman_diagram} can be applied to this setting, by 
expanding the set of types $\mathfrak{L}$ by $3$ new elements, with degrees 
\begin{align}
 \deg({\tikz{
 \path[use as bounding box] (0, -0.1) rectangle (1.2,0.1); 
 \draw[GKedge] (0,0) -- (1.2,0);
 \drawbox;
 }}) 
 = \deg({\tikz{
 \path[use as bounding box] (0, -0.1) rectangle (1.2,0.1); 
 \draw[GKepsedge] (0,0) -- (1.2,0);
 \drawbox;
 }})
 &= \rho - d \\
 \deg({\tikz{
 \path[use as bounding box] (0, -0.1) rectangle (1.2,0.1); 
 \draw[GKKedge,->] (0,0) -- (1.2,0);
 \drawbox;
 }}) 
 &= 2\rho - d\;.
 \label{eq:edge_degree2} 
\end{align}
The associated kernels are $G_\rho *_x \tilde P_\rho, G^\eps_\rho *_x \tilde 
P^\eps_\rho$ and $K_\rho * G^\eps_\rho *_x \tilde P^\eps_\rho$. We will say 
that a Feynman diagram is \emph{reduced} if the reduction 
rules~\eqref{eq:graphs_GK} and~\eqref{eq:graphs_Q} have been applied. Then 
Proposition~\ref{prop:degE_Gamma} extends as follows.  

\begin{prop}
\label{prop:degE_Gamma2} 
Let $\tau\in T \setminus I_E$. If $\tau$ has $p$ leaves and $q$ edges, then each 
reduced $\Gamma(\tau,P)$ has $q-p$ vertices. The 
relation~\eqref{eq:deg_Gamma_tau} still holds in this case, 
while~\eqref{eq:E_Gamma_tau} becomes
\begin{equation}
\label{eq:E_Gamma_tau2} 
 E(\tau) = \sum_{P\in\cP_\tau^{(2)}} \biggpar{-\frac12}^{1+\frac{p}{2}} 
E(\Gamma(\tau,P)) + \Order{1}\;,
\end{equation} 
where $\Order{1}$ denotes a constant uniform in $\eps$. 
\end{prop}
\begin{proof}
Recall that the unreduced Feynman diagram has $q$ edges of type $K_\rho$ or 
$K^\eps_\rho$, and $q+1-p+\frac{p}{2}$ vertices. Since $\tau$ cannot be a 
planted tree, the number of reductions is equal to $1+\frac{p}{2}$, each 
decreasing by 
$1$ the number of edges and vertices, which is why each reduced 
$\Gamma(\tau,P)$ has $q-p$ vertices. The degree is conserved 
by the reductions.  The relation~\eqref{eq:E_Gamma_tau2} is then a direct 
consequence of Lemma~\ref{lem:KtoG} and~\eqref{eq:graphs_GK}.
\end{proof}


\section{Forests}
\label{sec:forests}


\subsection{Zimmermann's forest formula}
\label{ssec:forest} 

The aim of this and the following section is to derive upper bounds 
for the expectations $E(\tilde \cA_- \tau)$ when $\tau \in T_-$. We want to 
prove that
\begin{equation}
\label{eq:bound_EAtau} 
 |E(\tilde \cA_- \tau)| \leqs C f(\tau) \varepsilon^{\abss{\tau}}\;,
\end{equation}
where the constant $C$ does not depend on $\tau$, $\eps$ or $\rho$, and 
$f(\tau)$ is a function to be determined, which depends on the structure of the 
tree $\tau$. 

A nice feature is that one can define a twisted antipode $\tilde\cA_-$ acting 
on Feynman diagrams of negative degree, which is essentially the same as 
in~\cite{Hairer_BPHZ}, and reduces in this case to a mere 
extraction/contraction of divergent subdiagrams. In the sequel, we will use the same notation for this antipode as the one on trees. From the context, it will be clear which one is used.
Denote by $\sG$ the vector space spanned by all admissible Feynman 
diagrams (not necessarily connected), and by $\sG_-$ the subspace 
spanned by diagrams of negative degree. We say that $\Gamma' = (\sV',\sE')$ is 
a subgraph of $\Gamma = (\sV,\sE)$ if $\sE'\subset\sE$, and $\sV'$ contains all 
vertices in $\sV$ which belong to at least one edge $e\in\sE'$. Then we define 
the twisted antipode to be the map $\tilde \cA_-:\sG_- \to {\mathscr 
G}$ given by 
\begin{equation}
\label{eq:antipode_Gamma} 
 \tilde \cA_-\Gamma = -\Gamma 
 - \sum_{\bar\Gamma\varsubsetneq\Gamma} \tilde \cA_-\bar\Gamma \cdot 
\Gamma/\bar\Gamma\;,
\end{equation} 
where the sum runs over all not necessarily connected subgraphs of negative 
degree, and $\Gamma/\bar\Gamma$ denotes the graph obtained by contracting 
$\bar\Gamma$ to a single vertex. 

\begin{remark}
The name twisted antipode is again related to the fact that one can introduce a 
Hopf algebra structure on (decorated) graphs, 
see~\cite[Section~2.3]{Hairer_BPHZ}, which generalises the 
extraction-contraction Hopf algebra on undecorated graphs introduced by Connes 
and Kreimer in~\cite{CK1,CK2}. The twisted antipode differs from the antipode 
of that Hopf algebra because of the use of a coaction instead of a coproduct, 
meaning that the extracted graphs $\bar\Gamma$ and the contracted graphs 
$\Gamma/\bar\Gamma$ are not in the same space: while the former have negative 
degree, the latter can have arbitrary degree.
\end{remark}

\begin{prop}
\label{prop:Etau_Feynman} 
One has 
\begin{equation}
 E(\tilde \cA_- \tau) = \sum_{P \in \mathcal{P}^{(2)}_{\tau}}  E(\tilde \cA_- 
\Gamma(\tau,P))\;.
\end{equation}
\end{prop}
\begin{proof} 
It follows from Propositions~\ref{prop:ceps} and~\ref{prop:AtildeE_minus} that 
\begin{equation}
 E(\tilde \cA_- \tau) = 
 -E(\tau) - \sum_{ \funit \varsubsetneq \tau_1 \cdot \ldots \cdot \tau_n 
\varsubsetneq_{E} \tau} 
E(\tilde \cA_- (\tau_1 \cdot \ldots \cdot \tau_n) \cdot \tau / (\tau_1 
\cdot \ldots \cdot \tau_n))\;.
\end{equation} 
We then apply Proposition~\ref{prop:degE_Gamma} to the 
expectations on the right-hand side and by an inductive argument, we get
\begin{equation}
	\begin{aligned}
 E(\tilde \cA_- \tau) ={}& 
 - \sum_{P\in\cP_\tau^{(2)}} E(\Gamma(\tau,P))  \\ 
 &{}- \sum_{ \funit \varsubsetneq 
\tau_1 \cdot \ldots \cdot \tau_n 
\varsubsetneq_{E} \tau} \prod_{i=1}^{n} \sum_{P_i\in\cP_{\tau_i}^{(2)}} E(  
\tilde \cA_- \Gamma(\tau_i,P_i))
\sum_{P_{n+1}\in\cP_{\tau_{n+1}}^{(2)}} E(   \Gamma(\tau_{n+1},P_{n+1}))
\end{aligned}
\end{equation}
where $\tau_{n+1} = \tau / (\tau_1 \cdot \ldots \cdot \tau_n)$. Indeed, one has 
\begin{equation}
	\cP_\tau^{(2)} = \bigcup_{\funit \varsubsetneq 
  \tau_1 \cdot \ldots \cdot \tau_n 
  \varsubsetneq_{E} \tau}  \left \lbrace  \bigsqcup_{i=1}^{n+1} P_i \, : \, P_i \in \cP_{\tau_i}^{(2)} \right \rbrace  \bigsqcup 	\hat{\cP}_{\tau}^{(2)}  
\end{equation}
where $ \hat{\cP}_{\tau}^{(2)} $ contains the pairings without any subdiagrams that could be extracted via $ \tilde{\mathcal{A}}_-$.
Moreover, any subdiagram of 
$\Gamma(\tau,P) $ is of the form $\Gamma(\bar \tau, \bar P)$ where $\bar \tau $ 
is a subtree of $\tau$ and $\bar P$ is a subpairing of $P$.
\end{proof}

\begin{example}
Consider the symbol $\tau = \RScombhSix$. The effect of the twisted antipode on 
$\tau$ has been determined in Example~\ref{ex:Atilde_comb}, and $E(\tau)$ is 
given in~\eqref{eq:E_combhSix}. Applying the twisted antipode directly 
to~\eqref{eq:E_combhSix}, we find 
\begin{equation}
 E(\tilde\cA_-(\RScombhSix)) = - E(\RScombhSix) 
 + \frac12 \FDTwoEdges{GKepsedge}{Kedge} 
 \FDFourEdges{Kedge}{GKepsedge}{GKepsedge}{GKedge}
 - \frac14 \Bigl(\FDTwoEdges{GKepsedge}{Kedge}\Bigr)^2
 \FDTwoEdges{GKepsedge}{GKedge}\;.
\label{eq:tildeA_Example} 
\end{equation} 
Indeed, one easily checks that since $\rho > \rhocrit = d/3$, the only nontrival 
subgraph of negative degree in~\eqref{eq:E_combhSix} is the \lq\lq bubble\rq\rq\ 
having two edges, one of type $K_\rho$ and one of type $\smash{G^\eps_\rho *_x 
\tilde P^\eps_\rho}$.  The expression~\eqref{eq:tildeA_Example} is indeed 
equivalent to the one obtained by transforming the 
expression~\eqref{eq:tildeAE_Example} for $\tilde \cA^E_-(\tau)$ into Feynman 
diagrams. 

Note that the degree of all diagrams in~\eqref{eq:E_combhSix} is $7\rho-3d$, 
while the total degree of the two extracted diagrams 
in~\eqref{eq:tildeA_Example} is $2(2\rho-d) < 7\rho-3d$. This is an instance of 
the degree of subdivergences being worse than the degree of the whole diagram.  
\end{example}

\begin{remark}
\label{rem:deg_gamma} 
If $\gamma$ is any (non-reduced) diagram with $n+1$ vertices and $q$ edges, 
then its degree can be written as 
\begin{equation}
\label{eq:deg_general_gamma} 
 \deg(\gamma) = (\rho+d)n - qd 
 = (4n-3q) \frac{d}{3} + n (\rho-\rhocrit)\;.
\end{equation} 
In particular, if $\gamma$ is of the form $\Gamma(\tau,P)$, one has 
\begin{equation}
 \deg(\gamma) = -\frac23 d + \frac{3m-1}{2} (\rho-\rhocrit)\;, 
 \qquad 
 \deg(\gamma) = -\frac13 d + \frac{3\bar m+1}{2} (\rho-\rhocrit)\;, 
\end{equation} 
respectively, for \full\ and \afull\ binary trees, where $m$, $\bar m$ are 
such that $\tau$ has $2m$ edges in the first case, and $2\bar m+1$ edges in the 
second case. Note that in both cases, the degree is a strictly increasing 
function of the number of edges. 

For practical counting of degrees, it is sometimes useful to consider the 
limiting case $\rho\searrow\rhocrit$, and to use $\frac{d}{3}$ as degree unit. 
Then edges of the three types in~\eqref{eq:edge_degree1} 
and~\eqref{eq:edge_degree2} count for $-3$, $-2$ and $-1$ respectively, while 
vertices have weight $+4$. Similarly, for trees $\tau\in\cT_-$, edges have 
weight $+1$ and leaves have weight $-2$. 
\end{remark}

Proposition~\ref{prop:Etau_Feynman} allows to reduce the estimation of the 
coefficients $c_\eps(\tau)$ to the problem of estimating the value of Feynman 
diagrams. The difficulty is that the twisted antipode is essential to obtain a 
bound of the form~\eqref{eq:bound_EAtau}: such a bound is not true in general 
for $E(\tau)$, because, as the above example shows, Feynman diagrams 
$\Gamma(\tau,P)$ may contain subdiagrams whose degree is strictly less than the 
degree of $\Gamma(\tau,P)$. In order to deal with this difficulty, our plan is 
now to adapt the approach of~\cite{Hairer_BPHZ} to the present situation.  While 
we will use its formalism, the main novelty is an adaptation of the proof 
of~\cite[Thm.~3.1]{Hairer_BPHZ} in order to derive $\varepsilon$-dependent 
bounds for Feynman diagrams given in Proposition~\ref{prop:sum_n_Hepp} below. 
This proposition can be considered as one of the main results of this work, as 
the bound it provides is new and was not proved in \cite{Hairer_BPHZ}.

\begin{definition}[Forests]
\label{def:forest} 
Let $\Gamma$ be a Feynman diagram, and denote by 
$\sG^-_\Gamma$ the set of all connected subgraphs 
$\bar\Gamma\subset\Gamma$ of negative degree. We denote by $<$ the partial 
order on $\sG^-_\Gamma$ defined by inclusion. A subset $\sF 
\subset \sG^-_\Gamma$ is called a \emph{forest} if any two elements of 
$\sF$ are either comparable by $<$, or vertex-disjoint. The 
set of forests on $\Gamma$ is denoted by $\sF^-_\Gamma$.
Given a forest $\sF$ and two graphs $\bar\Gamma,\bar\Gamma_1\in 
\sF$, we say that $\bar\Gamma_1$ is a \emph{child} of\/ $\bar\Gamma$ if\/ 
$\bar\Gamma_1<\bar\Gamma$, and there is no $\bar\Gamma_2\in\sF$ such 
that $\bar\Gamma_1<\bar\Gamma_2<\bar\Gamma$. In that case, $\bar\Gamma$ is 
called the \emph{parent} of\/ $\bar\Gamma_1$. 
\end{definition}

\begin{example}
\label{ex:forests} 
Let $\tau$ be the comb with eight leaves, and consider the following pairings:
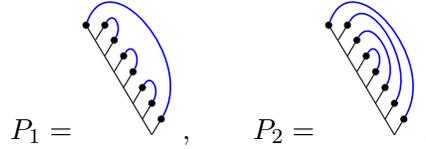
\begin{equation}
P_1 = 
 \tikz{%
\path[use as bounding box] (-0.95, -0.05) rectangle (0.3,1.9);
\FDtreeCombEight
\draw[pairing] (1) edge [distance=30,out=70,in=60] (8);
\draw[pairing] (2) edge [distance=5,out=60,in=60] (3);
\draw[pairing] (4) edge [distance=5,out=60,in=60] (5);
\draw[pairing] (6) edge [distance=5,out=60,in=60] (7);
\drawbox;
}
\;, 
\qquad
P_2 = 
 \tikz{%
\path[use as bounding box] (-0.95, -0.05) rectangle (0.3,1.9);
\FDtreeCombEight
\draw[pairing] (1) edge [distance=35,out=70,in=60] (8);
\draw[pairing] (2) edge [distance=22,out=60,in=60] (7);
\draw[pairing] (3) edge [distance=14,out=60,in=60] (6);
\draw[pairing] (4) edge [distance=5,out=60,in=60] (5);
\drawbox;
}\;.
\end{equation}
The corresponding Feynman diagrams are given by 
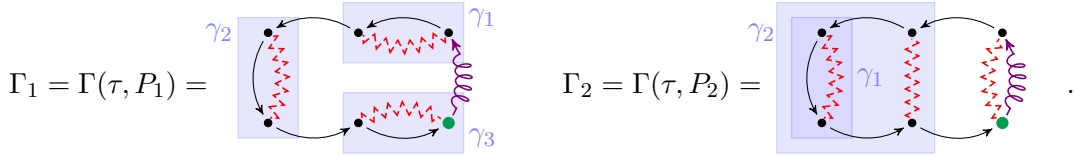
\begin{equation}
\Gamma_1 = \Gamma(\tau,P_1) = 
\raisebox{-8mm}{
\tikz{
 \path[use as bounding box] (-1.2, -1) rectangle (2.45,1);
 \draw[blue!20,fill=blue!10] (-1,-0.8) rectangle (-0.2,0.8);
 \draw[blue!20,fill=blue!10] (0.4,-1) rectangle (2,-0.2);
 \draw[blue!20,fill=blue!10] (0.4,1) rectangle (2,0.2);
 \node[blacknode] (1) at (-0.6,-0.6) {};
 \node[blacknode] (2) at (0.6,-0.6) {};
 \node[rootnode] (3) at (1.8,-0.6) {};
 \node[blacknode] (4) at (1.8,0.6) {};
 \node[blacknode] (5) at (0.6,0.6) {};
 \node[blacknode] (6) at (-0.6,0.6) {};
 \draw[] (6) edge [Kedge,out=-120,in=120] (1);
 \draw[] (6) edge [GKepsedge,out=-60,in=60] (1);
 \draw[] (1) edge [Kedge,out=-30,in=-150] (2);
 \draw[] (2) edge [Kedge,out=-30,in=-150] (3);
 \draw[] (2) edge [GKepsedge,out=30,in=150] (3);
 \draw[] (3) edge [GKKedge,out=60,in=-60,->] (4);
 \draw[] (4) edge [Kedge,out=150,in=30] (5);
 \draw[] (4) edge [GKepsedge,out=-150,in=-30] (5);
 \draw[] (5) edge [Kedge,out=150,in=30] (6);
 \node[blue!50] at (2.25,0.8) {$\gamma_1$};
 \node[blue!50] at (-1.25,0.6) {$\gamma_2$};
 \node[blue!50] at (2.25,-0.8) {$\gamma_3$};
 \drawbox;
 }
}
\qquad
\Gamma_2 = \Gamma(\tau,P_2) = 
\raisebox{-8mm}{
\tikz{
 \path[use as bounding box] (-1.2, -1) rectangle (2.45,1);
 \draw[blue!20,fill=blue!10] (-1.2,-1) rectangle (0.9,1);
 \draw[blue!30,fill=blue!15] (-1,-0.8) rectangle (-0.2,0.8);
 \node[blacknode] (1) at (-0.6,-0.6) {};
 \node[blacknode] (2) at (0.6,-0.6) {};
 \node[rootnode] (3) at (1.8,-0.6) {};
 \node[blacknode] (4) at (1.8,0.6) {};
 \node[blacknode] (5) at (0.6,0.6) {};
 \node[blacknode] (6) at (-0.6,0.6) {};
 \draw[] (6) edge [Kedge,out=-120,in=120] (1);
 \draw[] (6) edge [GKepsedge,out=-60,in=60] (1);
 \draw[] (1) edge [Kedge,out=-30,in=-150] (2);
 \draw[] (2) edge [Kedge,out=-30,in=-150] (3);
 \draw[] (3) edge [GKKedge,out=60,in=-60,->] (4);
 \draw[] (4) edge [Kedge,out=150,in=30] (5);
 \draw[] (5) edge [Kedge,out=150,in=30] (6);
 \draw[] (3) edge [GKepsedge,out=120,in=-120] (4);
 \draw[GKepsedge] (2) -- (5); 
 \node[blue!50] at (0.05,0) {$\gamma_1$};
 \node[blue!50] at (-1.4,0.6) {$\gamma_2$};
 \drawbox;
 }
}\;.
\label{eq:combs_diagrams} 
\end{equation} 
The diagram $\Gamma_1$ has $3$ identical divergent bubbles $\gamma_1, 
\gamma_2,\gamma_3$, indicated by shaded frames. The left-hand bubble $\gamma_2$ 
is part of two overlapping subdivergences, each consisting of two bubbles and 
the joining edge. However, these subdiagrams are $1$-connected, and thus do not 
matter in the analysis. If we restrict our attention to the set ${\mathscr 
G}^-_{\Gamma_1,E}$ of subgraphs with non-vanishing expectation, we obtain indeed 
a forest $\smash{\sG^-_{\Gamma_1,E}} = 
\set{\Gamma_1,\gamma_1,\gamma_2,\gamma_3,\varnothing}$. The corresponding 
parent-child relationship graph consists of the parent $\Gamma_1$ and its three 
children $\gamma_1,\gamma_2,\gamma_3$. 

The diagram $\Gamma_2$ has two nested subdivergences: a bubble $\gamma_1$, and 
the bubble together with the $3$ adjacent edges, denoted $\gamma_2$. In this 
case again, the set ${\mathscr G}^-_{\Gamma_2,E}$ is a forest, while the 
associated graph is a linear graph with parent $\Gamma$, child $\gamma_2$ and 
grandchild $\gamma_1$. 
\end{example}

In what follows, we will occasionally need decorated Feynman diagrams 
$\bar \Gamma^{\Labn}_{\Labe}$, though as in the case of trees, decorations will 
play almost no role. Such a diagram is defined by a graph $\Gamma=(\sV,\sE)$ 
with a distinguished node $v^\star\in\sV$, a node decoration 
$\Labn:\sV\to\N_0^{d+1}$ and a vertex decoration $\Labe:\sE\to\N_0^{d+1}$. The 
degree of $\bar \Gamma^{\Labn}_{\Labe}$ is defined as 
\begin{equation}
\label{eq:deg_Gamma_decorated} 
 \deg(\bar \Gamma^{\Labn}_{\Labe}) = (\rho+d) (\abs{\sV}-1) 
 + \sum_{v\in\sV} \abss{\Labn(v)}
 + \sum_{e\in\sE} \bigbrak{\deg(e) - \abss{\Labe(e)}}\;, 
\end{equation} 
and its value is given by 
\begin{equation}
\label{eq:value_Gamma_decorated} 
 E(\bar \Gamma^{\Labn}_{\Labe}) = 
 \int_{(\R^{d+1})^{\sV\setminus v^\star}} \prod_{e\in\sE} 
\partial^{\Labe(e)}K_{\mathfrak{t}(e)} (z_{e_+}-z_{e_-}) 
\prod_{w\in\sV\setminus v^\star} (z_w - z_{v^\star})^{\Labn(w)}
\6z\;. 
\end{equation} 
Note that when the decorations $\Labn$ and $\Labe$ vanish identically, 
\eqref{eq:deg_Gamma_decorated} and~\eqref{eq:value_Gamma_decorated} reduce to 
the expressions~\eqref{eq:deg_Gamma} and~\eqref{eq:EGamma} for undecorated 
Feynman diagrams. 
Given a divergent subdiagram $\gamma\in\sG^-_\Gamma$, we define an 
extraction-contraction operator $\sC_\gamma$ by 
\begin{equation}
\label{eq:def_Cgamma} 
 \sC_\gamma \bar \Gamma^{\Labn}_{\Labe}  = 
 \sum_{\Labe_\gamma,\Labn_\gamma} 
\indicator{\deg(\gamma^{\Labn_\gamma+\pi\Labe_\gamma}_{\Labe}) < 0} 
\frac{(-1)^{| \out \, \Labe_\gamma |}}{\Labe_\gamma!}
\binom{\Labn}{\Labn_\gamma}
 \gamma^{\Labn_\gamma+\pi\Labe_\gamma}_{\Labe} 
 \cdot \mathcal{R}_\gamma \bar \Gamma^{\Labn - \Labn_\gamma}_{ \Labe + 
\Labe_\gamma} \;,
\end{equation}
where $\pi\Labe_\gamma$ and $\mathcal{R}_\gamma$ are defined in the same way  as 
for decorated trees in \eqref{eq:co-action_minus}, and $|\out \, \Labe_\gamma |$ 
is the number of derivatives on outgoing edges from $\gamma$. This operator can 
be naturally extended to undecorated diagrams $\Gamma$, by identifying them with 
$\bar\Gamma^{\Labn}_{\Labe}$ with $\Labn=0$ and $\Labe=0$. Note that in that 
case, the sum over $\Labn_\gamma$ disappears in~\eqref{eq:def_Cgamma}. The main 
difference with the case of trees is that $\Labe_{\gamma}$ has a different 
support: it is supported on the edges $(x,y)$ such that either $x$ or $y$ 
belongs to the vertex set $\sV(\gamma)$. Therefore, one gets a minus sign for 
each derivative on outgoing edges. In the case of a tree, by contrast, 
$\Labe_\gamma$ is supported only on the incoming edges. However, using this 
representation does not make any difference. Indeed, by taking $v_{\star}$ to be 
the root of the underlying tree behind the construction of $\gamma$, one obtains 
a vanishing contribution whenever one puts a monomial at $v_{\star}$ and a 
derivative on the only outgoing edge at $v_{\star}$. 

We can now define a forest extraction operator $\sC_\sF$ recursively by 
setting $\sC_\varnothing\Gamma=\Gamma$ and 
\begin{equation}
 \sC_\sF\Gamma 
 = \sC_{\sF \setminus \varrho(\sF)}
 \prod_{\gamma\in\varrho(\sF)}
 \sC_\gamma\Gamma\;, 
\end{equation} 
where $\varrho(\sF)$ denotes the set of \emph{roots} of $\gamma$ in 
the graph of parent-child relationships. Then \emph{Zimmermann's forest 
formula} states that 
\begin{equation}
\label{eq:Zimmermann_forest} 
 \tilde\cA_-\Gamma = 
 - \sum_{\sF\in\sF^-_\Gamma} (-1)^{\abs{\sF}} 
 \sC_\sF\Gamma\;,
\end{equation} 
cf.~\cite[Prop.~3.3]{Hairer_BPHZ}. In the particular case where ${\mathscr 
G}^-_\Gamma$ is itself a forest, \eqref{eq:Zimmermann_forest} can be rewritten 
as 
\begin{equation}
 \tilde\cA_-\Gamma = - \sR_{\sG^-_\Gamma} \Gamma\;,
\end{equation} 
where $\sR$ is defined recursively by ${\mathscr 
R}_\varnothing\Gamma=\Gamma$ and 
\begin{equation}
\label{eq:Zimmermann_forest_simple} 
 \sR_\sF\Gamma 
 = \sR_{\sF \setminus \varrho(\sF)}
 \prod_{\gamma\in\varrho(\sF)}
 (\id - \sC_\gamma)\Gamma\;, 
\end{equation} 
which turns out to be simpler to handle than~\eqref{eq:Zimmermann_forest}. 
This is a consequence of the \lq\lq inclusion--exclusion identity\rq\rq\
\begin{equation}
 \prod_{i\in A} (\id-X_i) = \sum_{B\subset A} (-1)^B \prod_{j\in B} X_j
\end{equation} 
valid for any finite set $A$, and operators $\setsuch{X_i}{i\in A}$, 
cf.~\cite[(3.3)]{Hairer_BPHZ}. 
In general, however, $\sG^-_\Gamma$ is \emph{not} a forest, so 
that~\eqref{eq:Zimmermann_forest_simple} does not hold. This is the problem 
of overlapping subdivergences: a divergent subgraph $\bar\Gamma\subset\Gamma$ 
can be part of two different divergent subgraphs $\bar\Gamma_1$ and 
$\bar\Gamma_2$, none of which is included in the other one. 

The above example suggests that in our case, $\sG^-_{\Gamma,E}$ may 
always be a forest, so that~\eqref{eq:Zimmermann_forest_simple} is applicable. 
In order to establish this fact, we define a grafting operation on trees. If 
$\tau_1$ and $\tau_2$ are two non-planted trees, with $\tau_1$ being 
\afull, we denote by $\tau_1 \curvearrowleft \tau_2$ the tree obtained by 
joining the root of $\tau_2$ to the vertex of $\tau_1$ of degree $2$ which is 
not the root. For instance, we have 
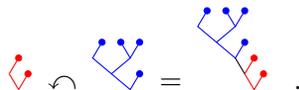
\begin{equation}
\RScombTwocol{red} \; \curvearrowleft \; \RScombThreeAcol{blue} 
\;=\; \RScombFiveAcol{blue}{red}\;.
\end{equation} 
Note that this operation is associative, but not commutative. 

The following observation allows to characterise divergent subgraphs. 

\begin{lemma}
\label{lem:subdiagrams} 
Let $\tau$ be a \full\ binary tree with an even number of leaves. Then there 
exists a pairing $P$ such that $\Gamma(\tau,P)$ is at least $2$-connected, and a 
divergent subdiagram $\bar\Gamma = \Gamma(\bar\tau,\bar P)\varsubsetneq\Gamma$, 
if and only if $\bar\tau$ is an \afull\ binary tree of negative degree, 
having an even number of leaves, and which does not contain the root of $\tau$. 
\end{lemma}
\begin{proof}
Assume first that $\bar\tau$ is an \afull\ binary tree of negative degree, 
not containing the root and with an even number of leaves. Let $\bar P$ be any 
pairing of the leaves of $\bar\tau$ and $\bar\Gamma=\Gamma(\bar\tau,\bar P)$. 
Then $\tau = \tau_0 \curvearrowleft \bar\tau \curvearrowleft 
\tau_1$, where $\tau_0$ is \afull\ and $\tau_1$ is \full. 
By pairing at least one leaf of $\tau_0$ and one leaf of $\tau_1$, we obtain a 
$2$-connected diagram $\Gamma$. 

Conversely, assume $\Gamma(\tau,P)$ is at least $2$-connected, with a divergent 
subdiagram $\bar\Gamma = \Gamma(\bar\tau,\bar P)$. Then $\bar\tau$ cannot 
contain the root of $\tau$. Indeed, if this were the case, $\bar\tau$ would  
necessarily be an \afull\ binary tree (being divergent and a proper subtree 
of $\tau$), so that $\bar\tau$ and $\tau_1=\tau\setminus\bar\tau$ would be 
connected by a single edge. Since $P$ cannot connect leaves of $\bar\tau$ to 
leaves of $\tau_1$, $\Gamma$ would be $1$-connected. Similarly, if 
$\tau = \tau_0 \curvearrowleft \bar\tau$, 
we would obtain a $1$-connected diagram. Thus $\tau$ has to be of the form 
$\tau = \tau_0 \curvearrowleft \bar\tau \curvearrowleft \tau_1$, showing that 
$\bar\tau$ is \afull\ and does not contain the root of $\tau$. 
\end{proof}

\begin{example}
Some examples of subtrees $\bar\tau$ leading to divergent subdiagrams are
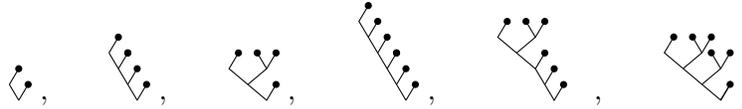
\begin{equation}
 \RScombTwo\;, \qquad 
 \RScombFour\;, \qquad 
 \RScombThreeA\;, \qquad 
 \RScombSix\;, \qquad
 \RScombFiveA\;, \qquad 
 \RScombFiveB\;.
\end{equation} 
One can check that they do not lead to any overlapping subdivergences. 
\end{example}

\begin{prop}
Assume a Feynman diagram $\Gamma(\tau,P)$ has two 
overlapping subdivergences $\Gamma_1$ and $\Gamma_2$. Then $\Gamma_1$ 
and $\Gamma_2$ are 1-connected.
As a consequence,  $\smash{{\mathscr G}^-_{\Gamma,E}}$ is always a forest. 
\end{prop}
\begin{proof}
Assume
there exist $3$ subdivergences $\bar\Gamma$, 
$\Gamma_1$, $\Gamma_2$, such that $\Gamma_1\setminus\Gamma_2$ and 
$\Gamma_2\setminus\Gamma_1$ are both non-empty and  
$\bar\Gamma\subset\Gamma_1\cap\Gamma_2$. Then there exist subtrees $\bar\tau$, 
$\tau_1$, $\tau_2$ such that $\tau_1\setminus\tau_2\neq\varnothing$, 
$\tau_2\setminus\tau_1\neq\varnothing$, 
$\bar\tau\subset\tau_1\cap\tau_2$ and each diagram 
is obtained by restricting the pairing $P$, e.g.\ $\bar\Gamma = 
\Gamma(\bar\tau,P\restr\bar\tau)$. In particular, $P$ can only pair leaves of 
$\bar\tau$. 

The previous lemma shows that we must have 
\begin{equation}
 \tau_1 = \tau_{1,-} \curvearrowleft \bar \tau \curvearrowleft \tau_{1,+} 
 \qquad\text{and}\qquad 
 \tau_2 = \tau_{2,-} \curvearrowleft \bar \tau \curvearrowleft \tau_{2,+}\;. 
\end{equation} 
Since $\tau_1\setminus\tau_2\neq\varnothing$ and 
$\tau_2\setminus\tau_1\neq\varnothing$, we may assume without 
restricting the generality that $\tau_{2,-} \subsetneq \tau_{1,-}$ and 
$\tau_{1,+} \subsetneq \tau_{2,+}$. Since the leaves of 
$\tau_{1,-}\setminus\tau_{2,-}$ cannot be paired with those of $\tau_2$, they 
have to be paired among themselves. 
But this results in $\Gamma_1$ and $\Gamma_2$ being 1-connected, 
by definition of the grafting operation. Therefore, they do not belong to 
$\smash{{\mathscr G}^-_{\Gamma,E}}$ by Lemma~\ref{lem:Gamma_1-connected}.
\end{proof}

\begin{remark}
\label{rem:deg_gamma2} 
Another consequence of Lemma~\ref{lem:subdiagrams} is that a divergent 
subdiagram $\gamma\subsetneq\Gamma$ has a degree strictly larger than $-\frac 
d3$. Therefore, in dimension $d\leqs 3$, the operator $\sC_\gamma \Gamma$ 
defined in~\eqref{eq:def_Cgamma} reduces to a simple extraction-contraction, 
while in dimension $d\in\set{4,5}$, the sum also contains terms 
$\gamma^{\pi\Labe_\gamma}$ with edge decorations $\Labe$ of degree at most $1$. 
However, the value~\eqref{eq:value_Gamma_decorated} of these additional 
terms vanishes by symmetry.  
\end{remark}


\subsection{Hepp sectors and forest intervals}
\label{ssec:Hepp} 

In this section, we present the main tools and definitions for renormalising 
Feynman diagrams: Hepp sectors, safe and unsafe forests, and forest intervals. 
All these notions have originally been introduced in the physics literature, 
see for instance~\cite[Chapter~II.3]{rivasseau} for an overview. We follow 
mainly~\cite{Hairer_BPHZ}, where these notions have been reformulated in 
connection with~\cite{BrunedHairerZambotti,ChandraHairer16}. They first appear 
in the context of singular SPDEs in~\cite{ChandraHairer16}, and were imported 
from~\cite{FMRS}.  A first important concept in order to evaluate Feynman 
diagrams is the one of Hepp sector (cf.~\cite[proof of 
Prop.~2.4]{Hairer_BPHZ}). 

\begin{definition}[Hepp sector]
\label{def:Hepp_sector} 
Fix a finite set $\sV$ and a bounded set $\Lambda\subset\R^{d+1}$. With any 
point configuration $z\in\Lambda^\sV$, one can associate a binary tree $T=T(z)$, 
whose leaves are given by $\sV$, and a function $\bn=\bn(z)$ defined on the 
inner nodes of\/ $T$ and taking values in $\N_0$, with the following properties:
\begin{itemize}
\item 	$u\mapsto\bn_u$ is increasing when going from the root to the 
leaves of\/ $T$,
\item 	for any leaves $v,\bar v\in\sV$, one has 
\begin{equation}
 \norm{z_v - z_{\bar v}}_\fraks \asymp 2^{-\bn_u}\;,
\end{equation} 
where $u = v \wedge \bar v$ is the first common ancestor of $v$ and $\bar v$ in 
$T$ and $\asymp$  is a shorthand notation for
\begin{equation}
\label{eq:def_C_Heppsector} 
C^{-1} 2^{-\bn_u} \leq \norm{z_v - z_{\bar v}}_\fraks \leq C 2^{-\bn_u}\;,
\end{equation} 
where the constant $C$ only depends on the size of $\Lambda$.
\end{itemize}
Writing $\bT = (T,\bn)$ for these data, the \emph{Hepp sector} 
$D_{\bT}\subset\Lambda^\sV$ is defined as the set of configurations 
$z\in\Lambda^\sV$ for which $(T(z),\bn(z)) = \bT$. 
\end{definition}

\begin{figure}[tbh]
\begin{center}
\begin{tikzpicture}[>=stealth',main 
node/.style={draw,circle,fill=white,minimum size=1pt,inner sep=1pt},
line/.style={draw,blue!40,shorten <=0.5pt,shorten >=0.5pt}]


\draw[thin,fill=blue!10] (0,0) rectangle (6,4); 

\node[main node,semithick,blue,fill=white,
label={[xshift=-0.2cm,yshift=-0.1cm]$z_{v_1}$}] (z1) at (1,1) {};

\node[main node,semithick,blue,fill=white,
label={[xshift=-0.2cm,yshift=-0.1cm]$z_{v_2}$}] (z2) at (1.5,1.4) {};

\node[main node,semithick,blue,fill=white,
label={[xshift=-0.15cm,yshift=-0.55cm]$z_{v_3}$}] (z3) at (4,1.2) {};

\node[main node,semithick,blue,fill=white,
label={[xshift=0.35cm,yshift=-0.5cm]$z_{v_4}$}] (z4) at (4.5,0.9) {};

\node[main node,semithick,blue,fill=white,
label={[xshift=-0.2cm,yshift=-0.1cm]$z_{v_5}$}] (z5) at (4.6,2.5) {};

\draw[line] (z1) -- (z2);
\draw[line] (z3) -- (z4);
\draw[line] (z2) -- (z3);
\draw[line] (z3) -- (z5);

\node[] at (0.25,3.7) {$\Lambda$};

\end{tikzpicture}
\hspace{10mm}
\begin{tikzpicture}[>=stealth',main 
node/.style={draw,circle,fill=white,minimum size=1pt,inner sep=1pt}]


\node[main node,semithick,blue,fill=white,
label={[blue,xshift=0cm,yshift=-0.6cm]$v_1$}] (z1) at (0,-3) {};

\node[main node,semithick,blue,fill=white,
label={[blue,xshift=0cm,yshift=-0.6cm]$v_2$}] (z2) at (1,-3) {};

\node[main node,semithick,blue,fill=white,
label={[blue,xshift=0cm,yshift=-0.6cm]$v_3$}] (z3) at (2,-3) {};

\node[main node,semithick,blue,fill=white,
label={[blue,xshift=0cm,yshift=-0.6cm]$v_4$}] (z4) at (3,-3) {};

\node[main node,semithick,blue,fill=white,
label={[blue,xshift=0cm,yshift=-0.6cm]$v_5$}] (z5) at (4,-3) {};

\node[main node,semithick,ForestGreen,
label={[ForestGreen,xshift=0cm,yshift=0cm]$0$}] (O) at (2,0) {};

\node[main node,semithick,ForestGreen,
label={[ForestGreen,xshift=0.2cm,yshift=-0.1cm]$1$}] (A) at (3,-1) {};

\node[main node,semithick,ForestGreen,
label={[ForestGreen,xshift=-0.2cm,yshift=-0.1cm]$2$}] (B) at (0.5,-2) {};

\node[main node,semithick,ForestGreen,
label={[ForestGreen,xshift=-0.2cm,yshift=-0.1cm]$2$}] (C) at (2.5,-2) {};

\draw[semithick] (z1) -- (B) -- (z2);
\draw[semithick] (z3) -- (C) -- (z4);
\draw[semithick] (C) -- (A) -- (z5);
\draw[semithick] (B) -- (O) -- (A);

\node[] at (0.5,-0.5) {$\bT(z)$};

\end{tikzpicture}
\end{center}
\vspace{-3mm}
 \caption[]{A point configuration $z\in\Lambda^\sV$ with its minimal spanning 
tree (left), and the associated labelled tree $\bT = 
(T(z),\bn(z))$ (right). Here $\sV=\{v_1,v_2,v_3,v_4,v_5\}$, and node 
decorations $\bn$ are shown in green. For instance, 
$\bn_{v_1\wedge v_2}=2$, so that $z_{v_1}$ and $z_{v_2}$ are at a 
distance of order $2^{-2}$, while $\bn_{v_3\wedge v_5}=1$, so that 
$z_{v_3}$ and $z_{v_5}$ are at a distance of order $2^{-1}$.}
 \label{fig_Heppsector}
\end{figure}

The main idea is that in each Hepp sector, the kernels have a given order of 
magnitude. Since the Hepp sectors provide a partition of $\Lambda^\sV$, the 
value of the Feynman diagram can be written as a sum of integrals over 
individual Hepp sectors, so that it suffices to obtain uniform bounds on the 
products of kernels valid in each sector.

In order to exploit cancellations, it turns out to be necessary to adapt the way 
contractions are performed to the particular Hepp sector, 
cf.~\cite[Section~3.2]{Hairer_BPHZ}. If $\Gamma$ is a 
Feynman diagram (possibly with decorations) and $\gamma$ is a 
divergent subdiagram of $\Gamma$, one defines a new diagram 
$\hat\sC_\gamma\Gamma$ as in~\eqref{eq:def_Cgamma}, but with the following 
differences. First the vertices of $\Gamma$ are given an arbitrary order, and 
its edges $e$ are assigned an additional label $\mathfrak{d}(e)=0$ indicating 
their depth. Instead of extracting the subdiagram $\gamma$, all edges of 
$\Gamma$ adjacent to $\gamma$ are reconnected to the first vertex of $\gamma$ 
(according to the chosen order), while the depth $\mathfrak{d}(e)$ of all edges 
$e$ of $\gamma$ is incremented by~$1$. Finally, when applying 
$\hat\sC_\gamma\Gamma$ to a diagram having edges of strictly positive depth, we 
set $\hat\sC_\gamma\Gamma=0$ unless all edges adjacent to $\gamma$ have a 
smaller depth than those of $\gamma$. 

\begin{example}
\label{ex:C_gamma}
Let 
\begin{equation}
 \Gamma=\Gamma_2 = 
\raisebox{-8mm}{
\tikz{
 \path[use as bounding box] (-1.2, -1) rectangle (2.45,1);
 \draw[blue!20,fill=blue!10] (-1.2,-1) rectangle (0.9,1);
 \draw[blue!30,fill=blue!15] (-1,-0.8) rectangle (-0.2,0.8);
 \node[blacknode] (1) at (-0.6,-0.6) {};
 \node[blacknode] (2) at (0.6,-0.6) {};
 \node[rootnode] (3) at (1.8,-0.6) {};
 \node[blacknode] (4) at (1.8,0.6) {};
 \node[blacknode] (5) at (0.6,0.6) {};
 \node[blacknode] (6) at (-0.6,0.6) {};
 \path[blue!50] (3) ++(0.2,-0.2)   node {\scriptsize{1}};
 \path[blue!50] (4) ++(0.2,0.2)    node {\scriptsize{2}};
 \path[blue!50] (5) ++(0,0.25)     node {\scriptsize{3}};
 \path[blue!50] (6) ++(-0.2,0.05)  node {\scriptsize{4}};
 \path[blue!50] (1) ++(-0.2,-0.05) node {\scriptsize{5}};
 \path[blue!50] (2) ++(0,-0.25)    node {\scriptsize{6}};
 \draw[] (6) edge [Kedge,out=-120,in=120] (1);
 \draw[] (6) edge [GKepsedge,out=-60,in=60] (1);
 \draw[] (1) edge [Kedge,out=-30,in=-150] (2);
 \draw[] (2) edge [Kedge,out=-30,in=-150] (3);
 \draw[] (3) edge [GKKedge,out=60,in=-60,->] (4);
 \draw[] (4) edge [Kedge,out=150,in=30] (5);
 \draw[] (5) edge [Kedge,out=150,in=30] (6);
 \draw[] (3) edge [GKepsedge,out=120,in=-120] (4);
 \draw[GKepsedge] (2) -- (5); 
 \node[blue!50] at (0.05,0) {$\gamma_1$};
 \node[blue!50] at (-1.4,0.6) {$\gamma_2$};
 \drawbox;
 }
}
\end{equation} 
be the second diagram in Example~\ref{ex:forests} (without decorations $\Labn$ 
and $\Labe$). We order the vertices counterclockwise, starting at the green 
vertex, as indicated by blue labels. Assume furthermore that $d\leqs 3$, so that 
$\sC_{\gamma}$ does not create any terms with nontrivial decoration. Then we 
have 
\begin{equation}
 \hat\sC_{\gamma_1}\Gamma
 = 
 \raisebox{-8.8mm}{
\tikz{
 \path[use as bounding box] (-2.1, -1) rectangle (2.1,1);
 \node[rootnode] (1) at (1.8,-0.6) {};
 \node[blacknode] (2) at (1.8,0.6) {};
 \node[blacknode] (3) at (0.6,0.6) {};
 \node[blacknode] (4) at (-0.6,0) {};
 \node[blacknode] (5) at (-1.8,0) {};
 \node[blacknode] (6) at (0.6,-0.6) {};
 \path[blue!50] (1) ++(0.2,-0.2) node {\scriptsize{1}};
 \path[blue!50] (2) ++(0.2,0.2)  node {\scriptsize{2}};
 \path[blue!50] (3) ++(0,0.25)   node {\scriptsize{3}};
 \path[blue!50] (4) ++(-0.1,0.3) node {\scriptsize{4}};
 \path[blue!50] (5) ++(-0.2,0)   node {\scriptsize{5}};
 \path[blue!50] (6) ++(0,-0.25)  node {\scriptsize{6}};
 \draw[] (4) edge [Kedge,out=150,in=30] node[violet,above] {\scriptsize{1}} (5);
 \draw[] (4) edge [GKepsedge,out=-150,in=-30] node[violet,below=1mm] 
{\scriptsize{1}} (5);
 \draw[] (4) edge [Kedge,out=-60,in=-180] (6);
 \draw[] (6) edge [Kedge,out=-30,in=-150] (1);
 \draw[] (1) edge [GKKedge,out=60,in=-60,->] (2);
 \draw[] (2) edge [Kedge,out=150,in=30] (3);
 \draw[] (3) edge [Kedge,out=180,in=60] (4);
 \draw[] (1) edge [GKepsedge,out=120,in=-120] (2);
 \draw[GKepsedge] (6) -- (3); 
 \drawbox;
 }
}
\;, \qquad 
 \hat\sC_{\gamma_2}\Gamma
 = 
 \raisebox{-9.8mm}{
\tikz{
 \path[use as bounding box] (-1.1, -1.1) rectangle (2.1,1.6);
 \node[rootnode] (1) at (1.8,0) {};
 \node[blacknode] (2) at (1.8,1.2) {};
 \node[blacknode] (3) at (0.6,0.6) {};
 \node[blacknode] (4) at (-0.6,0.6) {};
 \node[blacknode] (5) at (-0.6,-0.6) {};
 \node[blacknode] (6) at (0.6,-0.6) {};
 \path[blue!50] (1) ++(0.2,-0.2) node {\scriptsize{1}};
 \path[blue!50] (2) ++(0.2,0.2)  node {\scriptsize{2}};
 \path[blue!50] (3) ++(-0.1,0.3)   node {\scriptsize{3}};
 \path[blue!50] (4) ++(-0.2,0.2) node {\scriptsize{4}};
 \path[blue!50] (5) ++(-0.2,-0.2)   node {\scriptsize{5}};
 \path[blue!50] (6) ++(0.2,-0.2)  node {\scriptsize{6}};
 \draw[] (4) edge [Kedge,out=-120,in=120] node[violet,left] 
{\scriptsize{1}} (5);
 \draw[] (4) edge [GKepsedge,out=-60,in=60] node[violet,right] 
{\scriptsize{1}} (5);
 \draw[] (5) edge [Kedge,out=-30,in=-150] node[violet,below] 
{\scriptsize{1}} (6);
 \draw[] (3) edge [Kedge,out=-60,in=-180] (1);
 \draw[] (1) edge [GKKedge,out=60,in=-60,->] (2);
 \draw[] (2) edge [Kedge,out=180,in=60] (3);
 \draw[] (3) edge [Kedge,out=150,in=30] node[violet,above] {\scriptsize{1}} (4);
 \draw[] (1) edge [GKepsedge,out=120,in=-120] (2);
 \draw[GKepsedge] (6) -- node[violet,left] {\scriptsize{1}}  (3); 
 \drawbox;
 }
}\;,
\label{eq:Cgamma2} 
\end{equation} 
where  violet edge labels denote 
the depth $\mathfrak{d}(e)$ (we do note indicate zero depths). 
Extracting both subdiagrams, we obtain 
\begin{equation}
\label{eq:C_gamma12} 
 \hat\sC_{\gamma_1}\hat\sC_{\gamma_2}\Gamma
 = 
 \hat\sC_{\gamma_2}\hat\sC_{\gamma_1}\Gamma
 = \raisebox{-9mm}{
\tikz{
 \path[use as bounding box] (-2.1, -1) rectangle (2.1,1.6);
 \node[rootnode] (1) at (1.8,0) {};
 \node[blacknode] (2) at (1.8,1.2) {};
 \node[blacknode] (3) at (0.6,0.6) {};
 \node[blacknode] (4) at (-0.6,0) {};
 \node[blacknode] (5) at (-1.8,0) {};
 \node[blacknode] (6) at (0.6,-0.6) {};
 \path[blue!50] (1) ++(0.2,-0.2) node {\scriptsize{1}};
 \path[blue!50] (2) ++(0.2,0.2)  node {\scriptsize{2}};
 \path[blue!50] (3) ++(-0.2,0.25)   node {\scriptsize{3}};
 \path[blue!50] (4) ++(-0.1,0.3) node {\scriptsize{4}};
 \path[blue!50] (5) ++(-0.2,0)   node {\scriptsize{5}};
 \path[blue!50] (6) ++(0,-0.25)  node {\scriptsize{6}};
 \draw[] (4) edge [Kedge,out=150,in=30] node[violet,above] {\scriptsize{2}} (5);
 \draw[] (4) edge [GKepsedge,out=-150,in=-30] node[violet,below=1mm] 
{\scriptsize{2}} (5);
 \draw[] (4) edge [Kedge,out=-60,in=-180] node[violet,below] 
{\scriptsize{1}} (6);
 \draw[] (3) edge [Kedge,out=-60,in=-180] (1);
 \draw[] (1) edge [GKKedge,out=60,in=-60,->] (2);
 \draw[] (2) edge [Kedge,out=180,in=60] (3);
 \draw[] (3) edge [Kedge,out=180,in=60] node[violet,above] 
{\scriptsize{1}} (4);
 \draw[] (1) edge [GKepsedge,out=120,in=-120] (2);
 \draw[GKepsedge] (6) -- node[violet,left] {\scriptsize{1}} (3); 
 \drawbox;
 }
}\;.
\qedhere
\end{equation} 
\end{example}

Note that $\hat\sC_{\gamma_1}$ and $\hat\sC_{\gamma_2}$ commute. Given a 
forest $\sF \subset \sG_\Gamma^-$, one can thus define in an unambiguous way the 
operator $\mathfrak{K}_{\sF}$ performing all contractions  
$\smash{\hat\sC_\gamma}$ with $\gamma\in\sF$. We denote by $\sigma$ the 
bijection between vertices and edges of $\mathfrak{K}_{\sF}\Gamma$ and those of 
$\Gamma$. 

We now fix a Hepp sector $D_\bT$, $\bT = (T,\bn)$ and a forest $\sF \subset 
\sG_\Gamma^-$, which we assume to be \emph{full} in the sense that all 
$\gamma\in\sF$ contain all edges of $\Gamma$ joining two vertices of $\gamma$. 
As in~\cite[Section~3.2]{Hairer_BPHZ}, we construct a partition 
$\sP_{\bT}$ of $\sG_\Gamma^-$ into subsets which are adapted to the 
particular Hepp sector. The first step is to define, for each edge $e$ of 
$\Gamma$, the common ancestor of the extremities of $e$ viewed as an element of 
$\mathfrak{K}_{\sF}\Gamma$, that is 
\begin{equation}
 v_e = \sigma(\sigma^{-1}(e)_-) \wedge \sigma(\sigma^{-1}(e)_+)\;.
\end{equation} 
Then the integer 
\begin{equation}
 \scale_{\bT}^{\sF}(e) = \bn_{v_e}
\end{equation} 
measures the distance between the extremities of $e$ in 
$\mathfrak{K}_{\sF}\Gamma$. For $\gamma\in\sF$, define 
\begin{equation}
 \intT_{\bT}^{\sF}(\gamma) 
 = \inf_{e\in\sE_\gamma^\sF} \scale_{\bT}^{\sF}(e)\;, \qquad 
 \extT_{\bT}^{\sF}(\gamma) 
 = \sup_{e\in\partial\sE_\gamma^\sF} \scale_{\bT}^{\sF}(e)\;,  
\end{equation} 
where $\sE_\gamma^\sF$ denotes the set of edges belonging to $\gamma$, but not 
to any of its children in $\sF$, while $\smash{\partial\sE_\gamma^\sF}$ denotes 
the set of edges adjacent to $\gamma$ belonging to its parent $\sA(\gamma)$ in 
$\sF$. If $\gamma$ is a root of $\sF$, we set $\sA(\gamma)=\Gamma$. Thus 
$\intT_{\bT}^{\sF}(\gamma)$ describes the longest distance between points in 
$\gamma$ without its children, while $\extT_{\bT}^{\sF}(\gamma)$ describes 
the shortest distance between points in $\gamma$ and those in its parent in 
$\sF$. Examples~\ref{ex:safe_unsafe}, \ref{ex_suddivergence} and 
\ref{ex:extract_unsafe} below provide illustrations of these concepts. 

\begin{definition}[Safe and unsafe forests] \hfill 
\begin{itemize}
\item 	A subdiagram $\gamma\in\sF$ is \emph{safe} in $\sF$ if 
\begin{equation}
 \extT_{\bT}^{\sF}(\gamma) \geqs \intT_{\bT}^{\sF}(\gamma)
\end{equation} 
and \emph{unsafe} otherwise. 

\item 	A subdiagram $\gamma$ of $\Gamma$ is \emph{safe} (resp.\ \emph{unsafe}) 
for $\sF$ if $\sF\cup\set{\gamma}$ is a full forest and $\gamma$ is safe 
(resp.\ unsafe) in $\sF\cup\set{\gamma}$.

\item 	A forest $\sF$ is \emph{safe} if every $\gamma\in\sF$ is safe in 
$\sF$. 
\end{itemize}
\end{definition}

Loosely speaking, a subdiagram $\gamma$ is thus unsafe if the diameter of  
$\gamma$ (without its children) is much shorter than the distance between 
$\gamma$ and its parent. In other words, children are unsafe if they are small 
and far away from their parents. 

\begin{example}
\label{ex:safe_unsafe} 
Consider again the diagram $\Gamma$ of the previous example, with the forest 
$\sF=\set{\gamma_1,\gamma_2}$. Then for most edges $e=(e_-,e_+)$ we have 
$\scale_{\bT}^{\sF}(e) = \bn_{e_-\wedge e_+}$, except for the two cases 
\begin{equation}
 \scale_{\bT}^{\sF}((5,6)) = \bn_{4\wedge 6}\;, 
 \qquad 
 \scale_{\bT}^{\sF}((6,1)) = \bn_{3\wedge 1}\;. 
\end{equation} 
Indeed, the edges $(5,6)$ and $(6,1)$ are exactly those which are reconnected 
when applying $\mathfrak{K}_\sF$. It follows that $\gamma_1$ is safe in $\sF$ 
if and only if 
\begin{equation}
\label{eq:ex_safe1} 
 \bn_{3\wedge 4} \vee \bn_{4\wedge 6} \geqs \bn_{4\wedge 5}\;,
\end{equation} 
and one checks that this is also the condition for $\gamma_1$ to be safe in 
$\set{\gamma_1}$ (that is, for $\set{\gamma_1}$ to be a safe forest). The 
condition for $\gamma_2$ to be safe in $\sF$ reads 
\begin{equation}
\label{eq:ex_safe2} 
 \bn_{2\wedge 3} \vee \bn_{3\wedge 1} 
 \geqs \bn_{3\wedge 4} \wedge \bn_{4\wedge 6}\wedge \bn_{3\wedge 6}\;.
\end{equation} 
This time, it turns out that $\gamma_2$ is safe in the forest $\set{\gamma_2}$ 
if and only if 
\begin{equation}
\label{eq:ex_safe3} 
 \bn_{2\wedge 3} \vee \bn_{3\wedge 1} 
 \geqs \bn_{3\wedge 4} \wedge \bn_{4\wedge 5} \wedge \bn_{5\wedge 6} \wedge 
\bn_{3\wedge 6}\;,
\end{equation} 
because of the difference between $\hat\sC_{\gamma_2}$ and  
$\hat\sC_{\gamma_1}\hat\sC_{\gamma_2}$. Note, however, that the ultrametricity 
of $\bn_{\cdot\wedge\cdot}$ implies that $\bn_{4\wedge 6} \geqs \bn_{4\wedge 5} 
\wedge \bn_{5\wedge 6}$, so that if $\gamma_2$ is safe in $\sF$, then it 
is also safe in $\set{\gamma_2}$. 
\end{example}

This example shows that the property of being safe or unsafe may depend 
on the choice of forest $\sF$. A crucial property, shown 
in~\cite[Lemma~3.6]{Hairer_BPHZ}, is the following. If $\sFs$ is a safe full 
forest, and 
\begin{equation}
\label{eq:sFu} 
 \sFu = \bigsetsuch{\gamma\in\sG_\Gamma^-}{\gamma \text{ is unsafe for } 
\sFs}\;,
\end{equation} 
then $\sFs\cup\sFu \in \sF_\Gamma^-$ is a full forest, and every 
$\gamma\in\sFs$ is safe in 
$\sFs\cup\sFu$, while every $\gamma\in\sFu$ is unsafe in $\sFs\cup\sFu$.
This implies in particular that any full forest $\sF\subset\sG_\Gamma^-$ has a 
unique decomposition $\sF=\sFs\cup\sFu$, where $\sFs$ is safe and $\sFu$ is 
given by~\eqref{eq:sFu}. Moreover, the properties of being safe/unsafe and the 
construction of $\sFu$ depend only on the structure of the tree $T$, and not on 
the scale assignment $\bn$ defining $\bT$. 

The last step to construct the partition $\sP_\bT$ relies on the notion of 
forest interval, cf.~\cite[Section~3.1]{Hairer_BPHZ}. In general, forest 
intervals have two purposes: one of them is to deal with overlapping 
divergences, and the other one is to simplify the combinatorics when dealing 
with unsafe forests. In our model, we do not have overlapping divergences, but 
forest intervals are still useful to deal with unsafe forests. Moreover, they 
will allow us to obtain estimates on $c_\eps(\tau)$ that can be extended to 
cases with overlapping divergences.

\begin{definition}[Forest interval]
Let $\underline{\M}\subset\overline{\M}$ be two forests in $\sF_\Gamma^-$.
A \emph{forest interval} is a subset $\M\subset\sF_\Gamma^-$ defined by 
\begin{equation}
 \M = [\underline{\M},\overline{\M}] 
 = 
\bigsetsuch{\sF\in\sF_\Gamma^-}{\underline{\M}\subset\sF\subset\overline{\M}}\;.
\end{equation} 
Alternatively, we have 
\begin{equation}
 \M = \bigsetsuch{\underline{\M}\cup\sF}{\sF\subset\delta(\M)}\;,
\end{equation} 
where $\delta(\M)=\overline{\M}\setminus\underline{\M}$ is a forest such that 
$\delta(\M)\cap\underline{\M}=\varnothing$. 
\end{definition}

Given a Hepp sector $D_\bT$, $\bT=(T,\bn)$, we write $\Fsafe{\Gamma}(T)$ for 
the set of all safe full forests in $\Gamma$. Then we have a partition 
\begin{equation}
\label{eq:partition_FsFu} 
 \sP_\bT = \bigsetsuch{[\sFs,\sFs\cup\sFu]}{\sFs\in\Fsafe{\Gamma}(T)}\;,
\end{equation} 
where $\sFu$ is defined by~\eqref{eq:sFu}. The point of $\sP_\bT$ is that 
Zimmermann's forest formula~\eqref{eq:Zimmermann_forest_simple} can be rewritten 
as 
\begin{equation}
\label{eq:Zimmermann_hat} 
 \sR\Gamma = \sR_{\sG_\Gamma^-} \Gamma 
 = \sum_{\M_i\in\sP_\bT} \hat\sR_{\M_i}\Gamma\;,
\end{equation} 
where 
\begin{equation}
 \hat\sR_\M \Gamma = \prod_{\gamma\in\delta(\M)} (\id - \hat\sC_\gamma) 
 \prod_{\bar\gamma\in\underline\M} (-\hat\sC_{\bar\gamma}) \Gamma\;.
\end{equation}
Here, the factors $(\id - \hat\sC_\gamma)$ are interpreted as renormalising the 
subdiagrams in $\delta(\M)$, and the factors $(-\hat\sC_{\bar\gamma})$ as 
extracting those in $\underline\M$. 

\begin{example}
Continuing with the previous example, there are $4$ cases to be considered. 
\begin{enumerate}
\item 	If $\set{\gamma_1,\gamma_2}$ is a safe forest, then we have seen that 
both $\set{\gamma_1}$ and $\set{\gamma_2}$ are safe. We thus have 
\begin{equation}
 \Fsafe{\Gamma}(T) = \bigset{\varnothing,\set{\gamma_1},\set{\gamma_2}, 
\set{\gamma_1,\gamma_2}}\;,
\end{equation} 
and the corresponding partition is simply
\begin{equation}
 \sP_\bT = \bigset{[\varnothing,\varnothing], [\set{\gamma_1},\set{\gamma_1}], 
[\set{\gamma_2},\set{\gamma_2}], 
[\set{\gamma_1,\gamma_2},\set{\gamma_1,\gamma_2}]}\;,
\end{equation} 
which is in fact identical with $\Fsafe{\Gamma}(T)$. 
Thus~\eqref{eq:Zimmermann_hat} becomes
\begin{equation}
\label{eq:exGamma_1} 
 \sR\Gamma = \Gamma - \hat\sC_{\gamma_1}\Gamma - \hat\sC_{\gamma_2}\Gamma 
 + \hat\sC_{\gamma_1}\hat\sC_{\gamma_1}\Gamma\;,
\end{equation} 
which is indeed compatible with~\eqref{eq:Zimmermann_forest_simple}. 

\item 	If $\set{\gamma_1}$ is safe, but $\gamma_2$ is unsafe for 
$\set{\gamma_1}$, then $\set{\gamma_2}$ may be safe or unsafe. In the former 
case, we have  
\begin{align}
\Fsafe{\Gamma}(T) &= \bigset{\varnothing,\set{\gamma_1},\set{\gamma_2}}\;, \\
\sP_\bT &= \bigset{[\varnothing,\varnothing], 
[\set{\gamma_1},\set{\gamma_1,\gamma_2}], 
[\set{\gamma_2},\set{\gamma_2}]}\;, 
\\
\sR\Gamma & = \Gamma - (\id - \hat \sC_{\gamma_2}) \hat \sC_{\gamma_1} \Gamma -
\hat \sC_{\gamma_2}\Gamma
\label{eq:exGamma_2a} 
\end{align}
while in the latter case, 
\begin{align}
\Fsafe{\Gamma}(T) &= \bigset{\varnothing,\set{\gamma_1}}\;, \\
\sP_\bT &= \bigset{
[\set{\gamma_1},\set{\gamma_1,\gamma_2}], 
[\varnothing,\set{\gamma_2}]}\;,\\
\sR\Gamma &= -(\id-\hat\sC_{\gamma_2})\hat\sC_{\gamma_1}\Gamma + 
(\id-\hat\sC_{\gamma_2})\Gamma\;. 
\label{eq:exGamma_2b} 
\end{align}
Naturally, the expressions~\eqref{eq:exGamma_2a} and~\eqref{eq:exGamma_2b} are 
equivalent to~\eqref{eq:exGamma_1}, but the point is that the terms in each 
expression can be controlled individually. 

\item 	If $\set{\gamma_2}$ is safe, but $\gamma_1$ is unsafe for 
$\set{\gamma_2}$, then $\set{\gamma_1}$ is unsafe. Hence 
\begin{align}
\Fsafe{\Gamma}(T) &= \bigset{\varnothing,\set{\gamma_2}}\;, \\
\sP_\bT &= \bigset{[\varnothing,\set{\gamma_1}], 
[\set{\gamma_2},\set{\gamma_1,\gamma_2}]}\;,\\
\sR\Gamma &= (\id-\hat\sC_{\gamma_1})\Gamma - 
(\id-\hat\sC_{\gamma_1})\hat\sC_{\gamma_2}\Gamma\;. 
\end{align}

\item 	Finally, if both $\set{\gamma_1}$ and $\set{\gamma_2}$ are unsafe, then 
\begin{align}
\Fsafe{\Gamma}(T) &= \bigset{\varnothing}\;, \\
\sP_\bT &= \bigset{[\varnothing,\set{\gamma_1,\gamma_2}]}\;,\\
\sR\Gamma &= (\id-\hat\sC_{\gamma_1}) 
(\id-\hat\sC_{\gamma_2})\Gamma\;. \qedhere
\end{align}
\end{enumerate}
\end{example}


\section{Bounds on $E(\tilde\cA_-^{E}\tau)$}
\label{sec:bounds} 

Combining Zimmermann's forest formula~\eqref{eq:Zimmermann_hat}, our 
choice~\eqref{eq:partition_FsFu} of partition of\/ $\sF_\Gamma^-$, and the 
expression~\eqref{eq:EGamma} for the expectation of a Feynman diagram, we 
obtain (cf.~\cite[Section~3.2]{Hairer_BPHZ})
\begin{equation} 
\label{eq:Hepp_decomposition}
E(\tilde \cA_- \Gamma(\tau,P)) = -  \sum_{T} \sum_{\sFs \in 
\Fsafe{\Gamma}(T) } \sum_{\bn} \int_{D_{\bT}} 
(\sW^{K} \hat{\sR}_{[\sFs, \sFs \cup \sFu ]} 
\Gamma(\tau,P) ) (z) \6z\;, 
\end{equation}
where the sums run over all binary trees $T$ with $\abs{\sV}$ leaves, and 
all increasing node labels $\bn$ of\/ $T$. Here 
\begin{equation}
 (\sW^K \bar\Gamma^\Labn_\Labe)(z) 
 = \prod_{e\in\sE} \partial^{\Labe(e)}
 K_{\mathfrak{t}(e)} (z_{\sigma(e_+)}-z_{\sigma(e_-)})
 \prod_{w\in\sV\setminus v_\star}(z_{\sigma(w)}-z_{\sigma(v_\star)})^{\Labn(w)}
\end{equation} 
corresponds to the integrand in~\eqref{eq:value_Gamma_decorated} (recall that 
$\sigma$ is the bijection between vertices and edges of $\mathfrak{K}_\sF\Gamma$ 
and $\Gamma$), and $v_\star$ is by definition the first vertex in the 
component of $\mathfrak{K}_{\sFs}\Gamma$ containing $w$. An upper bound 
for~\eqref{eq:Hepp_decomposition} is given by 
\begin{equation} 
\label{eq:Hepp_bound}
\bigabs{E(\tilde \cA_- \Gamma(\tau,P))} 
\leqs  \sum_{T} \! \sum_{\sFs \in 
\Fsafe{\Gamma}(T) } \! \sum_{\bn} \sup_{z\in D_{\bT}} 
\bigabs{(\sW^{K} \hat{\sR}_{[\sFs, \sFs \cup \sFu ]} 
\Gamma(\tau,P)) (z)} C_0^{\abs{\sV_\Gamma}}
\prod_{v\in T}2^{-(\rho+d)\bn_v}\;.
\end{equation}
Here $\abs{\sV_\Gamma}$ is the number of vertices of the graph 
$\Gamma(\tau,P)$, $C_0$ is a constant depending only on the size of\/ 
$\Lambda$ 
through the definition~\ref{def:Hepp_sector} of Hepp sectors, and 
$\smash{C_0^{\abs{\sV_\Gamma}}}$ times the product corresponds to the volume of 
the Hepp sector $D_\bT$.
The aim of this section is to prove the following bound.

\begin{prop}
\label{prop:sum_n_Hepp} 
There exists a constant $K_1$, depending only on the kernels 
$K_{\mathfrak{t}}$, such that for any safe forest $\sFs \in 
\Fsafe{\Gamma}(T)$, 
\begin{equation}
 \sum_{\bn} \sup_{z\in D_{\bT}} 
\bigabs{(\sW^{K} \hat{\sR}_{[\sFs, \sFs \cup \sFu]}\Gamma)(z)} 
\prod_{v\in T}2^{-(\rho+d)\bn_v}
\leqs 
\begin{cases}
K_1^{\abs{\sE}} \eps^{\deg(\Gamma)} \bigbrak{\log(\eps^{-1})}^{\zeta}
& \text{if $\deg\Gamma<0$\;,} \\
K_1^{\abs{\sE}} \bigbrak{\log(\eps^{-1})}^{1+\zeta} 
& \text{if $\deg\Gamma=0$\;,}  
\end{cases}
\end{equation}
where $\zeta\in\set{0,1}$ is the number of children of\/ $\Gamma$ 
in $\sFs$ having degree $0$.
\end{prop}

The existence of the exponent $\zeta$ 
has no influence on the main result, because $\zeta=1$ occurs only 
for very few diagrams. The fact that $\zeta\in\set{0,1}$ is shown in 
Lemma~\ref{lem:number_of_children} below.

The proof of Proposition~\ref{prop:sum_n_Hepp} follows rather closely the one 
given in~\cite[Section~3.2]{Hairer_BPHZ}. There are a few differences, due to 
the facts that we work with a non-Euclidean scaling, and that the Feynman 
diagrams we consider have no legs. Owing to the special structure of the 
equations we consider, decorations of vertices and edges can be almost entirely 
avoided, they only arise in one estimate involving unstable forests 
(cf.~Section~\ref{ssec:bounds_unotempty}). 

We first need to quantify the singularity of the kernels. 
Similarly to~\cite{Hairer1,Hairer_BPHZ}, we use the notation 
\begin{equation}
 \norm{K_{\mathfrak{t}}} = \sup_{\abss{k}\leqs2} \sup_z 
 \frac{\abs{\partial^k 
K_{\mathfrak{t}}(z)}}{\norm{z}_\fraks^{\deg\mathfrak{t}-\abss{k}}}\;.
\end{equation} 
It then follows from~\cite[Lemma~10.7]{Hairer1} that there exists a constant 
$C_{\mathfrak{t}}$ such that 
\begin{equation}
 \abs{K^\eps_{\mathfrak{t}}(z)} \leqs 
 C_{\mathfrak{t}} \norm{K_{\mathfrak{t}}} 
\bigpar{\norm{z}_\fraks\vee\eps}^{\deg\mathfrak{t}}
\end{equation} 
holds uniformly in $\eps\in(0,1]$. We will write $K_0$ for the maximal value of 
$C_{\mathfrak{t}} \norm{K_{\mathfrak{t}}}$ for all kernels involved.  
Indeed $\mathfrak{t}$ runs over a finite set of types, so 
that the maximum is finite.

A difference with~\cite{Hairer_BPHZ} is that we have to deal explicitly with 
the fact that some kernels are regularised, and others are not. To indicate 
this, we attach to each edge $e\in\sE$ an additional label $\reg(e)$ with value 
$0$ if $e$ corresponds to a bare kernel, and with value $1$ if it corresponds 
to a mollified kernel, and we write 
\begin{equation}
 \sE^\circ = \setsuch{e\in\sE}{\reg(e)=0}\;, 
 \qquad 
 \sE^\eps = \setsuch{e\in\sE}{\reg(e)=1}\;.
\end{equation}


\subsection{The case $\sFu = \varnothing$}
\label{ssec:bounds_uempty} 

As in~\cite{Hairer_BPHZ}, we start by discussing the case $\sFu=\varnothing$. 
First note that according to Remark~\ref{rem:deg_gamma2}, any diagram 
with nontrivial decorations obtained by applying an operator $\hat\sC_\gamma$ 
has zero expectation. Therefore we may simply set 
\begin{equation}
 \hat{\sR}_{[\sFs, \sFs \cup \sFu]} \Gamma 
 = \prod_{\gamma\in\sFs} (-\hat\sC_\gamma)\Gamma 
 = (-1)^{\abs{\sFs}} \mathfrak{K}_{\sFs}\Gamma\;.
\end{equation} 

\begin{lemma}
\label{lem:bound_W_Fs} 
For any inner node $v\in T$, define 
\begin{align}
\eta^\circ(v) &= \rho + d 
 + \sum_{e\in\sE^\circ} \deg(e) \indicator{e^\uparrow}(v)\;, \\ 
 \eta^\eps(v) &= \sum_{e\in\sE^\eps} \deg(e) 
\indicator{e^\uparrow}(v)\;,
 \label{eq:eta_safe} 
\end{align}
where $e^\uparrow = \sigma(e_-) \wedge \sigma(e_+)$ is the last common 
ancestor of the vertices of $e$ seen as an edge of\/ $\Gamma$. 
Let 
\begin{equation}
\label{eq:def_f_v_nv} 
 f(v,\bn_v) = \eta^\circ(v)\bn_v + \eta^\eps(v) \bigbrak{\bn_v\wedge n_\eps}\;,
\end{equation} 
where $n_\eps$ is the smallest integer such that $2^{-n_\eps}\leqs\eps$. Then 
there exists a constant $\bar K_0\geqs K_0$ such that
\begin{equation}
\label{eq:bound_W_Fs} 
 \sum_{\bn} \sup_{z\in D_{\bT}} 
\bigabs{(\sW^{K} \mathfrak{K}_{\sFs}\Gamma)(z)} \prod_{v\in T}2^{-(\rho+d)\bn_v}
\leqs 
\bar K_0^{\abs{\sE}} 
C^{-\deg\Gamma}
\sum_{\bn} \prod_{v\in T}2^{-f(v,\bn_v)}\;,
\end{equation} 
where $C$ is the constant appearing in the relation~\eqref{eq:def_C_Heppsector} 
characterising Hepp sectors.
\end{lemma}
\begin{proof}
The definitions of Hepp sectors and of\/ $K_0$ imply that uniformly over 
$z\in D_\bT$, one has
\begin{equation}
 \bigabs{(\sW^{K} \mathfrak{K}_{\sFs}\Gamma)(z)} 
 \leqs 
 K_0^{\abs{\sE}}
 C^{-\sum_{e\in\sE}\deg(e)}
 \prod_{e\in\sE^\circ} 2^{-\bn(e^\uparrow)\deg(e)}
 \prod_{e\in\sE^\eps} 2^{-(\bn(e^\uparrow)\wedge n_\eps)\deg(e)}\;.
\end{equation}
Since only edges of negative degree give an unbounded contribution, the term 
$C^{-\sum_{e\in\sE}\deg(e)}$ can be bounded by $C^{-\deg\Gamma}$, enlarging if 
necessary the value of\/ $K_0$. Now it suffices to observe that  
\begin{align}
 -\log_2 \prod_{e\in\sE^\circ} 2^{-\bn(e^\uparrow)\deg(e)}
 &= \sum_{v\in T} \sum_{e\in\sE^\circ} \deg(e) \indicator{e^\uparrow}(v) 
\bn_v
 =  \sum_{v\in T} \bigbrak{\eta^\circ(v) - \rho-d} \bn_v\;, \\
 -\log_2 \prod_{e\in\sE^\eps} 2^{-(\bn(e^\uparrow)\wedge n_\eps)\deg(e)}
 &= \sum_{v\in T} 
 \sum_{e\in\sE^\eps} \deg(e) \indicator{e^\uparrow}(v) 
\bigbrak{\bn_v\wedge n_\eps}
 = \sum_{v\in T} \eta^\eps(v) \bigbrak{\bn_v\wedge n_\eps}\;.
\end{align} 
Substituting in the left-hand side of~\eqref{eq:bound_W_Fs} yields the result. 
\end{proof}

Our aim is now to bound the quantity 
\begin{equation}
\label{eq:sum_n_prod_v} 
 \sum_{\bn} \prod_{v\in T}2^{-f(v,\bn_v)}
\end{equation} 
by a recursive argument, starting from the leaves of\/ $T$. The argument is 
somewhat similar to the one given in~\cite[Lemma~A.10]{Hairer_Quastel_18}, but 
with an explicit control of the bound's dependence on the properties of the 
graph $\Gamma$. 

Given an inner node $v$ of\/ $T$, we say that $w$ is an \emph{offspring} of $v$ 
if $w>v$, and there exists no $\bar w$ with $w > \bar w > v$ (we do not use the 
term child to avoid confusion with the notion of child in $\sFs$). We denote 
the set of offspring of $v$ by $\sO(v)$. Note that since $T$ is a binary tree, 
$\sO(v)$ has at most two elements. 

For any $v\in T$ and $\bn_v\in\N_0$, we introduce the notation 
\begin{equation}
 \label{eq:SvT} 
 \cS_v(\bn_v) = 
 \sum_{\bar\bn \geqs \bn_v} \prod_{w>v} 2^{-f(w,\bar\bn_w)}\;, 
\end{equation} 
where the sum runs over all increasing node decorations $\bar\bn$ of 
$\setsuch{w}{w>v}$. We can rewrite this as  
\begin{equation}
\label{eq:iterate_S} 
 \cS_v(\bn_v) = \prod_{w_i\in\sO(v)} \widehat \cS_{w_i}(\bn_v)\;,
\end{equation} 
where  
\begin{equation}
\label{eq:iterate_Shat} 
 \widehat \cS_w(\bn_v) = 
 \sum_{\bar\bn \geqs \bn_v} \prod_{\bar w\geqs w} 
 2^{-f(\bar w,\bar\bn_{\bar w})}
 = \sum_{\bn_w \geqs \bn_v} 2^{-f(w,\bn_w)} \cS_w(\bn_w)\;. 
\end{equation} 
Then~\eqref{eq:sum_n_prod_v} is equal to $\widehat \cS_\varnothing (0)$, where 
$\varnothing$ denotes the root of\/ $T$. Our plan is now to compute the 
quantities 
$\widehat \cS_w(\bn_v)$ inductively, starting from the leaves of\/ $T$. In 
order 
to initialise the induction, we set $\widehat \cS_\ell(\bn_v) = 1$ on the 
leaves 
$\ell$ of\/ $T$. (Equivalently, one could set $\cS_v(\bn_v) = 1$ for all nodes 
$v$ of\/ $T$ with no offspring.) Then we have the following recursive bound. 

\begin{lemma}
\label{lem:recursive} 
Let $v$ be an inner node of\/ $T$. 
Assume that there exist non-negative functions $\alpha, \beta, \gamma$ and 
$\bar\alpha, \bar\beta$ such that the relation 
\begin{equation}
\label{eq:Shat_bound} 
 \widehat \cS_w(\bn_v) \lesssim 
 \begin{cases}
  2^{\alpha(w)n_\eps} 2^{-\beta(w)\bn_v} (n_\eps-\bn_v)^{\gamma(w)} 
  & \text{if $\bn_v < n_\eps$}\\
  2^{\bar\alpha(w)n_\eps} 2^{-\bar\beta(w)\bn_v}
  & \text{if $\bn_v \geqs n_\eps$}
 \end{cases}
\end{equation}
holds for all $w\in\sO(v)$. Assume furthermore that one has 
\begin{align}
\label{eq:Shat_bound_conditions1} 
 \eta^\circ(v) + \sum_{w_i\in\sO(v)} \bar\beta(w_i) &> 0\;, \\
 \alpha(w_i) - \beta(w_i) &\geqs \bar\alpha(w_i) - \bar\beta(w_i) 
 && \forall w_i\in\sO(v)\;,
\label{eq:Shat_bound_conditions2} 
\end{align}
and define 
\begin{equation}
\label{eq:def_eta} 
 \eta(v) = \eta^\circ(v) + \eta^\eps(v)\;, \qquad 
 \lambda(v) = \eta(v) + \sum_{w_i\in\sO(v)} \beta(w_i)\;.
\end{equation} 
Let $u$ be the parent of $v$, that is, the unique $u$ such that 
$v\in\sO(u)$. Then $\widehat \cS_v(\bn_u)$ satisfies the analogue of~\eqref{eq:Shat_bound}, 
with exponents given as follows: 
\begin{equation}
\label{eq:alpha_rec} 
 \alpha(v) = 
 \begin{cases}
  \displaystyle
  \sum_{w_i\in\sO(v)} \alpha(w_i) - \lambda(v)
  & \quad\text{if $\lambda(v)<0$\;,} 
\\
  \displaystyle
  \sum_{w_i\in\sO(v)} \alpha(w_i)
  & \quad\text{otherwise\;,}
 \end{cases}
\end{equation} 
while 
\begin{equation}
\label{eq:beta_rec} 
 \beta(v) = 
 \begin{cases}
  0
  & \quad\text{if $\lambda(v)\leqs 0$\;,} \\
  \lambda(v)
  & \quad\text{otherwise\;,}
 \end{cases}
\end{equation} 
and 
\begin{equation}
\label{eq:gamma_rec} 
 \gamma(v) = 
 \begin{cases}
  0 
  & \quad\text{if $\lambda(v)<0$\;,} 
\\
 \displaystyle
  \sum_{w_i\in\sO(v)} \gamma(w_i) + 1
  & \quad\text{if $\lambda(v)=0$\;,} 
\\
  \displaystyle
  \sum_{w_i\in\sO(v)} \gamma(w_i)
  & \quad\text{otherwise\;.}
 \end{cases}
\end{equation} 
Finally, we have 
\begin{align}
 \bar\alpha(v) &=  \sum_{w_i\in\sO(v)} \bar\alpha(w_i) - \eta^\eps(v)\;,\\
 \bar\beta(v) &=  \sum_{w_i\in\sO(v)} \bar\beta(w_i) + \eta^\circ(v)\;.
\label{eq:recurrence_bar} 
\end{align}
\end{lemma}
\begin{proof}
Combining~\eqref{eq:iterate_S} and~\eqref{eq:iterate_Shat}, we obtain 
\begin{equation}
 \widehat\cS_v(\bn_u) 
 = \sum_{\bn_v\geqs\bn_u} 2^{-f(v,\bn_v)}
 \prod_{w_i\in\sO(v)} \widehat\cS_{w_i}(\bn_v)\;.
\end{equation} 
Consider first the case $\bn_u\geqs n_\eps$. Using~\eqref{eq:Shat_bound} and 
the definition~\eqref{eq:def_f_v_nv} of $f(v,\bn_v)$, we get 
\begin{equation}
 \widehat\cS_v(\bn_u) 
 \lesssim 2^{\bar\alpha(v)n_\eps} \sum_{\bn_v\geqs\bn_u} 2^{-\bar\beta(v)\bn_v}
\end{equation} 
with $\bar\alpha(v)$ and $\bar\beta(v)$ given by~\eqref{eq:recurrence_bar}. 
By Condition~\eqref{eq:Shat_bound_conditions1}, one can sum the geometric 
series, yielding the claimed bound. 

For $\bn_u < n_\eps$, we decompose the sum into two parts, yielding  
\begin{equation}
 \widehat\cS_v(\bn_u) 
 \lesssim 2^{\sum_i\alpha(w_i)n_\eps}
 \sum_{\bn_v=\bn_u}^{n_\eps-1} 2^{-\lambda(v)\bn_v} (n_\eps - 
 \bn_v)^{\sum_i\gamma(w_i)}
 + 2^{\bar\alpha(v)n_\eps} \sum_{\bn_v\geqs n_\eps} 2^{-\bar\beta(v)\bn_v}\;.
\end{equation} 
The first sum can be evaluated using the bound
\begin{equation}
 \sum_{n=n_0}^{N-1} (N-n)^\gamma 2^{-\eta n} \lesssim 
 \begin{cases}
 (N-n_0)^\gamma 2^{-\eta n_0} & \text{if $\eta>0$\;,} \\
 (N-n_0)^{\gamma+1} & \text{if $\eta=0$\;,} \\
 2^{-\eta N} & \text{if $\eta<0$\;,}
 \end{cases}
\end{equation} 
valid for any $n_0<N\in\N$, $\eta\in\R$ and $\gamma >0$. The second sum 
has order $2^{(\bar\alpha(v)-\bar\beta(v))n_\eps}$, and is negligible 
thanks to Condition~\eqref{eq:Shat_bound_conditions2}. 
\end{proof}

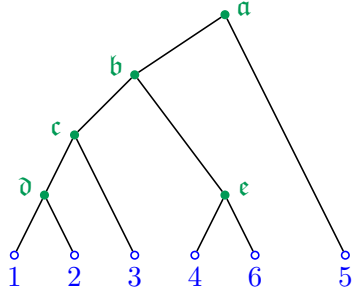
\begin{figure}[tb]
\begin{center}
\begin{tikzpicture}[>=stealth',main 
node/.style={draw,circle,fill=white,minimum size=1pt,inner sep=1pt},
xscale=0.8,yscale=0.8,baseline=-1.5cm]


\node[main node,semithick,blue,fill=white,
label={[blue,xshift=0cm,yshift=-0.6cm]$1$}] (1) at (0,-4) {};

\node[main node,semithick,blue,fill=white,
label={[blue,xshift=0cm,yshift=-0.6cm]$2$}] (2) at (1,-4) {};

\node[main node,semithick,blue,fill=white,
label={[blue,xshift=0cm,yshift=-0.6cm]$3$}] (3) at (2,-4) {};

\node[main node,semithick,blue,fill=white,
label={[blue,xshift=0cm,yshift=-0.6cm]$4$}] (4) at (3,-4) {};

\node[main node,semithick,blue,fill=white,
label={[blue,xshift=0cm,yshift=-0.6cm]$6$}] (6) at (4,-4) {};

\node[main node,semithick,blue,fill=white,
label={[blue,xshift=0cm,yshift=-0.6cm]$5$}] (5) at (5.5,-4) {};

\node[main node,semithick,ForestGreen,
label={[ForestGreen,xshift=0.25cm,yshift=-0.2cm]$\fa$}] (a) at (3.5,0) {};

\node[main node,semithick,ForestGreen,
label={[ForestGreen,xshift=-0.25cm,yshift=-0.2cm]$\fb$}] (b) at (2,-1) {};

\node[main node,semithick,ForestGreen,
label={[ForestGreen,xshift=-0.25cm,yshift=-0.2cm]$\fc$}] (c) at (1,-2) {};

\node[main node,semithick,ForestGreen,
label={[ForestGreen,xshift=-0.25cm,yshift=-0.2cm]$\fd$}] (d) at (0.5,-3) {};

\node[main node,semithick,ForestGreen,
label={[ForestGreen,xshift=0.25cm,yshift=-0.2cm]$\fe$}] (e) at (3.5,-3) {};

\draw[semithick] (3) -- (c) -- (b) -- (a) -- (5);
\draw[semithick] (c) -- (d) -- (2);
\draw[semithick] (d) -- (1);
\draw[semithick] (b) -- (e) -- (4);
\draw[semithick] (e) -- (6);

\vspace{-10mm}

\end{tikzpicture}
\hspace{10mm}
\begin{tabular}{|c|c|c|c|c|}
\hline
$e$ & $\sigma(e)$ & $e^\uparrow$ & $\deg(e)$ & $\reg(e)$ \\
\hline 
$\textcolor{violet}{(1,2)}$ & $(1,2)$ & $\fd$ & $-1$ & $1$ \\
$\textcolor{red}{(1,2)}$    & $(1,2)$ & $\fd$ & $-2$ & $1$ \\
$(2,3)$                     & $(2,3)$ & $\fc$ & $-3$ & $0$ \\
$(3,1)$                     & $(6,1)$ & $\fc$ & $-3$ & $0$ \\
$(3,4)$                     & $(3,4)$ & $\fb$ & $-3$ & $0$ \\
$(4,6)$                     & $(5,6)$ & $\fe$ & $-3$ & $0$ \\
$\textcolor{red}{(3,6)}$    & $(3,6)$ & $\fb$ & $-2$ & $1$ \\
$(4,5)$                     & $(4,5)$ & $\fa$ & $-3$ & $0$ \\
$\textcolor{red}{(4,5)}$    & $(4,5)$ & $\fa$ & $-2$ & $1$ \\
\hline
\end{tabular}
\end{center}
\vspace{-3mm}
 \caption[]{A tree $T$ defining a Hepp sector $D_\bT$ for the diagram 
$\mathfrak{K}_{\sFs}\Gamma$ in~\eqref{eq:C_gamma12}. The table shows, for each 
edge $e$, its image $\sigma(e)=(\sigma(e)_-,\sigma(e)_+)$, 
the ancestor $e^\uparrow$, the degree of $e$ measured in units of $\frac d3$ in 
the limit $\rho\searrow\rhocrit$, and the index showing whether the edge has 
been mollified. Since there can be multiple edges between two given vertices, 
they have been colour-coded according to their type.}
 \label{fig_safeforest}
\end{figure}

\begin{example} 
\label{ex:recurrence}
Consider again the diagram of Example~\ref{ex:C_gamma}, with the forest $\sF = 
\set{\gamma_1,\gamma_2}$. Consider a Hepp sector $D_\bT$ such that $T$ has the 
structure given in Figure~\ref{fig_safeforest}. The forest $\sF$ is safe 
according to~\eqref{eq:ex_safe1} and~\eqref{eq:ex_safe2}. 
\tabref{tab:recurrence} shows the values of the different exponents, computed 
iteratively starting from the leaves of the tree $T$, in the limit 
$\rho\searrow\rhocrit$. In particular, we obtain 
\begin{equation}
 \sum_{\bn} \prod_{v\in T}2^{-f(v,\bn_v)} 
 = \widehat\cS_{\fa}(0) 
 \lesssim 2^{\alpha(\fa)n_\eps}
 = 2^{2dn_\eps/3} \leqs \eps^{-2d/3}\;,
\end{equation} 
which is indeed equal to $\eps^{\deg \Gamma}$ in that limit. A similar 
computation can be made for any $\rho > \rhocrit$. 
\end{example}

\begin{table}
\begin{center}
\begin{tabular}{|c|r|r|r|r|r|r|r|r|r|r|}
\hline
\vrule width 0pt height 11pt depth 5pt
$v$ & $\eta^\circ(v)$ & $\eta^\eps(v)$ & $\eta(v)$ & $\eta_{\geqs}(v)$ & 
$\lambda(v)$ & $\alpha(v)$ & $\beta(v)$ & $\gamma(v)$ & $\bar\alpha(v)$ & 
$\bar\beta(v)$ \\
\hline 
\vrule width 0pt height 10pt depth 3pt
$\fe$ & $1$ &  $0$ &  $1$ &  $1$ &  $1$ & $0$ & $1$ & $0$ & $0$ & $1$ \\
$\fd$ & $4$ & $-3$ &  $1$ &  $1$ &  $1$ & $0$ & $1$ & $0$ & $3$ & $4$ \\
$\fc$ & $-2$ & $0$ & $-2$ & $-1$ & $-1$ & $1$ & $0$ & $0$ & $3$ & $2$ \\
$\fb$ & $1$ & $-2$ & $-1$ & $-1$ &  $0$ & $1$ & $0$ & $1$ & $5$ & $4$ \\
$\fa$ & $1$ & $-2$ & $-1$ & $-2$ & $-1$ & $2$ & $0$ & $0$ & $7$ & $5$ \\
\hline
\end{tabular}
\end{center}
\vspace{-3mm}
\caption[]{Coefficients appearing in the recursive computation described in 
Lemma~\ref{lem:recursive}, in the case of the Hepp tree $T$ given in 
\figref{fig_safeforest}. The first four exponents are defined 
in~\eqref{eq:eta_safe}, \eqref{eq:def_eta} and~\eqref{eq:def_eta_geqs}. All 
coefficients are shown in units of $\frac{d}{3}$ and in the limit 
$\rho\searrow\rhocrit$.}
\label{tab:recurrence} 
\end{table}

Let us now examine the inductive bounds in more detail. The initialisation is 
made by setting all functions $\alpha$, $\beta$, $\gamma$, $\bar\alpha$ and 
$\bar\beta$ equal to zero on the leaves of\/ $T$. Combining the recursive 
relations~\eqref{eq:alpha_rec} and~\eqref{eq:beta_rec}, we obtain 
\begin{equation}
 \alpha(v) - \beta(v) = \sum_{w_i\in\cO(v)} \bigpar{\alpha(w_i)-\beta(w_i)} - 
\eta(v)\;.
\end{equation}
Together with the initial values on the leaves, this yields 
\begin{equation}
\label{eq:def_eta_geqs} 
 \alpha(v) - \beta(v) = -\sum_{w\geqs v} \eta(w) =: -\eta_{\geqs}(v)
\end{equation} 
for all nodes $v$ of\/ $T$. In the same way, \eqref{eq:recurrence_bar} yields 
\begin{equation}
 \bar\alpha(v) - \bar\beta(v) = -\eta_{\geqs}(v)\;.
\end{equation} 
This shows in particular that Condition~\eqref{eq:Shat_bound_conditions2} is 
always satisfied. 
Regarding Condition~\eqref{eq:Shat_bound_conditions1}, we observe 
that~\eqref{eq:recurrence_bar} implies 
\begin{equation}
 \bar\beta(v) = \sum_{w\geqs v}\eta^\circ(w)
 =: \eta_{\geqs}^\circ(v)\;.
\end{equation} 
We thus have to show that $\eta_{\geqs}^\circ(v)$ is strictly positive on all 
inner vertices $v$ if\/ $T$, and to bound $\eta_{\geqs}(v)$ below in order to 
control~\eqref{eq:sum_n_prod_v}. To do this, we will import some further 
notations from~\cite{Hairer_BPHZ}. For $\gamma\in\sFs\cup\set{\Gamma}$, we 
write 
$\mathfrak{K}(\gamma) = (\sV_\gamma,\sE_\gamma)$ for the subgraph of 
$\mathfrak{K}_{\sFs}\Gamma$ with edge set $\sE_\gamma = 
\sigma^{-1}(\sE(\gamma\setminus\sC(\gamma)))$, where $\sC(\gamma)$ denotes the 
set of children of $\gamma$ in $\sFs$. Given an inner vertex $v\in T$, we let 
$\Gamma_0 = \Gamma_0(v) = (\sV_0,\sE_0)$ be the subgraph of 
$\mathfrak{K}_{\sFs}\Gamma$ containing all vertices $w\in\sV$ such that 
$\sigma(w)\geqs v$. Note that this implies 
\begin{equation}
 e \in \sE_0(v) \quad \Leftrightarrow \quad e^\uparrow \geqs v\;.
\end{equation} 
In addition, we have 
\begin{equation}
 \scale_\bT^{\sFs}(e) > \scale_\bT^{\sFs}(\bar e)\;,
\end{equation} 
and thus $e^\uparrow > \bar e^\uparrow$, for all $e\in\sE_0$ and all $\bar e$ 
adjacent to $\Gamma_0$ in $\mathfrak{K}_{\sFs}\Gamma$. 

\begin{example} 
\label{ex_suddivergence}
Continuing with Example~\ref{ex:recurrence}, we have 
\begin{equation}
 \mathfrak{K}(\gamma_1) = 
 \raisebox{-2mm}{
\tikz{
 \path[use as bounding box] (-2.1, -0.3) rectangle (-0.3,0.3);
 \node[rootnode] (4) at (-0.6,0) {};
 \node[blacknode] (5) at (-1.8,0) {};
 \path[blue!50] (4) ++(0.2,0) node {\scriptsize{4}};
 \path[blue!50] (5) ++(-0.2,0)   node {\scriptsize{5}};
 \draw[] (4) edge [Kedge,out=150,in=30] (5);
 \draw[] (4) edge [GKepsedge,out=-150,in=-30] (5);
 \drawbox;
 }
}\;, \qquad
 \mathfrak{K}(\gamma_2) = 
 \raisebox{-8mm}{
\tikz{
 \path[use as bounding box] (-0.9, -0.9) rectangle (0.9,0.9);
 \node[rootnode] (3) at (0.6,0.6) {};
 \node[blacknode] (4) at (-0.6,0) {};
 \node[blacknode] (6) at (0.6,-0.6) {};
 \path[blue!50] (3) ++(0.2,0.2)   node {\scriptsize{3}};
 \path[blue!50] (4) ++(-0.2,0) node {\scriptsize{4}};
 \path[blue!50] (6) ++(0.2,-0.2)  node {\scriptsize{6}};
 \draw[] (4) edge [Kedge,out=-60,in=-180] (6);
 \draw[] (3) edge [Kedge,out=180,in=60] (4);
 \draw[GKepsedge] (6) -- (3); 
 \drawbox;
 }
}\;, \qquad 
 \mathfrak{K}(\Gamma) = 
 \raisebox{-8mm}{
\tikz{
 \path[use as bounding box] (0.3, -0.3) rectangle (2.1,1.5);
 \node[rootnode] (1) at (1.8,0) {};
 \node[blacknode] (2) at (1.8,1.2) {};
 \node[blacknode] (3) at (0.6,0.6) {};
 \path[blue!50] (1) ++(0.2,-0.2) node {\scriptsize{1}};
 \path[blue!50] (2) ++(0.2,0.2)  node {\scriptsize{2}};
 \path[blue!50] (3) ++(-0.2,0)   node {\scriptsize{3}};
 \draw[] (3) edge [Kedge,out=-60,in=-180] (1);
 \draw[] (1) edge [GKKedge,out=60,in=-60,->] (2);
 \draw[] (2) edge [Kedge,out=180,in=60] (3);
 \draw[] (1) edge [GKepsedge,out=120,in=-120] (2);
 \drawbox;
 }
}\;.
\end{equation} 
Examples of subgraphs $\Gamma_0(v)$ are 
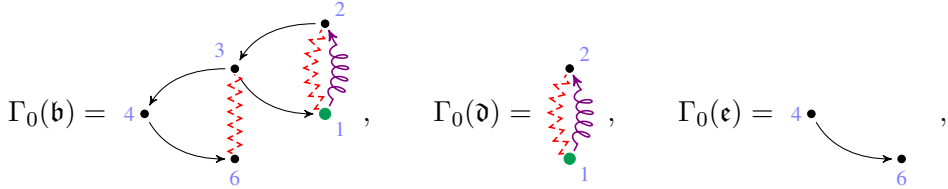
\begin{equation}
 \Gamma_0(\fb) = 
 \raisebox{-8mm}{
\tikz{
 \path[use as bounding box] (-0.9, -0.9) rectangle (2.1,1.5);
 \node[rootnode] (1) at (1.8,0) {};
 \node[blacknode] (2) at (1.8,1.2) {};
 \node[blacknode] (3) at (0.6,0.6) {};
 \node[blacknode] (4) at (-0.6,0) {};
 \node[blacknode] (6) at (0.6,-0.6) {};
 \path[blue!50] (1) ++(0.2,-0.2) node {\scriptsize{1}};
 \path[blue!50] (2) ++(0.2,0.2)  node {\scriptsize{2}};
 \path[blue!50] (3) ++(-0.2,0.25)   node {\scriptsize{3}};
 \path[blue!50] (4) ++(-0.2,0) node {\scriptsize{4}};
 \path[blue!50] (6) ++(0,-0.25)  node {\scriptsize{6}};
 \draw[] (3) edge [Kedge,out=-60,in=-180] (1);
 \draw[] (1) edge [GKKedge,out=60,in=-60,->] (2);
 \draw[] (2) edge [Kedge,out=180,in=60] (3);
 \draw[] (3) edge [Kedge,out=180,in=60] (4);
 \draw[] (1) edge [GKepsedge,out=120,in=-120] (2);
 \draw[] (4) edge [Kedge,out=-60,in=-180] (6);
 \draw[GKepsedge] (6) -- (3); 
 \drawbox;
 }
}\;, \qquad 
 \Gamma_0(\fd) = 
 \raisebox{-8mm}{
\tikz{
 \path[use as bounding box] (1.5, -0.3) rectangle (2.1,1.5);
 \node[rootnode] (1) at (1.8,0) {};
 \node[blacknode] (2) at (1.8,1.2) {};
 \path[blue!50] (1) ++(0.2,-0.2) node {\scriptsize{1}};
 \path[blue!50] (2) ++(0.2,0.2)  node {\scriptsize{2}};
 \draw[] (1) edge [GKKedge,out=60,in=-60,->] (2);
 \draw[] (1) edge [GKepsedge,out=120,in=-120] (2);
 \drawbox;
 }
}\;, \qquad 
 \Gamma_0(\fe) = 
 \raisebox{-8mm}{
\tikz{
 \path[use as bounding box] (-0.9, -0.9) rectangle (0.9,0.9);
 \node[blacknode] (4) at (-0.6,0) {};
 \node[blacknode] (6) at (0.6,-0.6) {};
 \path[blue!50] (4) ++(-0.2,0) node {\scriptsize{4}};
 \path[blue!50] (6) ++(0,-0.25)  node {\scriptsize{6}};
 \draw[] (4) edge [Kedge,out=-60,in=-180] (6);
 \drawbox;
 }
}\;,
\end{equation}
while $\Gamma_0(\fc) = \mathfrak{K}(\Gamma)$ and 
$\Gamma_0(\fa)=\mathfrak{K}_{\sFs}\Gamma$ is the diagram given 
in~\eqref{eq:C_gamma12}.
\end{example}

\begin{lemma}
\label{lem_sumeta1} 
Let $v$ be an inner vertex of\/ $T$ such that $\Gamma_0(v)$ is non-empty. 
Then the quantity $\eta_{\geqs}(v)$ satisfies the following properties~: 
\begin{enumerate}
\item 	$\eta_{\geqs}(v) = \deg(\Gamma)$ if $v=\varnothing$ is the 
root of\/ $T$, and $\eta_{\geqs}(v) \geqs \deg(\Gamma)$ otherwise;
\item 	if $v>\varnothing$, then $\eta_{\geqs}(v) = \deg(\Gamma)$ happens only 
if\/ $\Gamma$ has at least one child $\gamma\in\sC(\Gamma)$ satisfying 
$\deg(\gamma)=0$, and $\Gamma_0(v) = 
\bigcup_{\bar\gamma}\mathfrak{K}(\bar\gamma)$, where the union runs over all 
$\bar\gamma$ which are not descendents of a child with vanishing 
degree; 
\item 	if $\sO(v) = \set{w_1,w_2}$, then there exists at least one 
$i\in\set{1,2}$ such that $\eta_{\geqs}(w) > 0$ for all $w\geqs w_i$.
\end{enumerate}
If, furthermore, $\Gamma$ has at least one regularised edge, 
then there exists a constant $\kappa>0$ such that 
\begin{equation}
\label{eq:eta_geqs_circ} 
 \eta_{\geqs}^\circ(v) = \sum_{w\geqs v} \eta^\circ(w) \geqs \kappa\;.
\end{equation} 
\end{lemma}
\begin{proof}
Since $T$ is a tree, $\set{w\geqs v}$ has $\abs{\sV_0}-1$ elements, so that we 
can write 
\begin{equation}
\eta_{\geqs}(v) 
= (\rho+d) (\abs{\sV_0}-1) + \sum_{e\in\sE \cap \sE_0} \deg(e)\;.
\end{equation}
By construction, the $\mathfrak{K}(\gamma)$ have disjoint edge sets, and two 
$\mathfrak{K}(\gamma)$ can share at most one vertex. We can thus decompose 
\begin{equation}
\label{eq:a_agamma} 
 \eta_{\geqs}(v) 
 = \sum_{\gamma\in\sFs\cup\set{\Gamma}} \eta_{\geqs,\gamma}(v)\;,
\end{equation} 
where
\begin{equation}
 \eta_{\geqs,\gamma}(v) = (\rho+d) (\abs{\sV_0\cap\sV_\gamma}-1)
+ \sum_{e\in\sE_\gamma \cap \sE_0} \deg(e)\;.
\end{equation}
As in~\cite{Hairer_BPHZ}, we say that $\gamma\in\sFs\cup\set{\Gamma}$ is 
\begin{itemize}
\item 	\emph{full} if $\sE_\gamma\cap\sE_0 = \sE_\gamma$; 
\item 	\emph{empty} if $\sE_\gamma\cap\sE_0 = \varnothing$; 
\item 	\emph{normal} in all other cases. 
\end{itemize}
By~\cite[Lemma~3.7]{Hairer_BPHZ}, a full $\gamma$ cannot have an empty parent, 
and 
\begin{equation}
 \eta_{\geqs,\gamma}(v) = 
 \begin{cases}
  \deg(\gamma) - \sum_{\bar\gamma\in\sC(\gamma)} \deg(\bar\gamma)
  & \text{if $\gamma$ is full\;,} \\
  0 & \text{if $\gamma$ is empty\;,} \\
  \deg(\hat\gamma) - \sum_{\bar\gamma\in\sC_*(\gamma)} \deg(\bar\gamma)
  & \text{if $\gamma$ is normal\;,}
 \end{cases}
 \label{eq:full_empty_normal} 
\end{equation} 
where $\sC_*(\gamma)$ is the set of children $\bar\gamma$ of $\gamma$ such that 
$\mathfrak{K}(\bar\gamma)$ shares a vertex with $\Gamma_0(v)$, and 
$\hat\gamma$ 
is the subdiagram of $\Gamma$ with edge set 
$\sigma(\sE_\gamma\cap\sE_0) \cup \bigcup_{\bar\gamma\in\sC_*(\gamma)} 
\sE(\bar\gamma)$. The fact that $\gamma$ is safe implies that 
$\deg(\hat\gamma)>0$, and is also used to prove the absence of empty parent.

The result follows by considering all possibilities for the types of the 
subgraphs $\gamma$.

\begin{itemize}
\item 	A first case occurs when no $\gamma\in\sFs\cup\set{\Gamma}$ is full. 
Since $\Gamma_0$ is not empty, the $\gamma$ cannot all be empty, so that 
$\eta_{\geqs}(v) $ is 
a non-empty sum of strictly positive terms. Therefore, $\eta_{\geqs}(v) >0$. 

\item 	A second case occurs when $\Gamma$ is not full, but there exists at 
least one subgraph $\gamma\subsetneq\Gamma$ which is full. Since the parent of 
$\gamma$ is not empty, the negative term $\deg(\gamma)$ is compensated by the 
corresponding term stemming from its parent. Since $\Gamma$ is not full, there 
must exist a full subgraph $\gamma$ whose parent is normal. Since the 
inequality for normal subgraphs is strict, we have again $\eta_{\geqs}(v) >0$. 

\item 	It remains to consider the case where $\Gamma$ is full (which does not 
occur in~\cite{Hairer_BPHZ}). The case of all $\gamma\subset\Gamma$ also being 
full can only occur when $v=\varnothing$ (because only in that case is 
$\Gamma_0(v)$ equal to $\mathfrak{K}_{\sFs}(\gamma)$), and leads to the sum 
being equal to $\deg(\Gamma)$. 

Consider next the case when there is no normal subgraph. Then all subgraphs are 
full or empty. Since a full subgraph cannot have an empty parent, we obtain 
$\eta_{\geqs}(v)  = 
\deg(\Gamma) - \sum_i \deg(\gamma_i)$, where the $\gamma_i$ are all empty 
subgraphs with a full parent. This shows in particular that $\eta_{\geqs}(v) 
\geqs \deg(\Gamma)$ if $v$ is not the root. Equality can only hold when all 
$\gamma_i$ have zero degree.
These $\gamma_i$ must all be children of $\Gamma$, 
since Lemma~\ref{lem:subdiagrams} and Remark~\ref{rem:deg_gamma} imply that 
the degree of strict subdiagrams, which all arise from \afull\ trees, is 
strictly increasing in terms of their number of edges.
In addition, all $\gamma\subset\gamma_i$ are empty, so that the second property 
follows. 

The only remaining case occurs when there exists a $\gamma\subsetneq\Gamma$ 
which is normal. Then one obtains $\eta_{\geqs}(v)  \geqs \deg(\Gamma) - 
\deg(\gamma) + \deg(\hat\gamma) > \deg(\Gamma)$.

\end{itemize}

To prove the third property of $\eta_{\geqs}$, we note that since the edge sets 
$\sE_0(w_1)$ and $\sE_0(w_2)$ are disjoint, $\Gamma$ cannot be full for both 
$\Gamma_0(w_1)$ and $\Gamma_0(w_2)$. Since $\sE_0(w)\subset \sE(w_i)$ for all 
$w\geqs w_i$, there is at least one $i$ such that $\Gamma$ is not full for any 
$\Gamma_0(w)$ such that $w \geqs w_i$. Therefore, $\eta_{\geqs}(w) > 0$ for 
these $w$. 

It remains to prove \eqref{eq:eta_geqs_circ}. Here we note that 
$\eta_{\geqs}^\circ(v)$ can be written as the sum of 
$\eta_{\geqs,\gamma}^\circ(v)$, where 
\begin{equation}
 \eta_{\geqs,\gamma}^\circ(v) = (\rho+d) (\abs{\sV_0\cap\sV_\gamma}-1)
+ \sum_{e\in\sE^\circ \cap \sE_\gamma} \deg(e)\;. 
\end{equation} 
We define full, empty and normal subgraphs as above, but with $\sE_0$ replaced 
by $\sE_0\cap\sE^\circ$. Since $\Gamma$ admits at least one regularised edge, 
it cannot be full. The same argument as above thus shows that 
$\eta_{\geqs}^\circ(v)$ is strictly positive. 
\end{proof}

\begin{remark}
If follows from Lemma~\ref{lem:subdiagrams} that for any subtree 
$\bar\tau\subsetneq\tau$ of negative degree, $\tau$ has at least two leaves 
that do not belong to $\bar\tau$. As a consequence, for any divergent 
subdiagram $\gamma\subsetneq\Gamma$, $\Gamma\setminus\gamma$ admits at least 
one regularised edge. Therefore, the assumption that $\mathfrak{K}(\Gamma)$ 
admit at least one regularised edge is indeed satisfied in our situation. 
\end{remark}

\begin{example} 
\label{ex:full_normal_empty} 
We illustrate the lemma and the notions of full, normal and empty subgraphs 
used in its proof on Example~\ref{ex_suddivergence}: 
\begin{itemize}
\item 	for $\Gamma_0(\fa)$, all $\gamma\in\sFs\cup\set{\Gamma}$ are full;
\item 	for $\Gamma_0(\fb)$, $\Gamma$ and $\gamma_2$ are full, while $\gamma_1$ 
is empty;
\item 	for $\Gamma_0(\fc)$, $\Gamma$ is full, while $\gamma_2$ and $\gamma_1$ 
are empty;
\item 	for $\Gamma_0(\fd)$, $\Gamma$ is normal, and the other graphs are empty;
\item 	for $\Gamma_0(\fe)$, $\gamma_2$ is normal, and the other graphs are 
empty. 
\end{itemize}
The first three cases lead to $\eta_{\geqs}(v) \leqs 0$, since there is no 
normal subgraph. The last two cases lead to $\eta_{\geqs}(v) > 0$, since there 
is no full subgraph (compare with~\tabref{tab:recurrence}). 

Consider now the case where the Hepp tree is of the form 
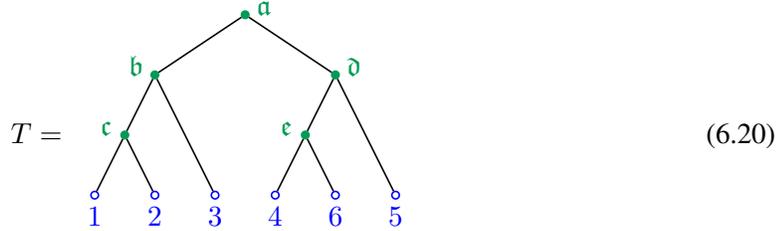
\begin{equation}
\label{eq:Hepp_ex_full_normal_empty} 
T = 
\raisebox{10mm}{
\begin{tikzpicture}[>=stealth',main 
node/.style={draw,circle,fill=white,minimum size=1pt,inner sep=1pt},
xscale=0.8,yscale=0.8,baseline=-1.5cm]


\node[main node,semithick,blue,fill=white,
label={[blue,xshift=0cm,yshift=-0.6cm]$1$}] (1) at (0,-4) {};

\node[main node,semithick,blue,fill=white,
label={[blue,xshift=0cm,yshift=-0.6cm]$2$}] (2) at (1,-4) {};

\node[main node,semithick,blue,fill=white,
label={[blue,xshift=0cm,yshift=-0.6cm]$3$}] (3) at (2,-4) {};

\node[main node,semithick,blue,fill=white,
label={[blue,xshift=0cm,yshift=-0.6cm]$4$}] (4) at (3,-4) {};

\node[main node,semithick,blue,fill=white,
label={[blue,xshift=0cm,yshift=-0.6cm]$6$}] (6) at (4,-4) {};

\node[main node,semithick,blue,fill=white,
label={[blue,xshift=0cm,yshift=-0.6cm]$5$}] (5) at (5,-4) {};

\node[main node,semithick,ForestGreen,
label={[ForestGreen,xshift=0.25cm,yshift=-0.2cm]$\fa$}] (a) at (2.5,-1) {};

\node[main node,semithick,ForestGreen,
label={[ForestGreen,xshift=-0.25cm,yshift=-0.2cm]$\fb$}] (b) at (1,-2) {};

\node[main node,semithick,ForestGreen,
label={[ForestGreen,xshift=-0.25cm,yshift=-0.2cm]$\fc$}] (c) at (0.5,-3) {};

\node[main node,semithick,ForestGreen,
label={[ForestGreen,xshift=0.25cm,yshift=-0.2cm]$\fd$}] (d) at (4,-2) {};

\node[main node,semithick,ForestGreen,
label={[ForestGreen,xshift=-0.25cm,yshift=-0.2cm]$\fe$}] (e) at (3.5,-3) {};

\draw[semithick] (1) -- (c) -- (b) -- (a) -- (d) -- (e) -- (4);
\draw[semithick] (c) -- (2);
\draw[semithick] (b) -- (3);
\draw[semithick] (d) -- (5);
\draw[semithick] (e) -- (6);

\vspace{-10mm}
\end{tikzpicture}}
\end{equation}
The forest $\sF=\set{\gamma_1,\gamma_2}$ is again safe, and we have 
in particular $\Gamma_0(\fb) = \mathfrak{K}(\Gamma)$. This shows that for 
$\Gamma_0(\fb)$, $\Gamma$ is full, and $\gamma_1$ and $\gamma_2$ are empty. If 
$\rho = \frac25d$, then $\deg(\Gamma)=\deg(\gamma_2)=0$, and we are in a 
situation where Property~2.\ of Lemma~\ref{lem_sumeta1} applies: we have 
$\eta_{\geqs}(\fa) = \eta_{\geqs}(\fb) = 0$, while $\eta_{\geqs}(v)>0$ for 
$v\in\set{\fc,\fd,\fe}$.   
\end{example}

\begin{cor}
\label{cor:Fs} 
There exists a constant $K_1$, depending only on $\bar K_0$ and $d$, such that 
\begin{equation}
\label{eq:cor_Fs} 
 \sum_{\bn} \sup_{z\in D_{\bT}} 
\bigabs{(\sW^{K} \mathfrak{K}_{\sFs}\Gamma)(z)} \prod_{v\in T}2^{-(\rho+d)\bn_v}
\leqs 
\begin{cases}
K_1^{\abs{\sE}} \eps^{\deg(\Gamma)} \bigbrak{\log(\eps^{-1})}^\zeta
& \text{if $\deg\Gamma<0$\;,} \\
K_1^{\abs{\sE}} \bigbrak{\log(\eps^{-1})}^{1+\zeta} 
& \text{if $\deg\Gamma=0$\;,}  
\end{cases}
\end{equation} 
where $\zeta$ is the number of children of\/ $\Gamma$ in $\sFs$
having degree $0$. 
\end{cor}
\begin{proof}
The lower bound~\eqref{eq:eta_geqs_circ} on $\eta_{\geqs}^\circ$ shows that 
Condition~\eqref{eq:Shat_bound_conditions1} is satisfied, so that 
Lemma~\ref{lem:recursive} applies for all inner vertices of $T$. 
Combining~\eqref{eq:def_eta_geqs} with the induction 
relations~\eqref{eq:alpha_rec} and~\eqref{eq:beta_rec}, we obtain 
\begin{equation}
\label{eq:alpha_v_rec} 
 \alpha(v) = 
 \max\biggset{-\eta_{\geqs}(v), \sum_{w_i\in\sO(v)}\alpha(w_i)}\;.
\end{equation} 
We claim that in fact, we have 
\begin{equation}
\label{eq:alpha_v} 
 \alpha(v) \leqs 
 \max\bigpar{\set{0} \cup \setsuch{-\eta_{\geqs}(w)}{w\geqs v}}\;.
\end{equation} 
This relation is clearly true if $v$ has no offspring. We now proceed by 
induction, and assume that~\eqref{eq:alpha_v} holds for all $w_i\in\sO(v)$. 
If all $\alpha(w_i)$ vanish, \eqref{eq:alpha_v} trivially holds.
Property~3. of Lemma~\ref{lem_sumeta1} implies that at most one of the 
$\alpha(w_i)$, say $\alpha(w_1)$, can be strictly positive. Indeed, if $v$ has two offspring $w_1$ and $w_2$, then the property 
implies that there is at most one offspring, say $w_2$, such that $\eta_{\geqs}(w)>0$ 
for all $w\geqs w_2$. But then $\alpha(w_2) = 0$ by the induction assumption. 
Therefore,
$\alpha(v) = \max\set{-\eta_{\geqs}(v),\alpha(w_1)}$, and~\eqref{eq:alpha_v} 
follows from the induction assumption. 

The result is then a consequence of the fact that~\eqref{eq:sum_n_prod_v} is 
bounded by 
\begin{equation}
 \widehat\cS_\varnothing(0) \lesssim 2^{\alpha(\varnothing)n_\eps}
 n_\eps^{\gamma(\varnothing)}\;.
\end{equation} 
Indeed, we have $\alpha(\varnothing) = -\deg(\Gamma)$, as a consequence 
of Property~1.\ of Lemma~\ref{lem_sumeta1}, which implies 
\begin{equation}
 \max\setsuch{-\eta_{\geqs}(w)}{w\geqs\varnothing} = 
-\eta_{\geqs}(\varnothing) = -\deg(\Gamma)\;.
\end{equation} 
Therefore~\eqref{eq:alpha_v} yields $\alpha(\varnothing) \leqs -\deg(\Gamma)$, 
but by~\eqref{eq:alpha_v_rec} this is actually an equality, 
because~\eqref{eq:alpha_v_rec} implies that 
$\alpha(\varnothing) \geqs -\deg(\Gamma)$.
Since $\deg(\Gamma)$ is bounded below by a constant depending only on $d$, the term 
$C^{-\deg\Gamma}$ in~\eqref{eq:bound_W_Fs} can be incorporated into $K_1$. 

It remains to determine $\gamma(\varnothing)$. We first note 
that~\eqref{eq:def_eta_geqs} implies 
\begin{equation}
\label{eq:eta+beta} 
 \lambda(v) = \eta_{\geqs}(v) + \sum_{w_i\in\sO(v)} \alpha(w_i)\;.
\end{equation}  
In the case $\deg(\Gamma)=0$, we have $\alpha(v)=0$, and thus 
$\lambda(v)=\eta_{\geqs}(v)$ for all $v\in T$. By Property~3.\ of 
Lemma~\ref{lem_sumeta1}, at most one of the offspring of $v$, say $w_1$, 
satisfies $\eta_{\geqs}(w_1)=0$. 
Therefore, there exists at most one path $(w_0 = \varnothing, w_1, \dots, w_\zeta)$ 
starting at the root, such that $\eta_{\geqs}(w_i) = 0$, 
and thus $\lambda(w_i) = 0$, for each $w_i$ in the path.
On such a path, \eqref{eq:gamma_rec} shows that $\gamma(v)$ increases by $1$ 
at each step. Extending this path to any leaf, and using $\gamma = 0$ as initial value 
on the leaf, we obtain that $\gamma(\varnothing) = \zeta$.  
By Property~2.\ of 
Lemma~\ref{lem_sumeta1}, each $\Gamma_0(w_i)$ is of the form 
$\bigcup_{\bar\gamma_i}\mathfrak{K}(\bar\gamma_i)$, where the $\bar \gamma_i$ 
are 
\emph{not} descendents of a given $\gamma_i\in\sC(\Gamma)$ with vanishing 
degree. Since $\Gamma_0(w_{i+1})\subsetneq\Gamma_0(w_i)$ for each $i$, $\zeta$ 
is bounded by the number of these $\gamma_i$. 

In the case $\deg(\Gamma)<0$, consider the longest sequence 
$(w_0=\varnothing,w_1,\dots,w_{\zeta'})$ such that $w_{i+1} \in \sO(w_i)$ and 
$\eta_{\geqs}(w_i)\leqs 0$ for each $i$. Then Property~3.\ of 
Lemma~\ref{lem_sumeta1} implies that $\eta_{\geqs}(v)>0$ for all other $v\in 
T$, which yields $\alpha(v)=0$, $\lambda(v)>0$ and thus $\gamma(v)=0$ for 
those $v$. For the $w_i$, we get the induction relations 
\begin{align}
\lambda(w_i) &= \eta_{\geqs}(w_i) + \alpha(w_{i+1}) \;, \\
\alpha(w_i) &= \max\bigset{-\eta_{\geqs}(w_i),\alpha(w_{i+1})} 
\geqs \alpha(w_{i+1})\;, \\
\gamma(w_i) &= 
\begin{cases}
 0 & \text{if $\lambda(w_i)<0$\;,}\\
 \gamma(w_{i+1})+1 & \text{if $\lambda(w_i)=0$\;,}\\
 \gamma(w_{i+1}) & \text{if $\lambda(w_i)>0$\;,}
\end{cases}
\end{align}
with the convention that $\alpha(w_{\zeta'+1}) = \gamma(w_{\zeta'+1}) = 0$. 
Note that we have the implications 
\begin{equation}
 \gamma(w_i) \neq 0 
 \quad\Rightarrow\quad
 \lambda(w_i) \geqs 0 
 \quad\Leftrightarrow\quad 
 \alpha(w_{i+1}) \geqs -\eta_{\geqs}(w_i)
 \quad\Leftrightarrow\quad 
 \alpha(w_i) = \alpha(w_{i+1})\;.
\end{equation} 
In addition, $\gamma$ is incremented only if $\lambda(w_i)=0$, which happens if 
and only if $\alpha(w_i) = \alpha(w_{i+1}) = -\eta_{\geqs}(w_i)$. It follows 
that $\gamma(\varnothing)$ is equal to the length $\zeta\leqs\zeta'$ of the 
longest sequence $(w_0,\dots,w_{\zeta-1})$ such that 
$\eta_{\geqs}(w_i)=\deg(\Gamma)$ for all $i$. Property~2.\ of 
Lemma~\ref{lem_sumeta1} again implies that $\zeta$ is bounded by the number of 
children of $\Gamma$ having degree $0$.
\end{proof}

\begin{example}
Consider again the Hepp sector with tree $T$ as  
in~\eqref{eq:Hepp_ex_full_normal_empty} in Example~\ref{ex:full_normal_empty}. 
As we have seen, when $\rho=\frac25$, one has $\deg(\Gamma)=\deg(\gamma_2)=0$. 
Therefore, the bound~\eqref{eq:cor_Fs} has order $(\log(\eps^{-1}))^2$. 
\end{example}

Though we find that nontrivial powers of $\log(\eps^{-1})$ can occur, the 
following result shows that in our situation, these powers cannot exceed the 
value $2$. 

\begin{lemma} \label{lem:number_of_children}
If\/ $\Gamma = \Gamma(\tau,P)$ and $\tau$ is an \afull\ binary tree, then 
$\Gamma$ cannot have any children of degree $0$, i.e., $\zeta=0$. 
If\/ $\Gamma = \Gamma(\tau,P)$ and $\tau$ is a \full\ binary tree, then 
$\Gamma$ can have at most one child of degree $0$, i.e., $\zeta\leqs1$. 
\end{lemma}
\begin{proof}
Let $\tau$ be \afull\ with $2m+1$ edges, and assume that $\Gamma$ contains a 
subdiagram $\gamma$ with $\deg(\gamma)=0$. By Lemma~\ref{lem:subdiagrams}, 
$\gamma$ is of the form $\Gamma(\bar\tau,\bar P)$ with $\bar\tau$ an \afull\ 
binary tree having $2\bar m+1 < 2m+1$ edges. By Remark~\ref{rem:deg_gamma}, we 
necessarily have $\deg(\gamma) < \deg(\Gamma)$, so that $\deg(\gamma)=0$ would 
imply $\deg(\Gamma)>0$, which is not permitted. 

If $\tau$ is \full\ with $2m$ edges, then any divergent subdiagram $\gamma$ 
results from an \afull\ tree $\bar\tau$ with $2\bar m+1 < 2m$ edges. By 
Remark~\ref{rem:deg_gamma}, the condition $\deg(\Gamma) \leqs 0 = \deg(\gamma)$ 
yields $m \leqs 2\bar m+1$. If $\Gamma$ contains $\zeta$ non-overlapping 
divergent subdiagrams of degree $0$, they must all have the same number of 
edges, and we obtain $\zeta(2\bar m+1) < 2m \leqs 2(2\bar m+1)$, yielding 
$\zeta < 2$. 
\end{proof}


\subsection{The case $\sFu \neq \varnothing$}
\label{ssec:bounds_unotempty} 

We turn now to the case $\sFu \neq \varnothing$, where we can write 
\begin{equation}
\label{eq:bounds_dsmall1} 
 \hat{\sR}_{[\sFs, \sFs \cup \sFu]} \Gamma
 = (-1)^{\abs{\sFs}} \mathfrak{K}_{\sFs}
 \prod_{\gamma\in\sFu} (\id-\hat\sC_{\gamma})\Gamma \;.
\end{equation} 
We define as before subgraphs $\mathfrak{K}(\gamma) = (\sV_\gamma,\sE_\gamma)$ 
of $\mathfrak{K}_{\sFs}$, except that $\sC(\gamma)$ now denotes the set of 
children of $\gamma$ in $\sFs\cup\set{\gamma}$. For any $\gamma\in\sFu$, we 
denote by $\gamma^\uparrow$ the inner vertex of $T$ such that 
$\sigma(\sV_\gamma) = \setsuch{v\in\sV}{v\geqs\gamma^\uparrow}$, and 
\begin{equation}
 \gamma^{\uparrow\uparrow}
 = \sup\bigsetsuch{e^\uparrow}{e\in\sE_{\sA(\gamma)} \text{ and } e \sim 
\mathfrak{K}(\gamma)}\;.
\end{equation} 
Recall that $\sA(\gamma)$ denotes the parent of $\gamma$ in $\sF$, while $\sim$ 
denotes adjacency. In other words, we are considering edges in 
$\sE_{\sA(\gamma)}$ which are not in $\sE_\gamma$. It follows that we 
necessarily have $\gamma^\uparrow > \gamma^{\uparrow\uparrow}$.
Finally, we set 
\begin{equation}
 N(\gamma) = 1 + \intpart{-\deg(\gamma)}\;.
\end{equation} 
Lemma~\ref{lem:subdiagrams} implies that all subdivergences $\gamma$ 
have a degree $\deg(\gamma) > -\frac{d}3$ (cf.~Remark~\ref{rem:deg_gamma}). 
Thus in space dimensions 
$d\leqs 3$, $N(\gamma)$ is always equal to $1$, while for $d\in\set{4,5}$ it 
can take the value $2$, and is always equal to $2$ when $\rho$ is sufficiently 
close to $\rhocrit$. In the latter case, the operator $\hat\sC_\gamma$ produces 
terms with nontrivial node labels, which are here essential for the 
renormalisation.

\begin{example}
\label{ex:extract_unsafe} 
Consider again the diagram of Example~\ref{ex:C_gamma}, with the forest $\sF = 
\set{\gamma_1,\gamma_2}$. Consider now a Hepp sector $D_\bT$ such that $T$ has 
the structure given in Figure~\ref{fig_unsafeforest}. 
In this example, $\gamma_1$ is unsafe, $\gamma_2$ is safe, and we have 
$\gamma_1^\uparrow = \fd$ and $\gamma_1^{\uparrow\uparrow} = \fb$. 

If $d\leqs 3$, the extraction operation $\hat\sC_{\gamma_1}\Gamma$ has the same 
form as in~\eqref{eq:Cgamma2} in Example~\ref{ex:C_gamma}. If $d\in\set{4,5}$, 
it becomes 
\begin{align} \label{exampleTaylor}
 \hat\sC_{\gamma_1}\Gamma
 ={}&
 \raisebox{-8.8mm}{
\tikz{
 \path[use as bounding box] (-2.1, -1) rectangle (2.1,1);
 \node[rootnode] (1) at (1.8,-0.6) {};
 \node[blacknode] (2) at (1.8,0.6) {};
 \node[blacknode] (3) at (0.6,0.6) {};
 \node[blacknode] (4) at (-0.6,0) {};
 \node[blacknode] (5) at (-1.8,0) {};
 \node[blacknode] (6) at (0.6,-0.6) {};
 \path[blue!50] (1) ++(0.2,-0.2) node {\scriptsize{1}};
 \path[blue!50] (2) ++(0.2,0.2)  node {\scriptsize{2}};
 \path[blue!50] (3) ++(0,0.25)   node {\scriptsize{3}};
 \path[blue!50] (4) ++(-0.1,0.3) node {\scriptsize{4}};
 \path[blue!50] (5) ++(-0.2,0)   node {\scriptsize{5}};
 \path[blue!50] (6) ++(0,-0.25)  node {\scriptsize{6}};
 \draw[] (4) edge [Kedge,out=150,in=30] (5);
 \draw[] (4) edge [GKepsedge,out=-150,in=-30] (5);
 \draw[] (4) edge [Kedge,out=-60,in=-180] (6);
 \draw[] (6) edge [Kedge,out=-30,in=-150] (1);
 \draw[] (1) edge [GKKedge,out=60,in=-60,->] (2);
 \draw[] (2) edge [Kedge,out=150,in=30] (3);
 \draw[] (3) edge [Kedge,out=180,in=60] (4);
 \draw[] (1) edge [GKepsedge,out=120,in=-120] (2);
 \draw[GKepsedge] (6) -- (3); 
 \drawbox;
 } 
 }
+ \sum_{i=1}^d 
 \raisebox{-8.8mm}{
\tikz{
 \path[use as bounding box] (-2.1, -1) rectangle (2.1,1);
 \node[rootnode] (1) at (1.8,-0.6) {};
 \node[blacknode] (2) at (1.8,0.6) {};
 \node[blacknode] (3) at (0.6,0.6) {};
 \node[blacknode] (4) at (-0.6,0) {};
 \node[blacknode] (5) at (-1.8,0) {};
 \node[blacknode] (6) at (0.6,-0.6) {};
 \path[blue!50] (1) ++(0.2,-0.2) node {\scriptsize{1}};
 \path[blue!50] (2) ++(0.2,0.2)  node {\scriptsize{2}};
 \path[blue!50] (3) ++(0,0.25)   node {\scriptsize{3}};
 \path[blue!50] (4) ++(-0.1,0.3) node {\scriptsize{4}};
 \path[blue!50] (5) ++(-0.2,0)   node {\scriptsize{5}};
 \path[blue!50] (6) ++(0,-0.25)  node {\scriptsize{6}};
 \path[Green] (4) ++(0.3,0)   node {\scriptsize{$e_i$}};
 \draw[] (4) edge [Kedge,out=150,in=30] (5);
 \draw[] (4) edge [GKepsedge,out=-150,in=-30] (5);
 \draw[] (4) edge [Kedge,out=-60,in=-180] (6);
 \draw[] (6) edge [Kedge,out=-30,in=-150] (1);
 \draw[] (1) edge [GKKedge,out=60,in=-60,->] (2);
 \draw[] (2) edge [Kedge,out=150,in=30] (3);
 \draw[] (3) edge [Kedge,out=180,in=60] node[Green,above] 
{\scriptsize{$e_i$}} (4);
 \draw[] (1) edge [GKepsedge,out=120,in=-120] (2);
 \draw[GKepsedge] (6) -- (3); 
 \drawbox;
 } 
 } \\
 &{} -  \sum_{i=1}^d
  \raisebox{-8.8mm}{
\tikz{
 \path[use as bounding box] (-2.1, -1) rectangle (2.1,1);
 \node[rootnode] (1) at (1.8,-0.6) {};
 \node[blacknode] (2) at (1.8,0.6) {};
 \node[blacknode] (3) at (0.6,0.6) {};
 \node[blacknode] (4) at (-0.6,0) {};
 \node[blacknode] (5) at (-1.8,0) {};
 \node[blacknode] (6) at (0.6,-0.6) {};
 \path[blue!50] (1) ++(0.2,-0.2) node {\scriptsize{1}};
 \path[blue!50] (2) ++(0.2,0.2)  node {\scriptsize{2}};
 \path[blue!50] (3) ++(0,0.25)   node {\scriptsize{3}};
 \path[blue!50] (4) ++(-0.1,0.3) node {\scriptsize{4}};
 \path[blue!50] (5) ++(-0.1,0.2)   node {\scriptsize{5}};
 \path[blue!50] (6) ++(0,-0.25)  node {\scriptsize{6}};
 \path[Green] (5) ++(-0.1,-0.2)   node {\scriptsize{$e_i$}};
 \draw[] (4) edge [Kedge,out=150,in=30] (5);
 \draw[] (4) edge [GKepsedge,out=-150,in=-30] (5);
 \draw[] (4) edge [Kedge,out=-60,in=-180] node[Green,below] 
{\scriptsize{$e_i$}} (6);
 \draw[] (6) edge [Kedge,out=-30,in=-150] (1);
 \draw[] (1) edge [GKKedge,out=60,in=-60,->] (2);
 \draw[] (2) edge [Kedge,out=150,in=30] (3);
 \draw[] (3) edge [Kedge,out=180,in=60] (4);
 \draw[] (1) edge [GKepsedge,out=120,in=-120] (2);
 \draw[GKepsedge] (6) -- (3); 
 \drawbox;
 } 
 }\;, 
 \label{eq:example_Fu} 
\end{align}
where edge and node decorations have been indicated in green ($e_i$ being the 
$i$th canonical basis vector). 
Note that this produces a factor 
\begin{equation}
 K_\rho(z_4-z_3) \Bigbrak{K_\rho(z_6-z_5) - K_\rho(z_6-z_4) 
 + \sum_{i=1}^d\partial_i K_\rho(z_6-z_4)(z_4-z_5)^{e_i}}
\end{equation} 
in the integrand giving the value of $(\id-\hat\sC_{\gamma_1})\Gamma$, since the 
terms proportional to $(z_4-z_{v_\star})^{e_i}$ stemming from the second term 
in~\eqref{eq:example_Fu} are killed because $v_\star = 4$. The point of the 
whole procedure is that the term in square brackets is bounded by a positive 
power of $\norm{z_4-z_5}_\fraks$, which is much smaller than 
$\norm{z_6-z_5}_\fraks$ owing to the fact that $\gamma_1$ is unsafe. 
\end{example}

\begin{figure}[tb]
\begin{center}
\begin{tikzpicture}[>=stealth',main 
node/.style={draw,circle,fill=white,minimum size=1pt,inner sep=1pt},
xscale=0.8,yscale=0.8,baseline=-1.5cm]


\node[main node,semithick,blue,fill=white,
label={[blue,xshift=0cm,yshift=-0.6cm]$4$}] (4) at (0,-4) {};

\node[main node,semithick,blue,fill=white,
label={[blue,xshift=0cm,yshift=-0.6cm]$5$}] (5) at (1,-4) {};

\node[main node,semithick,blue,fill=white,
label={[blue,xshift=0cm,yshift=-0.6cm]$1$}] (1) at (2,-4) {};

\node[main node,semithick,blue,fill=white,
label={[blue,xshift=0cm,yshift=-0.6cm]$2$}] (2) at (3,-4) {};

\node[main node,semithick,blue,fill=white,
label={[blue,xshift=0cm,yshift=-0.6cm]$6$}] (6) at (4,-4) {};

\node[main node,semithick,blue,fill=white,
label={[blue,xshift=0cm,yshift=-0.6cm]$3$}] (3) at (5.5,-4) {};

\node[main node,semithick,ForestGreen,
label={[ForestGreen,xshift=0.25cm,yshift=-0.2cm]$\fa$}] (a) at (3.5,0) {};

\node[main node,semithick,ForestGreen,
label={[ForestGreen,xshift=-0.7cm,yshift=-0.25cm]$\gamma_1^{\uparrow\uparrow}
 = \fb$ } ] (b) at (2.5,-1) {};

\node[main node,semithick,ForestGreen,
label={[ForestGreen,xshift=-0.25cm,yshift=-0.2cm]$\fc$}] (c) at (1.5,-2) {};

\node[main node,semithick,ForestGreen,
label={[ForestGreen,xshift=-0.7cm,yshift=-0.25cm]$\gamma_1^\uparrow = \fd$}] 
(d) at (0.5,-3) {};

\node[main node,semithick,ForestGreen,
label={[ForestGreen,xshift=0.25cm,yshift=-0.2cm]$\fe$}] (e) at (2.5,-3) {};

\draw[semithick] (4) -- (d) -- (c) -- (b) -- (a) -- (3);
\draw[semithick] (c) -- (e) -- (1);
\draw[semithick] (e) -- (2);
\draw[semithick] (d) -- (5);
\draw[semithick] (b) -- (6);

\vspace{-10mm}

\end{tikzpicture}
\hspace{10mm}
\begin{tabular}{|c|c|c|}
\hline
$e$ & $\sigma(e)$ & $e^\uparrow$ \\
\hline 
$\textcolor{violet}{(1,2)}$ & $(1,2)$ & $\fe$  \\
$\textcolor{red}{(1,2)}$    & $(1,2)$ & $\fe$  \\
$(2,3)$                     & $(2,3)$ & $\fa$  \\
$(3,1)$                     & $(6,1)$ & $\fa$  \\
$(3,4)$                     & $(3,4)$ & $\fa$  \\
$(5,6)$                     & $(5,6)$ & $\fb$  \\
$\textcolor{red}{(3,6)}$    & $(3,6)$ & $\fa$  \\
$(4,5)$                     & $(4,5)$ & $\fd$  \\
$\textcolor{red}{(4,5)}$    & $(4,5)$ & $\fd$  \\
\hline
\end{tabular}
\end{center}
\vspace{-3mm}
 \caption[]{A tree $T$ defining a Hepp sector $D_\bT$ for the diagram 
$\hat\sC_{\gamma_2}\Gamma$, cf.~\eqref{eq:Cgamma2}. The table shows, for 
each edge, its image $\sigma(e)=(\sigma(e)_-,\sigma(e)_+)$, and the ancestor 
$e^\uparrow$.}
 \label{fig_unsafeforest}
\end{figure}

\begin{lemma}
\label{lem:bound_W_Fu} 
There exists a constant $\bar K_0$ depending only on the kernels 
$K_{\mathfrak{t}}$ such that 
\begin{equation}
\label{eq:bound_W_Fu} 
 \sum_{\bn} \sup_{z\in D_{\bT}} 
\bigabs{(\sW^{K} \hat{\sR}_{[\sFs, \sFs \cup \sFu]}\Gamma)(z)} \prod_{v\in 
T}2^{-(\rho+d)\bn_v}
\leqs \bar K_0^{\abs{\sE}} C^{-\deg\Gamma}
\sum_{\bn} \prod_{v\in T}2^{-f(v,\bn_v)}\;,
\end{equation} 
where
\begin{equation}
 f(v,\bn_v) = \eta^\circ(v)\bn_v + \eta^\eps(v) \bigbrak{\bn_v\wedge n_\eps} 
+ \sum_{\gamma\in\sFu} N(\gamma) \bigbrak{\indicator{\gamma^\uparrow}(v) - 
\indicator{\gamma^{\uparrow\uparrow}}(v)}\bn_v\;,
\label{eq:eta_unsafe} 
\end{equation}
with the same $\eta^\circ(v)$ and $\eta^\eps(v)$ as in~\eqref{eq:eta_safe}. 
\end{lemma}
\begin{proof}
The difference with the proof of Lemma~\ref{lem:bound_W_Fs} is the presence of 
the factors $(\id-\hat\sC_{\gamma})$ with $\gamma$ unsafe 
in~\eqref{eq:bounds_dsmall1}. These produce a factor
\begin{equation}
\prod_{\substack{e\in\sE_{\sA(\gamma)} \\ e\sim\mathfrak{K}(\gamma)}}
\Bigbrak{ K_{\mathfrak{t}(e)} (z_{\sigma(e_+)}-z_{\sigma(e_-)}) 
 - \sum_{\abss{\ell} < N(\gamma)} 
 \frac{1}{\ell!}(z_{\sigma(e'_+)}-z_{\sigma(v_\star)})^\ell
 \partial^\ell K_{\mathfrak{t}(e)} (z_{\sigma(v_\star)}-z_{\sigma(e'_-)})}\;,
\end{equation} 
where $e'$ is the image of $e$ under $\hat\sC_\gamma$. 
By the Taylor formula-type bound given in~\cite[Lemma~3.8]{Hairer_BPHZ}, this 
factor is bounded by 
\begin{equation}
 K_12^{N(\gamma)[\bn_{\gamma^{\uparrow\uparrow}} - 
\bn_{\gamma^\uparrow}]} \prod_{\substack{e\in\sE_{\sA(\gamma)} \\ 
e\sim\mathfrak{K}(\gamma)}}
\norm{z_{\sigma(e_+)}-z_{\sigma(e_-)}}_\fraks^{\deg(e)}\;, 
\end{equation} 
which accounts for the last sum in~\eqref{eq:eta_unsafe}.
\end{proof}

Writing as before $\eta(v) = \eta^\circ(v) + \eta^\eps(v)$, we introduce the 
notations  
\begin{equation}
 \hat\eta(v) = \eta(v) + 
 \sum_{\gamma\in\sFu} N(\gamma) \bigbrak{\indicator{\gamma^\uparrow}(v) - 
\indicator{\gamma^{\uparrow\uparrow}}(v)}\;,
\qquad 
 \hat\eta_{\geqs}(v) = \sum_{w\geqs v}\hat\eta(w)\;.
\end{equation} 

\begin{lemma}
\label{lem_sumeta2} 
The conclusions of Lemma~\ref{lem_sumeta1} still hold in the present situation, 
with $\eta_{\geqs}$ replaced by $\hat\eta_{\geqs}$.
\end{lemma}
\begin{proof}
The only difference with the proof of Lemma~\ref{lem_sumeta1} is the presence 
of the sum over diagrams in $\gamma\in\sFu$. We claim that we have the 
equivalences 
\begin{align}
 v \leqs \gamma^\uparrow 
 &\qquad\Leftrightarrow\qquad 
 \mathfrak{K}(\gamma) \subset \Gamma_0(v) \cap \mathfrak{K}(\sA(\gamma))\;, \\
 v > \gamma^{\uparrow\uparrow}  
 &\qquad\Leftrightarrow\qquad 
 \Gamma_0(v) \cap \mathfrak{K}(\sA(\gamma)) \subset \mathfrak{K}(\gamma)\;. 
 \label{eq:equivalences} 
\end{align} 
Indeed, we always have $\mathfrak{K}(\gamma) \subset 
\mathfrak{K}(\sA(\gamma))$, so that the first equivalence follows from the fact 
that $v\leqs\gamma^\uparrow \Leftrightarrow \mathfrak{K}(\gamma) \subset 
\Gamma_0(v)$. For the second equivalence, we observe that if $e\in \Gamma_0(v) 
\cap \mathfrak{K}(\sA(\gamma))$, then $e^\uparrow \geqs v$, and $e$ is either 
in $\sE_\gamma$, or adjacent to $\mathfrak{K}(\gamma)$. However, the second 
case is ruled out if $v > \gamma^{\uparrow\uparrow}$. Conversely, if 
$\Gamma_0(v) \cap \mathfrak{K}(\sA(\gamma)) \subset \mathfrak{K}(\gamma)$, then 
any edge $e\sim\mathfrak{K}(\gamma)$ cannot belong to $\sE_0$, and must thus 
satisfy $e^\uparrow < v$, which implies that $\gamma^{\uparrow\uparrow} < v$. 

It follows from~\eqref{eq:equivalences} that 
\begin{equation}
 \sum_{w\geqs v} \bigbrak{\indicator{\gamma^\uparrow}(w) - 
\indicator{\gamma^{\uparrow\uparrow}}(w)}
= \indicator{\gamma^{\uparrow\uparrow} < v \leqs \gamma^\uparrow}
= \indicator{\mathfrak{K}(\gamma) = \Gamma_0(v) \cap 
\mathfrak{K}(\sA(\gamma))}\;.
\end{equation}
Thus, $\hat\eta_{\geqs}(v)$ satisfies the equivalent of the 
decomposition~\eqref{eq:a_agamma}, with 
\begin{equation}
 \hat\eta_{\geqs,\gamma}(v) = (\rho+d) (\abs{\sV_0\cap\sV_\gamma}-1)
+ \sum_{e\in\sE_\gamma \cap \sE_0} \deg(e) 
+ \sum_{\bar\gamma\in\sFu} N(\bar\gamma) \indicator{\mathfrak{K}(\bar\gamma) 
= \Gamma_0(v) \cap\mathfrak{K}(\cA(\bar\gamma))}\;. 
\label{eq:agamma_u} 
\end{equation}
One then shows that the properties~\eqref{eq:full_empty_normal} of full, empty 
and normal subgraphs still hold in this case. The case of $\gamma$ being normal 
requires the presence of the last term in~\eqref{eq:agamma_u}, to which only 
$\hat\gamma$ contributes, together with the fact that 
$\deg(\hat\gamma)+N(\hat\gamma)>0$. The remainder of the proof is the same as 
for Lemma~\ref{lem_sumeta1}. 
\end{proof}

The analogue of Corollary~\ref{cor:Fs} is then proved in the same way as above, 
completing the proof of Proposition~\ref{prop:sum_n_Hepp}. 


\section{Asymptotics}
\label{sec:asymp} 

Fix a tree $\tau\in T^F_-$ with $p$ leaves and $q$ edges. 
It follows from the definition~\eqref{eq:def_ceps} of $c_\eps(\tau)$, 
Propositions~\ref{prop:ceps} and~\ref{prop:Etau_Feynman}, the 
decomposition~\eqref{eq:Hepp_decomposition} into Hepp sectors and 
Proposition~\ref{prop:sum_n_Hepp} that 
\begin{equation}
 c_\eps(\tau) = 
 (-2)^{-1-\frac{p}{2}} \sum_{P\in \cP^{(2)}_\tau}  
 \sum_T \sum_{\sFs\in\Fsafe{\Gamma(\tau,P)}(T)}
 \sI(\tau,P,T,\sFs)\;,
\end{equation} 
where 
\begin{equation} 
\sI(\tau,P,T,\sFs) = \sum_{\bn} \int_{D_{\bT}} 
(\sW^{K} \hat{\sR}_{[\sFs, \sFs \cup \sFu ]} 
\Gamma(\tau,P) ) (z) \6z
\end{equation}
satisfies 
\begin{equation}
 \bigabs{\sI(\tau,P,T,\sFs)} 
 \leqs C_0^{\abs{\sV(\Gamma(\tau,P))}}
 K_1^{\abs{\sE(\Gamma(\tau,P))}}
 \eps^{\deg \Gamma(\tau,P)} \bigbrak{\log(\eps^{-1})}^{\zeta(\Gamma(\tau,P))}
\end{equation} 
(with the convention that $\eps^{\deg\Gamma(\tau,P)}$ is to be replaced by 
$\log(\eps^{-1})$ if $\deg\Gamma(\tau,P)=0$). To obtain an upper bound on 
$\abs{c_\eps(\tau)}$, it thus remains to control the sums over Hepp trees $T$, 
permutations $P$, and safe forests $\sFs$. Summing over all $\tau\in T^F_-$ will 
then provide an upper bound on the renormalisation constants. 


\subsection{\Full\ binary trees}
\label{ssec:asymp_complete} 

Recall that a \full\ binary tree $\tau$ with $p$ leaves has $q=2p-2$ edges 
and $p-1$ inner vertices. It will be useful to parametrise the set of 
\full\ binary trees with an even number of leaves by integers $k$ such that  
$p=2k+2$ and $q=4k+2$. It follows from Proposition~\ref{prop:degE_Gamma2} that 
for any pairing $P$, the corresponding (reduced) Feynman diagram 
$\Gamma=\Gamma(\tau,P)$ will have $2k$ vertices, $3k$ edges, and degree
\begin{align}
 \deg\Gamma(\tau,P) &= (3k+1)\rho - (k+1)d \\
 &= -\frac23 d + (3k+1)(\rho-\rhocrit) 
 &&\forall P\in\cP^{(2)}\;.
\label{eq:degGamma_k1} 
\end{align} 
This degree is negative if and only if 
\begin{equation}
\label{eq:kmax} 
 k \leqs \kmax 
 = \frac{d-\rho}{3\rho-d} 
 = \frac{d-\rho}{3(\rho-\rhocrit)}\;.
\end{equation} 
We can thus rewrite~\eqref{eq:degGamma_k1} as 
\begin{equation}
 \label{eq:degGamma_k} 
 \deg\Gamma = -(d-\rho) \biggpar{1-\frac{k}{\kmax}} =: \alpha_k\;.
\end{equation} 
The number of possible pairings of the $2k+2$ leaves is equal to $(2k+1)!! = 
\prod_{i=0}^k(2i+1)$. The number of Hepp trees $T$ is bounded above by 
$(2k-1)!$, and is reached when $T$ is a comb tree, whose $2k$ leaves can be 
associated in $(2k-1)!$ inequivalent ways to the $2k$ vertices of $\Gamma$. 
The number of safe forests $\sFs$ can be bounded as follows. 

\begin{lemma}
There are at most $2^{\abs{\sG_\Gamma^-}}$ safe forests in $\Gamma$, where the 
number of divergent subdiagrams satisfies $\abs{\sG_\Gamma^-} \leqs k$. 
\end{lemma}
\begin{proof}
Let $N_m$ denote the number of edges of a Feynman diagram $\Gamma$ having 
$m$ divergent subdiagrams. Then $N_1\geqs2$, and $N_{m_1+m_2} \geqs N_{m_1} + 
N_{m_2} + 1$, since elements of a forest have to be strictly included into one 
another or vertex disjoint. By induction on $m$, one obtains $N_m\geqs 3m-1$, 
implying $3\abs{\sG_\Gamma^-}-1\leqs 3k$, and thus $\abs{\sG_\Gamma^-}\leqs k$. 
The bound on the number of safe forests then simply follows from the fact that 
a finite set with $n$ elements has $2^n$ subsets, and is reached when all 
forests are safe. 
\end{proof}

Finally, we need to control the number of terms yielding an exponent $\zeta=1$ 
rather than $\zeta=0$. We write $\cP^{(2)}_\tau = \cP^{(2)}_{\tau,0} \sqcup 
\cP^{(2)}_{\tau,1}$ ,where $\cP^{(2)}_{\tau,i}$ denotes the set of pairings 
yielding a diagram $\Gamma(\tau,P)$ with $\zeta(\Gamma)=i$. Then we have the 
following key estimate.

\begin{lemma}
\label{lem:pairings_zeta}
$\cP^{(2)}_{\tau,1}$ is non-empty only when $\kmax$ is an odd integer and 
$2k\geqs\kmax+1$. In that case, we have 
\begin{equation}
 \frac{\abs{\cP^{(2)}_{\tau,1}}}{\abs{\cP^{(2)}_{\tau,0}}} 
 \leqs r(k) := \frac{\kmax!!(2k-\kmax)!!}{(2k+1)!!}\;.
\label{eq:r(k)} 
\end{equation} 
Furthermore, we have 
\begin{equation}
\label{eq:r(k)_bound} 
 r(k) \leqs M 2^{-(2k-\kmax)}
\end{equation} 
for $\kmax+1 \leqs 2k \leqs 2\kmax$, where $M$ is a constant independent of 
$k$, $\rho$ and $\eps$.  
\end{lemma}
\begin{proof}
Assume $\Gamma=\Gamma(\tau,P)$ has a child $\gamma$ having degree $0$. By 
Lemma~\ref{lem:subdiagrams},  $\gamma = \Gamma(\bar\tau,\bar P)$ where 
$\bar\tau$ is an \afull\ binary subtree of $\tau$, and $\bar P$ is the 
restriction of $P$ to the leaves of $\bar\tau$. Let $\bar k < k$ be such that 
$\bar\tau$ has $2\bar k+2$ leaves and $4\bar k+3$ edges. Then we have 
\begin{equation}
 \deg\gamma = -\frac13 d + (3\bar k+2)(\rho-\rhocrit) = 0\;.
\end{equation} 
In view of~\eqref{eq:degGamma_k1}, this implies 
\begin{equation}
 2(3\bar k+2)(\rho-\rhocrit) = \frac23 d 
 = (3k+1)(\rho-\rhocrit) - \deg\Gamma\;,
\end{equation} 
which yields $2\bar k+1 = \kmax$ by~\eqref{eq:kmax} and~\eqref{eq:degGamma_k}. 
Thus $\kmax$ must be an odd integer, and the condition $k > \bar k$ yields 
$2k\geqs\kmax+1$. Finally, the number of pairings that do not mix leaves of 
$\bar\tau$ with those of $\tau\setminus\bar\tau$ is given by $(2\bar 
k+1)!!(2k-2\bar k-1)!! = \kmax!!(2k-\kmax)!!$, which proves~\eqref{eq:r(k)}. To 
prove~\eqref{eq:r(k)_bound}, we write $k=x\kmax$ and use Stirling's formula to 
obtain 
\begin{align}
 \log r(k) &= \frac{\kmax}{2} \bigbrak{(2x-1)\log(2x-1) - 2x\log(2x)} 
 - \frac12\log(x) + \Order{1} \\
 &\leqs -\log(2) (2x\kmax-\kmax) + \Order{1}\;,
\end{align} 
where we have used a convexity argument to obtain the last line. 
\end{proof}

\begin{remark}
\label{rem:kmax} 
In what follows, we will always assume that $\kmax > 1$. Indeed, for $\kmax < 
1$ (that is, $\rho>\frac d2$), the only potentially divergent tree is $\RSV$, 
while for $\kmax=1$, the trees with $4$ leaves considered in 
Example~\ref{ex:four_leaves} have degree $0$. These cases can be treated \lq\lq 
by hand\rq\rq, in particular the expectation~\eqref{eq:E_comb_fourleaves} 
can be shown to diverge like $\log(\eps^{-1})$. 
\end{remark}

The above combinatorial considerations show that 
\begin{equation}
 \abs{c_\eps(\tau)} \leqs (2k+1)!! (2k-1)! 
 C_0^{2k}
 K_1^{3k} \eps^{\deg\Gamma}
 \bigbrak{1+r(k)\log(\eps^{-1})}\;,
\end{equation} 
where $\deg\Gamma$ is given by~\eqref{eq:degGamma_k}, and $\eps^{\deg\Gamma}$ 
has to be replaced by $\log(\eps^{-1})$ if $\deg\Gamma=0$. 
Let us write $C_0^{2k}K_1^{3k} = K_2^{3k}$, where $K_2 = C_0^{2/3} K_1$.
It follows from 
Stirling's formula that 
\begin{equation}
 \log\bigpar{(2k+1)!! (2k-1)! K_1^{3k} \eps^{\deg\Gamma}} 
 \leqs 3k\log k + 3k \bigbrak{\log2 -1 + \log K_2} 
 - \log(\eps^{-1}) \deg\Gamma + \Order{1}\;,
\end{equation} 
where the remainder term $\Order{1}$ is independent of $k$, $\rho$ and $\eps$. 

When computing the contribution of \full\ binary trees to the renormalisation 
counterterm~\eqref{eq:counterterm}, we have to take into account the number of 
these trees, as well as the combinatorial factor 
$2^{\ninner(\tau)-\nsym(\tau)}$. The latter can be bounded above by $2^{2k+1}$, 
while the former is given by the $(2k+2)$nd Wedderburn--Etherington number 
(sequence \texttt{A001190} in the On-Line Encyclopedia of Integer Sequences 
OEIS), cf.~\cite[Section~4.4.1]{BK17}. These numbers are known to grow like 
$k^{-3/2}\beta_2^{-2k}$, where $\beta_2 \simeq 0.4026975$ (OEIS sequence 
\texttt{A240943}) is the radius of convergence of the generating series of the 
sequence. 

The renormalisation counterterm $C_0(\eps,\rho)$ due to \full\ binary trees 
can thus be written in the form 
\begin{align}
\label{eq:C0} 
 C_0(\eps,\rho) ={}& \sum_{k=0}^{\intpart{\kmax} - 1} A_k \eps^{\alpha_k}
 + \indicator{\kmax\in\N} A_{\kmax} \log(\eps^{-1})
 \\ 
 &{}+ \indicator{\kmax\in2\N+1} \Biggbrak{\,
 \sum_{k=(\kmax-1)/2}^{\kmax - 1} A_k 
\eps^{\alpha_k} r(k) + A_{\kmax}r(\kmax)\log(\eps^{-1})}\log(\eps^{-1})\;,
\end{align} 
where $\alpha_k$ is defined in~\eqref{eq:degGamma_k} and 
\begin{align}
\log\abs{A_k} &\leqs 3 \bigbrak{k\log k + ak - \frac12\log(k+1)} + 
\Order{1}\;, \\
a &= \log2 - 1 + \log K_2 + \frac23 \log(\beta_2^{-1})\;.
\end{align}
As a consequence, we have the bound 
\begin{equation}
\label{eq:bound_Ak} 
 \bigabs{A_k \eps^{\alpha_k}} \leqs M \e^{3F(k)}\;,
\end{equation} 
where $M$ is a constant independent of $k$, $\rho$ and $\eps$, and 
\begin{equation}
\label{eq:F(k)} 
 F(k) = k \log k + (a-b_\eps) k + b_\eps \kmax - \frac12\log(k+1)\;,
\end{equation} 
with
\begin{equation}
 b_\eps = (\rho - \rhocrit) \log(\eps^{-1})\;.
\end{equation} 
Note in particular that $\e^{3F(0)}=\eps^{-(d-\rho)}$, and that 
$F$ is strictly convex. 

\begin{prop}
\label{prop:asymptotics} 
Define the threshold 
\begin{equation}
 \epscrit(\rho) 
= \exp\biggset{-\frac{1}{\rho-\rhocrit} \biggbrak{\log\kmax + a - 
\frac{\log(\kmax+1)}{2\kmax}}}\;.
\end{equation} 
Then there exist constants $M_1, M_2$, independent of $\eps$ and $\rho$, such 
that the counterterm $C_0(\eps,\rho)$ satisfies 
\begin{align}
 C_0(\eps,\rho)
 &= A_0 \eps^{-(d-\rho)}
 \bigbrak{1 + R_1(\eps,\rho)} 
 && \text{for $\eps < \epscrit(\rho)$}\;, \\
 \abs{C_0(\eps,\rho)}
 &\leqs M \epscrit(\rho)^{-(d-\rho)}
 \bigbrak{\log(\eps^{-1}) + R_2(\eps,\rho)} 
 && \text{for $\eps \geqs \epscrit(\rho)$}\;,
\end{align} 
where the remainders satisfy 
\begin{equation}
 \bigabs{R_1(\eps,\rho)} \leqs \frac{M_1}{\rho-\rhocrit} 
\biggpar{\frac{\eps}{\epscrit(\rho)}}^{3(\rho-\rhocrit)}\;, 
\qquad 
 \bigabs{R_2(\eps,\rho)} \leqs \frac{M_2}{\rho-\rhocrit}  
\biggpar{\frac{\epscrit(\rho)}{\eps}}^{3(\rho-\rhocrit)}
\;.
\end{equation} 
\end{prop}
\begin{proof}
Since $F$ is convex, we have $F(k)\leqs F(0)+Hk$ for all $k\in[0,\kmax]$, where 
\begin{equation}
\label{eq:def_H} 
 H = \frac{F(\kmax)-F(0)}{\kmax}
 = \log\kmax - b_\eps + a - \frac12 \frac{\log(\kmax+1)}{\kmax}\;.
\end{equation} 
Note that $\epscrit(\rho)$ has been defined in such a way that 
\begin{equation}
\label{eq:H_value} 
 H = (\rho-\rhocrit) \log\biggpar{\frac{\eps}{\epscrit(\rho)}}\;,
\end{equation} 
and that $\e^{3F(\kmax)} = \epscrit(\rho)^{-(d-\rho)}$. We will repeatedly use 
the fact that if $\beta\in\R$ and $N>k_0$ are positive integers, then 
\begin{equation}
 \sum_{k=k_0}^{N-1} \e^{\beta k} \leqs 
 \begin{cases}
  \bigbrak{(N-k_0)\wedge\beta^{-1}} \e^{\beta(N-1)}
  & \text{if $\beta>0$\;,}\\
  N-k_0 
  & \text{if $\beta=0$\;,}\\
  \bigbrak{(N-k_0)\wedge|\beta|^{-1}} \e^{\beta k_0}
  & \text{if $\beta<0$\;.}
 \end{cases}
 \label{eq:geom} 
\end{equation} 
In the case $\eps \geqs \epscrit(\rho)$, i.e.\ $H\geqs0$, we 
rewrite~\eqref{eq:C0} as 
\begin{align}
 C_0(\eps,\rho) - \indicator{\kmax\in\N} A_{\kmax}\log(\eps^{-1}) 
 ={}& 
 \sum_{k=0}^{\intpart{\kmax} - 1} A_k \eps^{\alpha_k} + 
\sum_{k=(\kmax+1)/2}^{\kmax - 1} A_k \eps^{\alpha_k} r(k) 
\log(\eps^{-1})\\
 &{}+ A_{\kmax}r(\kmax) \bigbrak{\log(\eps^{-1})}^2 \\
 =:{}& M \e^{3F(\kmax)}
 \bigbrak{\cR_1(\eps,\rho) + \cR_2(\eps,\rho) + 
\cR_3(\eps,\rho)}\;, 
\end{align} 
where the terms $\cR_2$ and $\cR_3$ vanish unless $\kmax$ is an odd integer, 
which we can assume to be at least $3$ by Remark~\ref{rem:kmax}. 
By~\eqref{eq:bound_Ak}, the leading term $A_{\kmax}\log(\eps^{-1})$ has 
indeed order $\epscrit(\rho)^{-(d-\rho)}\log(\eps^{-1})$. To bound $\cR_1$, we 
use~\eqref{eq:geom} with $N=\intpart{\kmax}$, $k_0=0$ and $\beta=3H$ to get 
\begin{equation}
 \abs{\cR_1(\eps,\rho)} \lesssim \kmax \e^{-3H} 
 \lesssim \frac{1}{\rho-\rhocrit}
 \biggpar{\frac{\epscrit(\rho)}{\eps}}^{3(\rho-\rhocrit)}\;.
\end{equation} 
Regarding $\cR_2$, we use~\eqref{eq:r(k)_bound} to get 
\begin{equation}
 \abs{\cR_2(\eps,\rho)} 
 \leqs M \e^{-3H\kmax} \e^{\kmax\log2} 
 \sum_{k=(\kmax+1)/2}^{\kmax - 1} \e^{(3H-2\log2)k} \log(\eps^{-1})\;.
\end{equation} 
We use~\eqref{eq:geom} differently in several regimes. 
If $3H\leqs 1$, we obtain 
\begin{equation}
 \abs{\cR_2(\eps,\rho)} \lesssim \e^{-\frac32 H(\kmax-1)}  
\log(\eps^{-1})
\leqs \e^{-3H} \log(\eps^{-1})\;.
\end{equation} 
If $1<3H<2\log2$, we get 
\begin{equation}
 \abs{\cR_2(\eps,\rho)} \lesssim \kmax \e^{-\frac32 H(\kmax-1)}  
\log(\eps^{-1})
\leqs \e^{-3H} \kmax\e^{-\frac32 (\kmax-3)}  
\log(\eps^{-1})\;,
\end{equation} 
which yields a bound of the same form, since $\kmax\e^{-\frac32 (\kmax-3)}$ is 
bounded uniformly in $\kmax\geqs3$. If $3H\geqs2\log2$, we have 
\begin{equation}
 \abs{\cR_2(\eps,\rho)} \lesssim \e^{-3H}\kmax 2^{-\kmax} \log(\eps^{-1})\;. 
\end{equation} 
Note that in all three regimes, we have $\abs{\cR_2(\eps,\rho)} \lesssim 
\e^{-3H}\log(\eps^{-1})$. Regarding $\cR_3$, 
we observe that it is bounded by $r(\kmax)\log(\eps^{-1})^2$, showing that 
\begin{equation}
 \e^{3H} \abs{\cR_3(\eps,\rho)} \lesssim 
 \frac{2^{-\kmax}}{\epscrit(\rho)^{3(\rho-\rhocrit)}} 
 \eps^{3(\rho-\rhocrit)} \bigbrak{\log(\eps^{-1})}^2\;.
\end{equation} 
For fixed $\rho$, the right-hand side is maximal for $\eps=\eps_* = 
\e^{-2/(3(\rho-\rhocrit)}$. Therefore, by definition of $\kmax$ and 
$\epscrit(\rho)$, we have 
\begin{equation}
 \e^{3H} \abs{\cR_3(\eps,\rho)} \lesssim 
 \frac{2^{-\kmax}}{\epscrit(\rho)^{3(\rho-\rhocrit)}} 
 \frac{4\e^{-2}}{9(\rho-\rhocrit)^2}
 \lesssim 2^{-\kmax} \kmax^3 \kmax^2\;,
\end{equation} 
which is bounded uniformly in $\kmax\geqs1$. Therefore, we have  
$\abs{\cR_3(\eps,\rho)} \lesssim \e^{-3H}$, completing the proof for 
$\eps>\epscrit(\rho)$. 

It remains to consider the case $\eps < \epscrit(\rho)$, i.e.\ $H<0$. Here we 
decompose 
\begin{align}
 C_0(\eps,\rho) - A_0 \eps^{\alpha_0} 
 ={}& 
 \sum_{k=1}^{\intpart{\kmax} - 1} A_k \eps^{\alpha_k} 
 + A_{\kmax}\log(\eps^{-1})\\
 &{}+ \sum_{k=(\kmax+1)/2}^{\kmax - 1} A_k \eps^{\alpha_k} r(k) 
\log(\eps^{-1}) 
 + A_{\kmax}r(\kmax) \bigbrak{\log(\eps^{-1})}^2 \\
 =:{}& M \e^{3F(0)}
 \bigbrak{\cR_1(\eps,\rho) + \cR_2(\eps,\rho) + 
\cR_3(\eps,\rho) + \cR_4(\eps,\rho)}\;, 
\end{align} 
where $\cR_3$ and $\cR_4$ vanish unless $\kmax$ is an odd integer. 
Applying~\eqref{eq:geom} with $N=\intpart{\kmax}$, $k_0=1$ and $\beta=3H$, we 
obtain 
\begin{equation}
 \abs{\cR_1(\eps,\rho)} \lesssim \kmax\e^{3H} 
 \lesssim \frac{1}{\rho-\rhocrit}
\biggpar{\frac{\eps}{\epscrit(\rho)}}^{3(\rho-\rhocrit)}\;.
\end{equation}
Using~\eqref{eq:bound_Ak}, we find 
\begin{equation}
 \abs{\cR_2(\eps,\rho)} 
 \leqs \e^{3H\kmax}\log(\eps^{-1})
 = \biggpar{\frac{\eps}{\epscrit(\rho)}}^{d-\rho}\log(\eps^{-1})\;.
\end{equation} 
Regarding $\cR_3$, using again~\eqref{eq:geom} we get 
\begin{equation}
 \abs{\cR_3(\eps,\rho)} 
 \lesssim \e^{\frac32 H(\kmax+1)} \log(\eps^{-1})
 = 
\biggpar{\frac{\eps}{\epscrit(\rho)}}^{\frac32(\kmax+1)(\rho-\rhocrit))}
\log(\eps^{-1})\;.
\end{equation} 
Finally, by~\eqref{eq:r(k)_bound} we also have 
\begin{equation}
 \abs{\cR_4(\eps,\rho)}
 \leqs \biggpar{\frac{\eps}{\epscrit(\rho)}}^{d-\rho}
 2^{-\kmax}\bigbrak{\log(\eps^{-1})}^2\;.
\end{equation} 
Since $\kmax > 1$, we have $d-\rho > 3(\rho-\rhocrit)$, so that $\cR_2$ and 
$\cR_4$ are negligible with respect to $\cR_1$. In addition, $\cR_3$ only 
occurs when $\kmax\geqs3$, and then it is also dominated by $\cR_1$. 
\end{proof}


\subsection{\Afull\ binary trees}
\label{ssec:asymp_incomplete} 

It remains to consider the case of \afull\ binary trees without decorations 
$X_i$, as the contribution of \afull\ binary trees with a decoration $X_i$ 
vanishes by symmetry. 

\Afull\ trees with an even number of leaves can be parametrised by an 
integer $k$ such that these trees have $p=2k+2$ leaves and $q=4k+3$ edges. The 
corresponding reduced Feynman diagrams have $2k+1$ vertices, $3k+1$ edges, and 
degree 
\begin{equation}
\label{eq:degGamma_k2} 
 \deg\Gamma(\tau,P) = (3k+2)\rho - (k+1)d
 \qquad \forall P\in\cP^{(2)}\;.
\end{equation} 
The main difference with the case of \full\ trees is that the maximal value 
of $k$ for $\deg\Gamma(\tau,P)$ to be negative is now 
\begin{equation}
\label{eq:kbarmax} 
 \kbarmax 
 = \frac{d-2\rho}{3\rho-d} 
 = \frac{d-2\rho}{3(\rho-\rhocrit)}\;,
\end{equation} 
which is smaller than $\kmax$ by a factor approaching $2$ as 
$\rho\searrow\rhocrit$. Furthermore, Lemma~\ref{lem:pairings_zeta} shows that 
$\zeta$ always vanishes in this case. The remaining combinatorial arguments 
remain unchanged, with the result that 
\begin{equation}
 C_1(\eps,\rho) = \sum_{k=0}^{\kbarmax - 1} \bar A_k \eps^{\bar\alpha_k} 
 + \indicator{\kbarmax\in\N}\bar A_{\kbarmax} \log(\eps^{-1})\;,
\end{equation} 
where
\begin{equation}
 \bar\alpha_k = -(d-2\rho) \biggpar{1-\frac{k}{\kbarmax}}
\end{equation} 
and $\abs{\bar A_k \eps^{\bar\alpha_k}} \leqs M\e^{3 \bar F(k)}$ with 
\begin{equation}
 \bar F(k) = k \log k + (a-b_\eps) k + b_\eps \kbarmax - \frac12\log(k+1)\;.
\end{equation} 
It thus suffices to modify the threshold value of $\eps$ to obtain the result. 


\subsection{A remark on lower bounds}
\label{ssec:asymp_remark} 

The results we have obtained provide upper bounds on the counterterms. 
Obtaining matching lower bounds seems out of reach at this stage, because, as 
we have seen, the behaviour in $\eps^{\deg\Gamma(\tau,P)}$ of the terms 
$c_\eps(\tau)$ is a consequence of cancellations of more singular terms in 
Zimmermann's forest formula. 

However, we can at least argue that among the many terms contributing to the 
counterterms $C_0(\eps,\rho)$ and $C_1(\eps,\rho)$, there exist terms which are 
bounded above \emph{and} below by a quantity of the same order. This does not 
of course exclude that cancellations among these terms exist, which ultimately 
make the counterterms much smaller. However, such a scenario seems unlikely, 
unless some hidden symmetries have been overlooked. 

Indeed, assume that $\tau$ is a regular binary tree 
(cf.~\eqref{eq:regular_tree}). Then $\ninner(\tau) = \nsym(\tau)$, so that the 
contribution of $\tau$ to $C_0(\eps,\rho)$ is given by 
\begin{equation}
 c_\eps(\tau) = E(\tilde\cA^E_-\tau) 
 = \sum_{P\in\cP_\tau^{(2)}} E(\tilde\cA_-\Gamma(\tau,P))\;.
\end{equation} 
It follows from Lemma~\ref{lem:subdiagrams} that $\Gamma(\tau,P)$ cannot have 
any divergent strict subdiagram, since a regular binary tree does not contain 
any \afull\ binary subtree. Therefore, \eqref{eq:Hepp_decomposition} reduces 
to 
\begin{equation} 
E(\tilde \cA_- \Gamma(\tau,P)) = - \sum_{T} \sum_{\bn} \int_{D_{\bT}} 
(\sW^{K} \Gamma(\tau,P) ) (z) \6z\;.
\end{equation}
It is known that whenever $\rho < d$, the fractional heat kernel $P_\rho$ is 
given by the Riesz kernel which has a constant sign. This sign is not conserved 
by the decomposition~\eqref{eq:PKR} of the kernel into its singular and smooth 
parts, but one can add a bounded, compactly supported term to $K_\rho$ in such 
a way that $K_\rho$ has a constant sign, and without changing the divergent 
part of the integrals. Therefore, we obtain 
\begin{equation}
 \abs{E(\tilde \cA_- \Gamma(\tau,P))} \geqs a \sum_{T} K_2^{\abs{\sE}} 
\eps^{\deg(\Gamma)}
\end{equation} 
for a constant $a>0$. Since the number of pairings $P$ and of Hepp trees $T$ 
have the same factorial behaviour as above, we indeed obtain for $c_\eps(\tau)$ 
an asymptotic behaviour in $\epscrit(\rho)^{-(d-\rho)}\log(\eps^{-1})$.  


\subsection{Extension to other parameter regimes}
\label{ssec:other_parameters} 

In this section, we extend the results to the more general family of equations  
\begin{equation} 
\label{eq:SPDE_parameters}
\partial_t u - \gamma\Delta^{\rho/2} u = 
gu^2 + \sigma\xi\;,
\end{equation}
where $\gamma$, $g$ and $\sigma$ are parameters measuring the strength of 
each component of the equation, that is, the smoothing effect given by the 
fractional Laplacian, the nonlinearity $u^2$ and the noise $\xi$. Our aim is to 
understand how the model behaves when one lets these parameters vary, and 
to determine some potentially interesting parameter regimes.

The extension can actually be done in two equivalent ways: using directly the 
BPHZ renormalisation on the parameter-dependent 
equation~\eqref{eq:SPDE_parameters}, or using a scaling argument for the 
original equation~\eqref{eq:SPDE}. We will briefly outline both arguments, 
which will serve as a reality check of the results. Using the BPHZ 
renormalisation~\eqref{eq:counterterm} on~\eqref{eq:SPDE_parameters}, one 
obtains the renormalised equation 
\begin{equation}
\label{eq:SPDE_parameters_renormalised} 
\partial_t u - \gamma\Delta^{\rho/2} u = 
gu^2 + C^{\gamma,g,\sigma}(\eps, \rho, u) + 
\sigma\xi^{\eps}\;,
\end{equation}
where the new counterterm is given by
\begin{equation}
C^{\gamma,g,\sigma}(\eps, \rho, u)
= C_0^{\gamma,g,\sigma}(\eps, \rho) 
+ C_1^{\gamma,g,\sigma}(\eps, \rho) u 
\end{equation}
with
\begin{align}
C_0^{\gamma,g,\sigma}(\eps, \rho) &= \sum_{\tau \text{ full}} 
c^{\gamma,\sigma}_{\eps}(\tau) g^{\ninner(\tau)}2^{\bar n(\tau)}\;, \\
C_1^{\gamma,g,\sigma}(\eps, \rho) &= \sum_{\tau \text{ almost full}} 
c^{\gamma,\sigma}_{\eps}(\tau) g^{\ninner(\tau)}2^{\bar n(\tau)}\;,
\end{align}
and $\bar n(\tau) = \ninner(\tau) - \nsym(\tau)$. The new 
renormalisation constant $c^{\gamma,\sigma}_{\eps}$ associated to a 
tree $\tau$ is given by
\begin{equation}
  c^{\gamma,\sigma}_{\eps}(\tau)  = \sigma^{\nleaves(\tau)} 
c_{\eps}^\gamma(\tau)\;,
\end{equation}
where $\nleaves(\tau)$ is the number of leaves of $\tau$. Here in the notation 
$c_{\eps}^\gamma(\tau)$, we stress that the value of the renormalisation 
constant depends on $\gamma$ via the scaled Green function 
$\smash{K_{\rho}^{(\gamma)}} = (\partial_t - \gamma \Delta^{\rho/2})^{-1}$ of 
the fractional Laplacian that now appears on every edge of the tree $\tau$. One 
key property that we use in the sequel is 
\begin{equation}
K_{\rho}^{(\gamma)}(t,x) = K_{\rho}(\gamma t,x).
\end{equation}
Indeed, since $f(t,x) = (K_\rho*h)(t,x)$ satisfies the equation 
$\partial_t f - \Delta^{\rho/2} f = h$, setting $\bar h(t,x) = \gamma h(\gamma 
t,x)$ it is easy to check that 
\begin{equation}
 \bar f(t,x) 
 = f(\gamma t,x) 
 = \int_\R\int_{\R^d} K_\rho(\gamma(t-s),x-y) \bar h(s,y) \6y\6s
\end{equation} 
satisfies $\partial_t \bar f - \gamma\Delta^{\rho/2} \bar f = \bar h$.
Then, performing a linear change of variable for all 
time integrals in the definition~\eqref{eq:EGamma} of $E(\Gamma)$, we get
\begin{equation}
c_{\eps}^\gamma(\tau) = \frac{1}{\gamma^{\nu_\tau - 1}} c_{\eps}(\tau)
\end{equation}
where $\nu_{\tau}$ is the number of nodes of the Feynman diagram 
$\Gamma(\tau,P)$ associated with $\tau$.
\begin{itemize}
\item 	If $\tau$ is a full tree with $p$ leaves and $q = 2p-2$ edges, 
then we have
\begin{equation}
\nu_{\tau} = q + 1 - \frac{p}{2} = \frac{3}{2} p -1\;, 
\qquad \ninner(\tau) = p-1\;.
\end{equation}
Therefore, we obtain 
\begin{equation} 
\label{pre_factor_full}
C_0^{\gamma,g,\sigma}(\eps, \rho)
= \sum_p \frac{g^{p-1}\sigma^p}{\gamma^{\frac{3}{2}p-2}} 
\sum_{\text{$\tau$ full with $p$ leaves} } 
c_{\eps}(\tau)
2^{\bar n(\tau)}\;.
\end{equation}
\item 	If $\tau$ is an almost full tree with $p$ leaves and $q = 2p-1$ edges, 
then we have 
\begin{equation}
\nu_{\tau} = \frac{3}{2} p\;, 
\qquad n_{\text{inner}}(\tau) = p\;,
\end{equation}
yielding 
\begin{equation} 
\label{pre_factor_almost_full}
C_1^{\gamma,g,\sigma}(\eps, \rho) 
= \sum_p \frac{g^p\sigma^p}{\gamma^{\frac{3}{2}p-1}} 
\sum_{\text{$\tau$ almost full with $p$ leaves}} 
c_{\eps}(\tau)
2^{\bar n(\tau)}\;.
\end{equation}
\end{itemize}

The second argument allowing to obtain~\eqref{pre_factor_full} 
and~\eqref{pre_factor_almost_full} is based on scaling. If $u$ satisfies the 
renormalised equation
\begin{equation}
\partial_t u - \Delta^{\rho/2} u = 
u^2 + C(\eps,\rho,u) +   \xi^{\eps}\;,
\end{equation}
then for any $\lambda>0$ and $\alpha,\beta\in\R$, 
$\bar u(t,x) = \lambda^{\alpha} u(\lambda^{\beta} t, \lambda x)$
solves the equation
\begin{equation} 
\label{equa_baru}
\partial_t \bar u - \lambda^{\beta- \rho} \Delta^{\rho / 2} \bar u  
  = \lambda^{\beta-\alpha} \bar u^2 + \lambda^{\alpha + \beta} 
C(\eps,\rho,\lambda^{-\alpha}\bar u) +  \lambda^{\alpha + \beta} 
\xi^{\eps}_{\lambda^\beta,\lambda}\;,
\end{equation} 
where 
\begin{align}
 \xi^{\eps}_{\lambda^\beta,\lambda}(t,x) 
 &= \xi^\eps(\lambda^\beta t,\lambda x) \\
 &= \frac{1}{\eps^{\rho+d}} 
 \int_{\R^{d+1}} \varrho\biggpar{\frac{\lambda^\beta t-t'}{\eps^{\rho}}, 
 \frac{\lambda x-x'}{\eps}} \xi(\6t',\6x') \\
 &= \frac{1}{\bar\eps^{\rho+d}} 
 \int_{\R^{d+1}} \lambda^{\beta-\rho} 
\varrho\biggpar{\frac{\lambda^{\beta-\rho}(t-t'')}{\bar\eps^{\rho}}, 
 \frac{x-x''}{\bar\eps}} \xi(\lambda^\beta\6t'',\lambda\6x'')\;.
\end{align}
In the last line, we have set $\eps = \lambda\bar\eps$, and made the change of 
variables $t'=\lambda^\beta t''$, $x'=\lambda x''$. By the scaling property of 
space-time white noise, this is equal in law to 
\begin{equation}
 \frac{\lambda^{-\frac{\beta}{2}-\frac{d}{2}}}{\bar\eps^{\rho+d}} 
 \int_{\R^{d+1}} \lambda^{\beta-\rho}
\varrho\biggpar{\frac{\lambda^{\beta-\rho}(t-t'')}{\bar\eps^{\rho}}, 
 \frac{x-x''}{\bar\eps}} \xi(\6t'',\6x'')
 = \lambda^{-\frac{\beta}{2}-\frac{d}{2}} 
 \tilde\xi^{\bar\eps}(t,x)\;,
\end{equation} 
where $\tilde\xi^{\bar\eps} = \tilde\varrho^{\bar\eps} * \xi$ is defined like 
$\xi^\eps$, but using $\tilde\varrho(t,x) = 
\lambda^{\beta-\rho}\varrho(\lambda^{\beta-\rho}t,x)$ as new mollifier.
Thus~\eqref{equa_baru} is indeed of the 
form~\eqref{eq:SPDE_parameters_renormalised} with parameters 
\begin{equation} 
\label{eq:values_sigma_g_lambda}
\gamma = \lambda^{\beta- \rho}, \qquad g = \lambda^{\beta-\alpha},  \qquad 
\sigma = \lambda^{\alpha + \frac{\beta}{2} - \frac{d}{2}}\;. 
\end{equation}
Using the expression~\eqref{eq:counterterm} of $C(\eps,\rho,u)$, we find
\begin{equation}
\lambda^{\alpha + \beta} C(\lambda \bar \eps,\rho,\lambda^{-\alpha}\bar u) 
= \lambda^{\alpha+ \beta}\sum_{\tau \text{ full}} c_{\lambda \bar \eps}(\tau) 
2^{\bar n(\tau)} + \lambda^{\beta} \sum_{\tau \text{ almost full}} c_{\lambda 
\bar \eps}(\tau) 2^{\bar n(\tau)} u
\end{equation}
If $ \deg(\tau) < 0 $, then 
\begin{equation}
c_{\eps}(\tau) \sim \eps^{\deg(\tau)}\;, 
\qquad 
c_{\lambda \bar \eps}(\tau) \sim 
\bar \eps^{\deg(\tau)} \lambda^{\deg(\tau)}\;, 
\end{equation}
which implies that $c_{\lambda \bar \eps} = \lambda^{\deg(\tau)} c_{\bar 
\eps}(\tau)$ (note that since $\log(\lambda\bar\eps) = \log\lambda + 
\log\bar\eps$, logarithmic divergences do not change the leading order of 
$c_\eps(\tau)$).
In the case of $\tau$ being a full tree with $p = 2k+2$ leaves, by 
\eqref{eq:degGamma_k1} we have
\begin{equation}
\deg(\tau)  = (3k +1)\rho - (k+1)d  
= \frac{3}{2} p (\rho - \rho_c) - 2 \rho
\end{equation}
where we have used the fact that $\rho_c = \frac{d}{3}$. For almost full trees, 
we obtain 
\begin{equation}
\deg(\tau)  = (3k +2)\rho - (k+1)d  
= \frac{3}{2} p (\rho - \rho_c) - \rho\;.
\end{equation}
It follows that 
\eqref{equa_baru} becomes
\begin{equation}
\partial_t \bar u - \gamma \Delta^{\rho / 2} \bar u  
  = g \bar u^2 +  \bar C_0(\bar \eps,\rho) + \bar C_1(\bar \eps,\rho) \bar u + \sigma
\xi^{\bar \eps}
\end{equation}
with 
\begin{align}
 \bar C_0(\bar \eps, \rho) 
 &= \lambda^{\alpha + \beta} \sum_{p}
 \lambda^{\frac{3}{2}p(\rho-\rho_c) - 2 \rho } 
 \sum_{\text{$\tau$ full with $p$ leaves}} 
 c_{\bar \eps}(\tau) 2^{\bar n (\tau)}\;, \\
 \bar C_1(\bar \eps, \rho) 
 &= \lambda^{\beta} \sum_{p} \lambda^{\frac{3}{2}p(\rho-\rho_c) - \rho } 
 \sum_{\text{$\tau$ almost full with $p$ leaves}} 
c_{\bar \eps}(\tau) 2^{\bar n 
(\tau)}\;.
\end{align}
By~\eqref{eq:values_sigma_g_lambda}, these expressions for the counterterms are 
indeed equivalent to~\eqref{pre_factor_full} 
and~\eqref{pre_factor_almost_full}. 

It remains to prove the relations~\eqref{eq:C_full_parameters} 
and~\eqref{eq:C_almost_full_parameters}. For full binary trees, setting as 
before $p=2k+2$, we obtain from~\eqref{pre_factor_full} that 
\begin{equation}
 C_0^{\gamma,g,\sigma}(\eps, \rho)
= \sum_{k=0}^{\intpart{\kmax}-1} \frac{g^{2k+1}\sigma^{2k+2}}{\gamma^{3k+1}} 
f(k)
= \frac{g\sigma^2}{\gamma}\sum_{k=0}^{\intpart{\kmax}-1} \delta^k f(k)\;,
\end{equation}
where $f(k)$ is as on the right-hand side of~\eqref{eq:C0} and $\delta = 
g^2\sigma^2\gamma^{-3}$. We can now proceed as in the proof of 
Proposition~\ref{prop:asymptotics}, replacing $F(k)$ defined in~\eqref{eq:F(k)} 
by $\hat F(k) = F(k) + \frac13k\log\delta$, and $H$ in~\eqref{eq:def_H} by $\hat 
H = H + \frac13\log\delta$. Note that $\hat F$ is still convex.

For $\eps>\epscrit$, the sum is dominated by $k = \kmax$, and yields the 
same bound as for $C_0^{1,1,1}(\eps, \rho) = C_0(\eps,\rho)$, up to an 
additional factor 
\begin{equation}
 \frac{g\sigma^2}{\gamma} \delta^{\kmax} 
 = \biggpar{\frac{g^2\sigma^2}{\gamma^3}}^{\kmax}
 \frac{g\sigma^2}{\gamma}\;.
\end{equation} 
For $\eps<\epscrit$, the sum is dominated by the term $k=0$, which has the same 
scaling behaviour as $C_0(\eps,\rho)$, up to an additional factor 
$g\sigma^2\gamma^{-1}$. An analogous argument applies to the 
expression~\eqref{eq:C_almost_full_parameters} for $C_1^{\gamma,g,\sigma}(\eps, 
\rho)$.


\tableofcontents

\small
\bibliography{BBK}
\bibliographystyle{abbrv}               

\bigskip\bigskip\noindent
{\small 
Nils Berglund \\ 
Institut Denis Poisson (IDP) \\ 
Universit\'e d'Orl\'eans, Universit\'e de Tours, CNRS -- UMR 7013 \\
B\^atiment de Math\'ematiques, B.P. 6759\\
45067~Orl\'eans Cedex 2, France \\
{\it E-mail address: }{\tt nils.berglund@univ-orleans.fr}

\bigskip\noindent
{\small 
Yvain Bruned \\
Institut Elie Cartan de Lorraine (IECL)\\
Université de Lorraine, CNRS--UMR 7502 \\
Faculté des Sciences et Technologies \\
Campus, Boulevard des Aiguillettes \\
54506 Vandœuvre-lès-Nancy \\
{\it E-mail address: }{\tt yvain.bruned@univ-lorraine.fr}
}

\end{document}